\documentclass[11pt]{amsart}
\usepackage[margin=0.9in]{geometry}

\def\uvuciautor{\hspace{-0.4cm}}

\usepackage[dvipsnames]{xcolor}
\usepackage[centertableaux]{ytableau}
\usepackage{graphicx}

\usepackage{wrapfig}
\usepackage{eso-pic}
\usepackage{mathtools}
\usepackage{float} 
\usepackage{amsmath, amsthm, amsfonts, amssymb, amscd, bm}
\usepackage{mathrsfs}
\usepackage{multicol}
\usepackage{xcolor}
\definecolor{OxBlue}{HTML}{002147}

\usepackage[colorlinks]{hyperref}
\hypersetup{colorlinks = true, citecolor={OxBlue}, urlcolor  ={OxBlue}, linkcolor={OxBlue} }

\usepackage{perpage}
\MakePerPage{footnote}

\usepackage[normalem]{ulem}

\usepackage{enumerate}
\usepackage[capitalise, noabbrev]{cleveref}

\usepackage{tikz-cd} 
\usepackage{subfigure}

\usepackage[all,cmtip]{xy}
\usepackage[utf8]{inputenc}
\usepackage[T1]{fontenc}

\newtheorem{thm}{Theorem}[section]
\newtheorem{cor}[thm]{Corollary}

\newtheorem{prop}[thm]{Proposition}
\newtheorem{lm}[thm]{Lemma}

\theoremstyle{definition}
\newtheorem{de}[thm]{Definition}
\newtheorem{ex}[thm]{Example}

\theoremstyle{remark}
\newtheorem{rmk}[thm]{Remark}

\def\gens{m}
\def\GIT{/\!\!/}
\def\PP{\mathscr{P}}

\def\z{\zeta}

\def\Q{\mathbb Q}
\def\R{\mathbb R}
\def\C{\mathbb C}
\def\N{\mathbb N}

\def\Z{\mathbb Z}

\def\KH{K\"{a}hler }
\def\Kh{K\"{a}hler}
\def\P{{\mathbb P}}

\def\a{\alpha}

\def\Fi{\varphi}

\def\fun{\rightarrow}

\def\M{\mathfrak{M}}

\def\CM0{\C[\M_0]}

\def\F{\mathfrak{F}}

\def\Con1{Con_1(\M)}

\def\FF{\mathscr{F}}
\def\EE{\mathscr{E}}

\def\T{\mathbb{T}}

\def\MM0{\mathfrak{M}_{\tiny{(\zeta_{\mathbb{R}},0})}(Q,{\normalfont\textbf{v}},{\normalfont\textbf{w}})}

\def\MG0{\mathcal{M}_{0,\z_\C}(Q,{\normalfont\textbf{v}},{\normalfont\textbf{w}})}

\def\iso{\cong}

\def\om{\omega}

\def\k{\mathbb{K}}

\def\ph{pseudoholomorphic}

\def\CC{$\C^*$}

\def\Fil{\mathscr{F}}

\def\Core{\mathrm{Core}}

\def\llambda{\Fi}

\def\ku{\k [\![u]\!]}
\def\kuu{\k(\!(u)\!)}
\def\karo{\diamondsuit}

\def\MM{\mathcal{M}}

\def\CP{\mathbb{C}P}

\def\MB{Morse--Bott}

\def\MBF{Morse--Bott--Floer}

\makeatletter
\newcommand{\doublewidetilde}[1]{{%
  \mathpalette\double@widetilde{#1}%
}}
\newcommand{\double@widetilde}[2]{%
  \sbox\z@{$\m@th#1\widetilde{#2}$}%
  \ht\z@=.9\ht\z@
  \widetilde{\box\z@}%
}
\makeatother
\begin{document}

\title
[Semiprojective toric manifolds]
{Quantum cohomology and Floer invariants\\ of semiprojective toric manifolds}
\author{Alexander F. Ritter}
\address{\uvuciautor A. F. Ritter, 
Mathematical Institute, University of Oxford, 
OX2 6GG, U.K.}
\email{ritter@maths.ox.ac.uk}
\author{Filip Živanović}
\address{\uvuciautor F. T. Živanović, 
Simons Center for Geometry and Physics, 
Stony Brook, NY 11794-3636, U.S.A.}
\email{fzivanovic@scgp.stonybrook.edu}

\begin{abstract} 
We use Floer theory to describe invariants of 
symplectic $\C^*$-manifolds admitting several commuting $\C^*$-actions.
The $\C^*$-actions induce filtrations by ideals on quantum cohomology, as well as filtrations on Hamiltonian Floer cohomologies, and we prove relationships between these filtrations.
We also carry this out in the equivariant setting, in particular $\C^*$-actions then give rise to Hilbert-Poincar\'{e} polynomials on ordinary cohomology that depend on Floer theory.
For semiprojective toric manifolds, we obtain an explicit presentation for quantum and symplectic cohomology in the Fano and CY setting, both in the equivariant and non-equivariant setting.
\end{abstract}

\maketitle
\setcounter{secnumdepth}{3}
\setcounter{tocdepth}{1}

\tableofcontents  %

\section{Introduction}\label{Introduction}

\subsection{Motivation}\label{Introduction Motivation}
We use novel Floer-theoretic methods to describe quantum cohomology for non-compact symplectic manifolds $(Y,\omega)$ admitting several commuting Hamiltonian $S^1$-actions, in the non-equivariant and equivariant setting.
Toric manifolds are prototypical such spaces: they arise in algebraic geometry by compactifying an algebraic torus $(\C^*)^n$; any one-parameter subgroup $\C^*\to (\C^*)^n$ induces a $\C^*$-action on $Y$;
and these give commuting Hamiltonian $S^1$-actions for a suitable symplectic form $\omega$.\vspace{2mm}

We briefly recall the story for closed symplectic manifolds $M$. In this case, quantum cohomology and Hamiltonian Floer cohomology can be identified, irrespective of the choice of the Hamiltonian function $H: M \rightarrow \R$. Seidel's seminal paper \cite{Sei97} observed that a Hamiltonian $S^1$-action induces a natural chain isomorphism between Hamiltonian Floer complexes, and thus an automorphism of $QH^*(M)$. This automorphism is quantum multiplication by an invertible element (``\emph{Seidel element}''), and it defines
the \emph{Seidel representation}: group multiplication between such $S^1$-actions induces quantum product of their Seidel elements. McDuff--Tolman \cite{mcduff2006topological} used this to obtain a symplectic proof of the Batyrev--Givental presentation of quantum cohomology for closed Fano toric manifolds. Namely, the $S^1$-rotations around the toric divisors $D_i$ (corresponding to the facets of the moment polytope) give rise to generators $\mathrm{PD}[D_i]\in H^2(M)$ for quantum cohomology as a ring, and the group relations between these rotations give rise to all the \emph{quantum Stanley-Reisner relations} required to present quantum cohomology as a ring (in addition to the classical linear relations in $H^2(M)$ between the $\mathrm{PD}[D_i]$).

For non-compact symplectic manifolds, the story is more intricate: Hamiltonian Floer cohomology depends on the growth of the Hamiltonian at infinity. Taking a direct limit over these groups as the growth diverges gives rise to symplectic cohomology $SH^*(Y)$, which comes with a ring homomorphism %
\begin{equation}\label{Equation canonical hom}
c^*:QH^*(Y) \to SH^*(Y).
\end{equation}
The first author in \cite{R14} generalised the Seidel representation when $Y$ is {\bf convex}: outside of the interior of a compact subset, $Y$ can be identified with a positive symplectisation $Y^{\mathrm{out}}\cong \Sigma \times [R_0,\infty)$ of a closed contact manifold $(\Sigma,\alpha)$, using the symplectic form $d(R\alpha)$.
Most symplectic literature on non-compact symplectic manifolds in the past two decades assumed this convexity; unfortunately it is extremely restrictive in the context of toric manifolds. For non-compact Fano toric manifolds $Y$, \cite{R16} required another harsh condition: that all ``canonical rotations'' around toric divisors agree with the Reeb flow at infinity (the actual condition is mildly more general). Under those assumptions, \eqref{Equation canonical hom} is a quotient map, indeed it is localisation at the Seidel elements associated to the canonical rotations. Interestingly, these algebras differ: it is $SH^*(Y)$, not $QH^*(Y)$, which recovers what would normally be Batyrev's combinatorial presentation of quantum cohomology. We shall discuss these results in greater detail in \cref{Example intro toric varieties}. Our original motivation for this paper was to use our foundational work \cite{RZ1,RZ2} to non-compact toric manifolds in order to remove all those restrictive conditions. We were able to do this for essentially all toric manifolds, more precisely:
all {\bf semiprojective toric manifolds}. These were first introduced by 
Hausel--Sturmfels \cite{HS02}, and can be characterized as the smooth GIT/{\Kh} quotients of $\C^n$ under a torus action. Some authors consider this to be the definition of toric manifolds.

Our foundational work applies to many important classes of spaces\footnote{Nakajima quiver varieties;
moduli spaces of Higgs bundles;
{\Kh} quotients of $\C^n$; 
hypertoric varieties;
cotangent bundles of flag varieties;
negative complex vector bundles;
crepant resolutions of quotient singularities $\C^n/G$ for finite subgroups $G\subset SL(n,\C)$; all Conical Symplectic Resolutions; and most
equivariant resolutions of affine singularities.} from algebraic geometry and representation theory, described in \cite{RZ1}.
We defined {\bf symplectic $\C^*$-manifolds} to encompass all of these spaces: they are non-compact symplectic manifolds $(Y,\omega)$ with a Hamiltonian $S^1$-action that extends to a {\ph} $\C^*$-action $\Fi$ (for some $\omega$-compatible almost complex structure $I$). Additionally, they admit 
a proper {\ph}\footnote{On $\Sigma \times [R_0,\infty)$, a choice of $d(R\alpha)$-compatible almost complex structure of contact type is made.} map 
\begin{equation}\label{Equation intro Psi}
\Psi: Y^{\mathrm{out}}\to \Sigma \times [R_0,\infty),
\end{equation}
which sends the Hamiltonian $S^1$-vector field to the Reeb vector field (up to rescaling by a non-zero constant\footnote{The Reeb field is the Hamiltonian vector field $X_R$ of the radial coordinate $R\in [R_0,\infty).$ To define quantum and symplectic cohomology, and their filtrations, one can allow $\Psi_* X_{S^1}=X_{fR}$ for a function $f:\Sigma \to (0,\infty)$ invariant under the Reeb flow \cite[Remark 1.2]{RZ1}. Although this is a mild extension, it was crucial for example in toric examples \cite{R16}. However, to construct the {\MBF} spectral sequence discussed later, we needed $f$ to be constant.}).
We say $Y$ is {\bf globally defined over $X$} if a stronger condition holds: there is a proper {\ph} $\C^*$-equivariant map $\Psi: Y \to X$, where $X$ is a convex symplectic manifold, with a choice of compatible almost complex structure, admitting a $\C^*$-action, whose $S^1$-part is a Hamiltonian flow that at infinity agrees with the Reeb flow. Many examples arise in this way with $X \cong \C^m$, for some large $m$, with $\Sigma \cong S^{2m-1}$, and $R$ a constant multiple of $\|z\|^2$.

\begin{rmk}
We do not require \eqref{Equation intro Psi} to be symplectic: that would imply convexity of $Y.$ Indeed the target of $\Psi$ is often of much higher dimension than $Y$ with projective varieties as fibres, rich in closed holomorphic curves; so $\omega$ is often non-exact at infinity, in contrast to the convex case.

The map \eqref{Equation intro Psi} is not part of the datum of a symplectic $\C^*$-manifold $(Y,\omega,\Fi,I)$, it is only a technical tool that ensures quantum cohomology and Floer cohomology are well-defined despite the non-compactness of $Y$. For this reason, an {\bf isomorphism of symplectic $\C^*$-manifolds} means a {\ph } $\C^*$-equivariant symplectomorphism, without any conditions on the possible $\Psi$-maps.
\end{rmk}

In \cite{RZ1,RZ2} we used Floer theory to construct a $\Fi$-dependent filtration by ideals on quantum cohomology. We also developed the equivariant theory in \cite{RZ3} for quantum and Floer cohomology. The equivariant theory is arguably more interesting here, because the non-equivariant quantum product is often classical cup product (e.g.\;this occurs for all quiver varieties and indeed all Conical Symplectic Resolutions, but typically not for semiprojective toric manifolds). The Floer theory also displays dramatically different behaviour in the equivariant setting, as shown in explicit computations in \cite{RZ3}. We will briefly recall some of the foundational results in the next section, keeping the discussion as brief as possible: the interested reader can find more details in the introductions to our foundational papers \cite{RZ1,RZ2,RZ3}.
After the overview, we return to the goal of this paper: to consider situations where two or more commuting $\C^*$-actions are in play, and to relate their respective filtrations.

\subsection{A brief overview of the structural properties of invariants for symplectic $\C^*$-manifolds}
\label{Subsection Introduction A brief overview of the structural properties}

In this section, $(Y,\omega)$ is any symplectic $\C^*$-manifold with $\C^*$-action $\Fi$. 
Let $H$ denote the moment map of the corresponding Hamiltonian $S^1$-action by $\Fi$. We use Hamiltonians $H_{\lambda}=c(H): Y \to \R$ that are suitable functions of $H$, with slope $c'(H)=\lambda$ when $H$ is very large (i.e. outside of a compact set).
The ($\Fi$-dependent) symplectic cohomology $SH^*(Y,\Fi)=\varinjlim HF^*(H_{\lambda})$ is the direct limit over continuation maps of the Hamiltonian Floer cohomology groups. 
We call the periods of the $S^1$-action the ``\emph{non-generic}'' slopes: they form a discrete subset of $[0,\infty)$ that can be explicitly described in terms of the weights of the linearised $S^1$-action at the components $\F_\a$ of the fixed locus $\F$. 

Regarding conventions: we refer to \cite{RZ1} for the definition of the Novikov field $\k$ that we use (over any base field $\mathbb{B}$), in particular $T$ will denote the formal Novikov variable. There we also explain the additional technical assumptions on $Y$ needed to make sense of Floer theory by classical transversality methods \`{a} la Floer-Hofer-Salamon, as opposed to the recent global Kuranishi structure methods.

We now summarise some outcomes of our previous work. The {\bf rotation class} $Q_{\Fi}\in QH^*(Y)$, below, is the generalisation of the Seidel element for the $S^1$-action $\Fi$ to our non-compact setting. However, unlike Seidel elements, $Q_{\Fi}\in QH^*(Y)$ is \textit{not} invertible, whereas $c^*(Q_{\Fi})\in SH^*(Y,\Fi)$ is invertible.

\begin{thm}[{\cite{RZ1}}]\label{Theorem Intro SH as loc of QH}
The homomorphism $c^*:QH^*(Y)\to SH^*(Y,\Fi)$ is surjective. It equals localisation at 
$Q_{\Fi}\in QH^{2\mu}(Y),$
where $\mu$ is the Maslov index of the $S^1$-action $\Fi$.
$$
SH^*(Y,\Fi) \cong QH^*(Y)_{Q_{\Fi}} \cong QH^*(Y)/E_0(Q_{\Fi}),
$$
where
$
E_0(Q_{\Fi})=\ker c^* \subset QH^*(Y)$ 
is the generalised $0$-eigenspace of quantum product by $Q_{\Fi}$.
For $N^+\in \R$ just above $N\in \N$, the continuation map $c^*_{N^+}$ (whose direct limit is $c^*$ in \eqref{Equation canonical hom} as $N\to \infty$), 
\begin{equation}\label{Equation cNplus maps intro}
c^*_{N^+}: QH^*(Y) \to HF^*(H_{N^+})\cong QH^*(Y)[2N\mu],
\end{equation}
is\footnote{Above, $A[d]$ means we shift a graded group $A$ down by $d$, so $(A[d])_n=A_{n+d}$.} quantum product $N$ times by $Q_{\Fi}\in QH^*(Y)$.
\end{thm}

\begin{ex}[{\cite{R16}}]\label{Introduction Example O-k over CPm}
Let $Y$ be the negative line bundle $\mathcal{O}(-k) \to \C P^m$, with the natural $\C^*$-action $\Fi$ on fibres. It is Fano for $1\leq k \leq m$ and Calabi-Yau (CY) for $k=1+m$. In the CY case, $SH^*(Y)=0$. In the Fano case: $QH^*(Y)\cong \k[x]/(x^{1+m}-(-kx)^k T^{1+m-k})$ and its localisation at $Q_{\Fi}=-kx$ is
$$QH^*(Y)_{Q_{\Fi}} \cong SH^*(Y) \cong \k[x]/(x^{1+m-k}-(-k)^kT^{1+m-k}).$$
\end{ex}

In \cite{RZ2}, we constructed a filtration on these Floer complexes, which gives rise to a {\MBF} spectral sequence converging to $SH^*(Y)$ (or, respectively, $HF^*(H_{\lambda})$ if one ignores the columns in the spectral sequence corresponding to slopes larger than $\lambda$). The contributions to the $E_1$-page of the {\MBF} spectral sequences is essentially the ordinary cohomology of the fixed locus $\F = Y^{S^1}=Y^{\C^*}$, and of the {\MB} manifolds $B_{p,\beta}$ of non-constant $1$-periodic Hamiltonian orbits of $H_{\lambda}$ corresponding to $S^1$-orbits of various periods $p=c'(H)$.

\begin{thm}[{\cite{RZ1,RZ2}}]\label{Theorem 1 filtration}
$QH^*(Y)$ is filtered by graded ideals $\FF^{p}$ ordered by  $p\in \R\cup\{\infty\}$,
\begin{equation}\label{Equation introduction filtration}
\FF^{p} :=\bigcap_{\mathrm{generic}\,\lambda\geq p} \left(\ker c_{\lambda}^*:QH^*(Y)\to HF^*(H_{\lambda})\right), \qquad \FF^{\infty}:=QH^*(Y),
\end{equation} 
where $c_\lambda^*$ is a Floer continuation map, a grading-preserving $QH^*(Y)$-module homomorphism.
Note:
$$
(\FF^{p}=QH^*(Y)\textrm{ for some }p<\infty) \Leftrightarrow (Q_{\Fi}\in QH^{2\mu}(Y)\textrm{ is nilpotent}) \Leftrightarrow SH^*(Y,\Fi)=0,
 $$
and this holds for example whenever $c_1(Y)=0$.

The filtration is an invariant of $Y$ up to isomorphism of symplectic $\C^*$-manifolds.
The real parameter $p$ is a part of that invariant, and has a geometric interpretation: $x\in \Fil^p$ means that $x$ can be represented as a Floer chain involving non-constant $S^1$-orbits of period $\leq p.$ In terms of the {\MBF} spectral sequence, it means that the class $x$ is eventually killed by the columns of slope at most $p$.

If we vary the choice of $\C^*$-action, we obtain a map
\begin{equation}
\label{filtration functor}
\{\text{contracting } \C^*\textrm{-actions on }Y\} \to \{\R_{\infty}\textrm{-ordered filtrations of }QH^*(Y)\}, \ \ \Fi \mapsto \Fil_\Fi^{p},
\end{equation}
such that powers cause a dilation in the ordering: $\Fil_{\Fi^k}^{p}=\Fil_{\Fi}^{k p}$. 

The filtration also determines a periodic persistence module with a graded periodic barcode, which encodes birth-death phenomena of Floer invariants. So in \eqref{filtration functor}, we could instead map to the periodic persistence modules, and declare the ``distance'' between two actions to be the interleaving distance between persistence modules (equivalently, the bottleneck distance between barcodes).
\end{thm}

In \cite{RZ3}, we constructed three versions of $S^1$-equivariant quantum and symplectic cohomology. We append the letter $E^-,E^{\infty},E^+$ to denote these: the version depends on the choice of $\k$-module $E^-H^*(\mathrm{point})=\ku$, $E^{\infty}H^*(\mathrm{point})=\kuu$, $E^+H^*(\mathrm{point})=\mathbb{F}=\kuu/u\ku$, where $\k$ is the Novikov field and $u$ is the equivariant parameter in degree $|u|=2$. The fact that we allow series in $u$, in the power series ring $\ku$ and Laurent series ring $\kuu$, is motivated by the possibility that arbitrarily high powers of $u$ may contribute to the Floer differential, whereas this does not occur for the differential for ordinary equivariant cohomology or for equivariant quantum products. Apart from this $u$-completion, $E^-H^*(Y)$ would be the usual Borel construction of $S^1$-equivariant cohomology, whereas $E^{\infty}H^*(Y)$ and $E^+H^*(Y)$ are motivated by ideas from cyclic homology. The $E^{\infty}$-theory is always the $u$-localisation of the $E^-$-theory, e.g.\;$E^{\infty}H^*(Y)\cong E^-H^*(Y)\otimes_{\ku}\kuu$.

In addition to the $S^1$-action $\Fi_{\theta}$ on $Y$, acting by $w=e^{2\pi i \theta}\in S^1\subset \C^*$, in Floer theory we also have an $S^1$-reparametrisation action on Hamiltonian $1$-orbits $x=x(t):S^1 \to Y$ (the generators of the Floer complex). Thus, after fixing a choice of weight $(a,b)\in \Z^2$, we make $w$ act by
$$
(w\cdot x)(t) = \Fi_{a\theta}(x(t-b\theta)).
$$

\begin{de}\label{Definition free weight}
We call $(a,b)$ a {\bf free weight} if the {\MB} manifolds $B_{p,\beta}$ of non-constant $1$-orbits in $Y$ are never fixed points of the above $S^1$-action.
More explicitly, one looks at the weights $m\in \Z$ of the linearised $S^1$-action at each fixed component $\F_\a\subset \F$; then $(a,b)$ is free precisely if $ma\notin b\Z_{\geq 1}$ for any positive weight $m\geq 1\in \N$ that occurs. For example: $(a,b)\neq (0,0)$ is free if $a\cdot b \leq 0$.
\end{de} 

Classical ordinary and equivariant cohomology can be described as $\k$-modules as follows.

 \begin{lm}[\cite{RZ1,RZ3}]\label{Lemma equivariant formality}
There are non-canonical $\ku$-module isomorphisms
     $$
     E^-QH^*(Y)\cong H^*(Y)[\![u]\!]
     ,\quad
     E^{\infty}QH^*(Y)\cong  H^*(Y)(\!(u)\!)
     ,
     \quad
     E^{+}QH^*(Y)\cong H^*(Y)\otimes_{\k} \mathbb{F},
     $$
and the connected components $\F_\a\subset\F$, viewed as {\MB} manifolds of the moment map $H$, together with their (even) 
{\MB } indices $\mu_\a\geq 0$, determine ordinary cohomology as a $\k$-vector space,
$$
H^*(Y)\cong \oplus H^*(\F_\a)[-\mu_\a].
$$
 \end{lm}

For each $\karo\in \{ -,\infty,+\}$, there is a canonical map 
$$E^{\karo}c^*:E^{\karo}QH^*(Y) \to E^{\karo}SH^*(Y), \textrm{ where }E^{\karo}c^*=\varinjlim\,\left(E^{\karo}c^*_{\lambda}: E^{\karo}QH^*(Y) \to E^{\karo}HF^*(H_{\lambda})\right),$$
for generic slopes $\lambda \to \infty$.
Just like the kernels of $c^*_{\lambda}$ in \cref{Theorem 1 filtration} typically grow with $\lambda$, the same usually occurs with the kernels of $E^+c_\lambda^*$ (e.g.\;when $Y$ is CY, those kernels are exhausting because $SH^*(Y)=0$ and $E^+SH^*(Y)=0$).
We refined the $E^+$-filtration by such kernels by also taking into account the $u$-action:
for $j\in \Z\cup \{\pm\infty\}$, $p\in [0,\infty]$, define $\EE_{j}^p :=\bigcap_{\mathrm{generic}\,\lambda\geq p} E_{j}^{\lambda}$ where
\begin{equation}\label{Equation E+ filtration}
E_{j\leq 0}^{\lambda}:=\ker (u^{|j|+1} \cdot E^+c_{\lambda}^*), \qquad 
E_{j> 0}^{\lambda}:=u^{j-1}\ker E^+c_{\lambda}^*, \qquad 
E_{-\infty}^{\lambda}:=E^+QH^*(Y), \;
E_{\infty}^{\lambda}:=\{0\}.
\end{equation}
The $\ku$-submodules $\EE_{j}^p\subset E^+QH^*(Y)$ satisfy
$\EE_{j}^p\subset \EE_{j}^{p'}$ for $p<p'$, and $u\EE_{j}^p =  \EE_{j+1}^{p} \subset \EE_{j}^p.$

The $E^-$,\,$E^{\infty}$-theories behave differently: continuation maps become injective, as follows. 
\begin{thm}[Localisation and Injectivity Theorems, {\cite{RZ3}}]\label{Thm intro injectivity thm}
Let $(a,b)$ be a free weight. The following diagram of $\ku$-module isomorphisms commutes,
    $$
\begin{tikzcd}[column sep=0.6in]
E^-QH^*(Y)_u 
 \arrow[r, "(E^-c^*)_u","\cong"'] 
 \arrow[d, "\cong","u\otimes 1"']
 &
 E^{-} SH^*(Y)_u  
\arrow[d, "\cong"',"u\otimes 1"]
\\
E^{\infty}QH^*(Y)
\arrow[r, "E^{\infty}c^*","\cong"'] 
& E^{\infty} SH^*(Y)
\end{tikzcd}
$$
That the $u$-localisation of $c^*\!:E^-QH^*(Y) \to E^-SH^*(Y)$ is an isomorphism (the top row, above), implies in general that $E^-QH^*(Y) \to E^-HF^*(H_{\lambda})$ and $E^-QH^*(Y) \to E^-SH^*(Y)$ are injective.
\end{thm}
\noindent In view of this, we constructed the filtration differently for $E^-,E^{\infty}$.
Abbreviate $E^-_{\lambda}:=E^-HF^*(H_{\lambda})$ with its $u$-torsion\footnote{submodule of elements $x$ such that $u^k \cdot x=0,$ for some $k\in\N$} $T_{\lambda},$ and $E^-:=\varinjlim E^-_{\lambda}=E^-SH^*(Y)$ with $u$-torsion $T$. Then:
\begin{de}
For any $p\in [0,\infty)$, we define the {\bf $\mathbf{E^-}$-filtration} on $E^-QH^*(Y)$ by
$$
\FF_{j}^p:=\!\!\!\!\!\!\!\!\bigcap_{\textrm{generic }\lambda\geq p} \!\!\!\!\!\! \ker \left(E^-QH^*(Y) \stackrel{E^-c_{\lambda}^*\;}{-\!\!\!\longrightarrow} E^-_{\lambda}/(T_{\lambda} + u^jE^-_{\lambda})\right), \quad\textrm{so }\FF_{0}^{\lambda}=E^-QH^*(Y),\; \FF_{\infty}^{\lambda}=\ker E^-c_{\lambda}^*,
$$
for $j\in \Z_{\geq 0}$;\, $\FF_j^{\infty}:=\ker (E^-c^*:E^-QH^*(Y)\to E^-/(T+u^j E^-))$; and $\FF_{j}^p:=E^-QH^*(Y)$ for $j\in \Z_{< 0}$.
\end{de}

The $\ku$-submodules $\FF_{j}^p\subset E^-QH^*(Y)$ satisfy
$\FF_{j}^p\subset \FF_{j}^{p'}$ for $p<p'$, $u\FF_j^{p}   \subset \FF_{j+1}^{p} \subset \FF_j^{p}$, and $u\FF_j^{p} =  \FF_{j+1}^{p} \cap u \,E^-QH^*(Y)$. The $E^-$-filtration induces a filtration on $E^{\infty}QH^*(Y)$ which, together with the above $\EE$-filtration on
$E^+QH^*(Y)$, is
consistent with the short exact sequence $0 \to E^-QH^*(Y)[-2]\stackrel{u}{\to}E^{\infty}QH^*(Y)\to E^+QH^*(Y) \to 0$; but we omit this discussion.

There is an equivariant quantum product on $E^-QH^*(Y)$ and $E^{\infty}QH^*(Y)$, but there is usually no product structure on equivariant symplectic cohomology, except for the (free) weights $(a,0)$. In the latter case, the $\ku$-submodules $\FF_j^{p} \subset E^-QH^*(Y)$ are ideals (w.r.t.\;equivariant quantum product).

\begin{de}
For $p\in [0,\infty]$, the {\bf slice dimensions}, {\bf slice series}, and  {\bf filtration polynomial} are
    $$
    d_{j}^{p}:=\dim_{\k} \FF_{j}^{p}/u\FF_{j-1}^{p},
    \qquad
    s_{p}(t) := \sum_{j\geq 0} d_j^{p} t^j \qquad \textrm{ and } \qquad 
    f_{p}(t) := \sum_{j\geq 0} f_j^{p} t^j = \sum_{j\geq 0} (d_j^{p}-d_{j+1}^{p}) t^j.
    $$
\end{de}
By considering the map $E^-c^*_{p^+}: E^-QH^*(Y)\to E^-_{p^+}/T_{p^+}$, where $p^+>p$ is a (generic) slope sufficiently close to $p$, and noting that its Smith normal form over the principal ideal domain $\ku$ will be a finite collection of powers of $u$, called {\bf invariant factors}, we showed that:
\begin{thm}
The $d_{j}^{p}$ and $f_j^p$  
are finite invariants, and they satisfy the following properties:
\begin{align*}
0\leq d_j^{p} \leq  \dim_{\k}H^*(Y),
\qquad\quad
d_j^p \leq d_{j}^{p'} \;&\textrm{ for }\; p\leq p', \qquad\quad
d_j^p \geq d_{j+1}^p,
\\
f_j^p = \#\{\textrm{invariant factors }u^j\textrm{ of } E^-c_{p^+}\}, 
\qquad
 &\textrm{ and } \qquad
d_j^p = \#\{\textrm{invariant factors }u^{\geq j}\textrm{ of } E^-c_{p^+}\}.
\end{align*}
In particular, $f_p$ is always a polynomial, and $s_p$ is a polynomial precisely when $E^-c_{p^+}$ is injective. 
\end{thm}

\begin{de}
    Abbreviate $\PP := E^-QH^*(Y)/uE^-QH^*(Y)$. This is filtered, for $j\in \Z$, by 
    $$\PP_j^p:=\FF_{j}^{p}/u\FF_{j-1}^{p}
    =
    \FF_{j}^{p}/(\FF_{j}^{p} \cap u \,E^-QH^*(Y)), \qquad \PP_{\infty}^{p}:=0, \qquad (\textrm{so }\PP_j^p=\PP\textrm{ for }j\leq 0),$$ 
    satisfying $\PP_{j+1}^p\subset \PP_{j}^p$ and $\PP_j^p\subset \PP_j^{p'}$ for $p\leq p'.$
    Define $\mathrm{gr}^p_j\PP:=\PP_j^p/\PP_{j+1}^p$ and $\mathrm{gr}^p\,\PP:=\oplus_{j\geq 0} \mathrm{gr}^p_j\PP$.
\end{de}

\begin{cor}
There is a canonical identification
$\PP\cong QH^*(Y)$ with non-equivariant quantum cohomology. Thus quantum cohomology is filtered by $\PP_j^p$ and its Hilbert-Poincar\'{e} series equals $f_{p}(t)$: 
$$f_j^{p}=\dim_{\k} \PP_j^{p}/\PP_{j+1}^{p} = \dim_{\k}\mathrm{gr}_j^p \PP.$$
\end{cor}

\subsection{Examples of algebraic torus actions on symplectic $\C^*$-manifolds}

We are now interested in situations where a symplectic manifold $(Y,\omega)$ admits two or more commuting $\C^*$-actions that yield the structure of a symplectic $\C^*$-manifold. We first show that there are many classes of such spaces.

\begin{ex}[Equivariant projective morphisms]\label{Example intro CSR have global map} 
Let $Y$ be a variety, with an action by an algebraic torus $\mathbb{T}\cong (\C^*)^d$, admitting a $\mathbb{T}$-equivariant projective morphism $\pi: Y \to X$ to an affine variety. By replacing $X$ with the image of $\pi$ we may assume that $\pi$ is surjective. We assume that $Y$ is non-compact, equivalently $X\neq \mathrm{point}.$

Suppose first that the $\mathbb{T}$-action on $X$ is contracting.\footnote{By this, we mean that there is a unique fixed point, denoted $0\in X$, and an isomorphism $\mathbb{T}\cong (\C^*)^d$ such that each $\C^*$-factor contracts $X$ to $0$, by acting with $t\in \C^*$ and letting $t\to 0$. Algebraically: the coordinate ring $\C[X]$ is $\N^d$-graded with $\C[X]_0=\C$.}
We then show that $Y$ is a quasi-projective variety admitting a K\"{a}hler form $\omega$, such that the maximal compact subgroup $\mathbb{T}_{\R}\subset \mathbb{T}$ yields 
an $\omega$-Hamiltonian action on $Y$.
The $\C^*$-actions $\Fi_v$ induced by $1$-parameter subgroups $v:\C^* \to \mathbb{T}$ are labelled by a lattice, $N:=\mathrm{Hom}(\C^*,\mathbb{T})\cong \Z^d$.
Those that are contracting define a semigroup $N_+\subset N$, containing\footnote{it may be larger, e.g. if $\mathbb{T}=(\C^*)^2=\C^* \times \C^*$, acting on $\C^2$ with weights 
$(1,2)$ and $(1,3)$ for the first and second $\C^*$-factor respectively, then $N_+=\{(m,n)\in \Z^2:m+n>0, \ 2m+3n>0\}\supsetneq \N_0^2\setminus {0}$.
} 
$\N_0^d\setminus \{0\}$. We show that $(Y,\omega,\Fi_v)$ is a symplectic $\C^*$-manifold for any $v\in N_+,$ in particular we construct an explicit $\Psi$-map as in \eqref{Equation intro Psi}.
Two such classes of examples, discussed in \cite{RZ1}, are: Conical Symplectic Resolutions (CSRs) and crepant resolutions of quotient singularities.

Now suppose the $\mathbb{T}$-action on $X$ is not contracting. Then $N_+\subset N\cong \Z^d$ may not contain $\N^d\setminus \{0\}$. Provided $N_+\neq \emptyset$, there will be a non-trivial algebraic subtorus $\mathbb{T}'\subset \mathbb{T}$ whose action on $X$ is contracting (one can make $\mathbb{T}'$ maximal with respect to inclusion, but it is typically not unique). In any case, the previous discussion applies to the $\mathbb{T}'$-equivariant projective morphism $\pi: Y \to X$.
\end{ex}

\begin{ex}
[\Kh/GIT quotients of $\C^n$]
\label{Example intro GIT}
An instance of \cref{Example intro CSR have global map} are
\KH reductions of unitary actions
$K \subset U(V)$ on a complex vector space $V$, or equivalently 
(by the Kempf--Ness theorem) GIT quotients of the complexified actions %
$K_\C=:G \subset GL(V)$.
By the GIT picture, they always come with a projective morphism
$\pi: Y= V\GIT_{\chi} G \fun V \GIT_0 G = X $ to the affine GIT quotient, where $\chi:G\fun \C^*$ is a character.
The scalar matrices $\C^* \iso S \leq GL(V)$ cannot be a subgroup of $G$, as otherwise
$X$ would be a single point,\footnote{Given any $z\in V,$ 
the closure of its orbit $G z \supset S z$ would contain the origin ($\lim_{s\fun 0} s \cdot z =0$), 
thus by definition
$[z]=[0]$ in the GIT quotient $X.$} which we disallow.
Hence, there is an induced $\pi$-equivariant action of $S$ on $Y$ and $X$, which is contracting.\footnote{$\lim_{s\fun 0} s \cdot [z]=
\lim_{s\fun 0} [s\cdot z]=[0].$}
When $G$ is a torus, this gives rise to one of several equivalent definitions of semiprojective toric manifolds.
\end{ex}

\begin{ex}[CSRs] We discussed Conical Symplectic Resolutions (CSRs) extensively in \cite{RZ1}. A CSR, say $Y$, typically contains many different \CC-actions, arising from compositions $\Fi:=\phi^k \circ G$ of a canonically given contracting \CC-action $\phi$ and any $1$-parameter \CC-subgroup $G$ %
of a maximal torus $T \leq \mathrm{Symp}_{\phi}(Y,\om_\C)$ of symplectomorphisms commuting with $\phi$. In most examples,\footnote{This is always the case for quiver varieties, hypertoric varieties, and resolutions of Slodowy varieties, which constitute the main examples of CSRs.}
this torus is non-trivial. 
Subgroups $G$ for which $\Fi$ is contracting, for fixed $k$, constitute a convex subset $K_k$ %
of the lattice $\Lambda=\mathrm{Hom}(\C^*,T).$ 
For CSRs, quantum product is ordinary cup product, so \eqref{filtration functor} determines a $K_k$-labelled family of filtrations on $H^*(Y)$ by cup-product ideals,
where $\cup_{k\in\N} K_k=\Lambda$ exhausts the whole lattice. 
We discussed this for $Y=T^*\CP^n$ explicitly in \cite{RZ2}.
\end{ex}

\begin{ex}[$A_2$-singularity]\label{Example running example of intro}
CSRs of the lowest dimension are ADE resolutions.\footnote{Minimal resolutions of quotient singularities $M\fun \C^2/\Gamma,$ where $\Gamma \leq SL(2,\C)$ is a finite subgroup.} Let us consider the $A_2$ case, %
the minimal resolution $\pi:M \fun \C^2/(\Z/3)$ of the quotient singularity for the action $(x,y)\mapsto (\zeta x,\zeta^{-1}y)$ of third roots of unity $\zeta$ on $\C^2$. Embedding $\C^2/(\Z/3)\hookrightarrow \C^3$, $[x,y]\mapsto (x^3,y^3,xy)$, the resolution $M $ arises from blowing up the image variety $V(XY-Z^3)\subset \C^3$ at $0$. 
The core $\mathfrak{L}=\pi^{-1}(0)=S^2_1 \cup S^2_2$ consists of two copies of $S^2$ intersecting transversely at a point $p$.
Two natural commuting $\C^*$-actions are obtained by lifting via $\pi$ the following actions\footnote{The respective $\Psi$-maps are
$\Psi=(X^2,Y,Z^2): M \to \C^3$ and $\Psi=(X,Y^2,Z^2)$, using a weight $2$ diagonal $\C^*$-action on $\C^3$. For the composite action (a), $\Psi=(X^2,Y^2,Z^3)$ for the weight $6$ diagonal $\C^*$-action on $\C^3$.} on $V(XY-Z^3)\subset \C^3$:
\\
(b) %
$(tX,t^2Y,tZ)$; $\F=S_1^2 \sqcup p_2$, where $S_1^2=\F_{\min}$.
\\
(c) %
$(t^2X,tY,tZ)$; $\F=p_1\sqcup S^2_2$, where $S_2^2=\F_{\min}$.
\\
This yields an action by a torus $T:=(\C^*)^2$. The diagonal $\C^*$-subgroup in $T$ yields the classical McKay $\C^*$-action on $M$ induced by $(x,y)\mapsto (tx,ty)$, so lifting the following action:\\
(a) %
$(t^3 X, t^3Y, t^2Z)$; $\F=p_1 \sqcup p \sqcup p_2$ (3 points), where $p=\F_{\min}$ and $p_i\in S^2_i$.
\\
Even when two actions have the same fixed loci with the same {\MB} indices, the $\R$-ordering by periods can distinguish the filtrations: let $\Fi_1$ correspond to $(-1,3) \in \Z^2 \cong \mathrm{Hom}(\C^*,T)$, and $\Fi_2$ to $(-1,4)$,
then we find that: $\F=p_1 \sqcup p \sqcup p_2$ with {\MB} indices $2,2,0$ respectively,
$\Fil_{1/5}^{\Fi_1}= \k p_1$, $\Fil_{2/5}^{\Fi_1}=\k p_1 \oplus \k p$, whereas $\Fil_{1/7}^{\Fi_2}= \k p_1$, and $\Fil_{2/7}^{\Fi_2}=\k p_1 \oplus \k p$ (see \cite[Example 1.49]{RZ1}). 
\end{ex}

\begin{ex}[Toric varieties]\label{Example intro toric varieties}
By definition, a toric variety $Y$ contains an open dense algebraic torus $T\cong (\C^*)^n$, whose action extends to the whole space. Any $1$-parameter subgroup $\C^*\to T$ defines a holomorphic $\C^*$-action. 
Recall that a toric manifold can be described by combinatorial data: a moment polytope $$\Delta=\{x\in \R^n:\langle x,e_i\rangle \geq \lambda_i \},$$ 
which is also the image of the moment map $\mu:Y \to \Delta$ for the $(S^1)^n$-action. Abbreviate $z^{e_i}:=z_1^{e_{i,1}}\cdots z_n^{e_{i,n}}$, where $e_i=(e_{i,1},...,e_{i,n})\in \Z^n$ are the minimal ray generators of the fan, i.e.\;the primitive inward normals for the facets $\{x:\langle x,e_i\rangle = \lambda_i\}$ of $\Delta$. The \emph{Landau--Ginzburg superpotential} is $W:=\sum T^{-\lambda_i} z^{e_i}$, where $T$ is the formal variable that defines the Novikov ring $\k$, and its Jacobian ring is
$$\mathrm{Jac}(W):=\k[z_i^{\pm 1}]/(\partial_{z_i} W).$$
A classical result due to Batyrev and Givental states that for closed Fano toric manifolds $M$, we have isomorphism $QH^*(M)\iso \mathrm{Jac}(W)$, via $\mathrm{PD}[D_i]\mapsto T^{-\lambda_i}z^{e_i}$. The derivatives $\partial_{z_i}W$ are the \emph{linear relations} mentioned in \cref{Introduction Motivation}, whereas the \emph{quantum Stanley-Reisner relations} are automatic in $\mathrm{Jac}(W)$ because they generate the $\Z$-linear relations between the edges $e_i\in \Z^n$.
For non-compact Fano toric manifolds $Y,$ the first author showed that this no longer holds, rather it is symplectic cohomology which recovers the Jacobian ring \cite{R16},
$$
SH^*(Y) \cong \mathrm{Jac}(W).
$$
The natural Hamiltonian $S^1$-rotations $\Fi_i$ around the toric divisors $D_i=\{y\in Y:\langle \mu(y),e_i\rangle = \lambda_i \}\subset Y$ yield well-defined classes $Q_{\Fi_i} \in QH^2(Y)$, which become invertible in $SH^*(Y,\Fi)$ via $c^*$. Explicitly,
$$c^*:QH^*(Y)\to SH^*(Y)\cong \mathrm{Jac}(W), \quad \mathrm{PD}[D_i]=Q_{\Fi_i}  \mapsto c^*Q_{\Fi_i} \mapsto T^{-\lambda_i} z^{e_i}.$$  
 Also $D_i:=\min H_i$ for the Hamiltonian $H_i$ generating the $S^1$-action $\Fi_i$, so it is part of $\mathrm{Fix}(\Fi_i)\subset Y$.
These results from \cite{R16} required harsh assumptions on the non-compact Fano toric manifold $Y$: \cite{R16} assumed that $Y$ is convex (so $\Psi$ in \eqref{Equation intro Psi} is a symplectic isomorphism) and that all $\Fi_i$ essentially agree with the Reeb flow at infinity (the actual condition is slightly milder). These conditions hold for all negative line bundles over closed toric manifolds whose total space is Fano (e.g.\,\cref{Introduction Example O-k over CPm}). 
However, it was unclear how many more interesting classes of examples satisfy those harsh conditions, and whether the results persist without those conditions. The main application of this paper is to lift all the restrictive conditions above, so that this holds for all Fano semiprojective toric manifolds.  
\end{ex}

\begin{ex}[Semiprojective toric manifolds]\label{Example def semiprojective toric mfds intro}
A {\bf semiprojective toric variety} is a quasi-projective toric variety defined by a fan $\Sigma\subset N\cong \Z^n$ whose support $|\Sigma|\subset N_{\R}\cong \R^n$ (the union of all cones) is convex and of dimension $n$. One can then build a $\mathbb{T}$-equivariant projective morphism $Y \to X$ as in \cref{Example intro CSR have global map}, where $X$ is a normal affine toric variety whose torus $T$ contracts $X$ to a point. The map $\pi$ is induced by a quotient map $\mathbb{T} \to \mathbb{T}/S \cong T$ between those tori (which are open dense subsets of $Y,X$). 
We wish to avoid the case $X=\mathrm{point}$, i.e.\;we assume $Y$ is non-compact: equivalently, $|\Sigma|\neq N_{\R}$. We also want $Y$ to be non-singular: equivalently, each cone of $\Sigma$ is generated by part of a $\Z$-basis of the lattice $N$. We then call $Y$ a {\bf semiprojective toric manifold}.

For a picture, see \cref{Example toric intro}: there, $\mathbb{T}=(\C^*)^2$, the first $\C^*$-factor is the usual $\C^*$-action on $\C$; and the second $\C^*$-factor $S\subset \mathbb{T}$ is the toric $\C^*$-action on $\P^1$ lifted to the blow-up $Y$, which is non-contracting and it acts trivially on $X=\C$. Thus $T=\mathbb{T}/S$ is the contracting $\C^*$-action for $X=\C$.
\end{ex}

\begin{ex}[New from old]
A symplectic $\C^*$-manifold $Y$ gives rise to new symplectic $\C^*$-manifolds by considering $\C^*$-invariant properly embedded $I$-{\ph} submanifolds, or by taking blow-ups of compact subvarieties for which the $\C^*$-action lifts to the blow-up. In the presence of commuting $\C^*$-actions, we require those conditions (invariance, resp. lifting) to hold for each $\C^*$-action.
\end{ex}

\subsection{Equivariant projective morphisms}\label{Subsection Equivariant projective morphisms}
Continuing \cref{Example intro CSR have global map}, let $\pi:Y\to X$ denote any $\mathbb{T}$-equivariant projective morphism, where $Y$ is a non-singular variety and $X$ is affine. Replacing $X$ with the (closed) image of $\pi$, we may assume that $\pi$ is surjective.
The lattice of $1$-parameter subgroups, 
$$N:=\mathrm{Hom}(\C^*,\mathbb{T})\cong \Z^d,$$
determines $\C^*$-actions $\Fi_v$ on $Y$, labelled by $v\in N$. 

We call {\bf complete actions}, those  $\Fi_v$ for which convergence points $y_0:=\lim_{\C^*\ni t \to 0} t\cdot_{\Fi_v} y$ exist for all $y\in Y$; they determine a semigroup $N_0(Y)\subset N$. We call {\bf contracting actions} those complete actions whose \CC-fixed locus is compact; they determine a semigroup $N_+(Y)\subset N_0(Y) \subset N$.

For affine varieties $X\subset \C^m$ (and using the Euclidean/analytic topology)
we obtain $N_+(X)\subset N_0(X) \subset N$. We will determine these in terms of specific conditions on the grading of the coordinate ring by characters, as shown below. Namely, the $\mathbb{T}$-action determines a grading $\C[X]=\oplus_{w\in M} \C[X]_w$ by characters $M=\mathrm{Hom}(\mathbb{T},\C^*)\cong \Z^d$, and the conditions are expressed in terms of the natural pairing 
$$\langle  \cdot,\cdot  \rangle:N\times M := \mathrm{Hom}(\C^*,\mathbb{T})\times \mathrm{Hom}(\mathbb{T},\C^*)\to \mathrm{Hom}(\C^*,\C^*)\cong \Z.$$
We obtain various combinatorial descriptions of which actions are complete and which are contracting. The simplest one to state is the following, although it should be noted that $N_0(Y),N_+(Y)$ do not depend on the choice of generators $f_{w^j}$ in the statement.

\begin{lm}\label{Lemma description of Nplus and Nzero}
For $\pi: Y \to X$ as above: 
$N_+(Y)=N_+(X)$, $N_0(Y) = N_0(X)$, so we just write $N_+,N_0$.

More explicitly, let $f_{w^1},\ldots,f_{w^{\gens}}$ be non-constant $\mathbb{T}$-homogeneous generators for the coordinate ring $\C[X]=\oplus_{w\in M} \C[X]_w$ as a unital $\C$-algebra, so $f_{w^j}\in \C[X]_{w^j}$. Then:
\begin{equation*}
    \begin{split}
N_0  &= 
\{v\in N: \langle v,w^j \rangle \geq 0 \textrm{ for }j=1,\ldots,{\gens} \},
\\
N_+  & =
\{v\in N: \langle v,w^j \rangle > 0 \textrm{ for }j=1,\ldots,{\gens}\}.
\end{split}
\end{equation*}
We remark that if $N_+\neq \emptyset$, then all $w^j\neq 0$ as $\C[X]_0=\C$ (the constants), and that $\Fi_v$ contracts $X$ to the unique $\mathbb{T}$-fixed point in $X$ for all $v\in N_+$.
\end{lm}

\begin{rmk}[Contracting actions]\label{Remark contracting actions}
Let $Y$ be any non-compact symplectic manifold, with a choice of compatible almost complex structure. Let $\psi$ be any $\C^*$-action on $Y$, whose $S^1$-part is Hamiltonian with moment map $H_{\psi}$ say. As explained in \cite[Sec.3.1]{RZ1}, $\psi_{e^{2\pi t}}$ is the time $t\in \R$ flow of $\nabla H_{\psi}$; the fixed locus satisfies $\mathrm{Fix}(\psi)=\mathrm{Crit}(H_{\psi})$; and the convergence point $y_0$ above exists precisely if the $-\nabla H_{\psi}$ flow converges to $y_0$. An equivalent definition for $\psi$ to be {\bf contracting} is the existence of some compact subdomain $Y^{\mathrm{in}}\subset Y$ such that the $-\nabla H_{\psi}$ flow of any point $y\in Y$ will eventually land in $Y^{\mathrm{in}}$. Another equivalent condition, is that $H_{\psi}$ is bounded below and $\mathrm{Fix}(\psi)$ is compact.

The existence of the $\Psi$-map \eqref{Equation intro Psi} forces the action $\Fi$ of a symplectic $\C^*$-manifold to automatically be contracting. Thus we are especially interested in $N_+$, because contracting actions might yield a symplectic $\C^*$-structure on $Y$, with a typically new $\Psi$-map.
\end{rmk}

\begin{cor}
Any $v\in N_+$ induces on $(Y,\omega,\Fi_v)$ the structure of a symplectic $\C^*$-manifold. In particular, by \cref{Theorem Intro SH as loc of QH}, the canonical homomorphism is a surjection,
$$
c^*:QH^*(Y) \to SH^*(Y,\Fi_v) \cong QH^*(Y)_{Q_{\Fi_v}},
$$
which equals localisation at $Q_{\Fi_v}\in QH^{2\mu_v}(Y)$ where $\mu_v$ is the Maslov index of the $S^1$-action by $\Fi_v$.
\end{cor}

As the construction of the unital $\k$-algebra $SH^*(Y,\Fi)$ depends on $\Fi$ 
(cf.\;the first paragraph of \cref{Subsection Introduction A brief overview of the structural properties}),
a natural question is to determine how $SH^*(Y,\Fi_v)$ depends on $v\in N_+.$

\begin{thm}[Invariance]
For any $v,v'\in N_+$, there is a commutative diagram, 
$$
\begin{tikzcd}[column sep=0.6in]
& QH^*(Y) 
 \arrow[dl, "","c^*"']
 \arrow[dr, "","c^*"]
 &
\\
SH^*(Y,\Fi_v) 
\arrow[rr,"\cong"] 
& & SH^*(Y,\Fi_{v'})
\end{tikzcd}
$$
   where the lower isomorphism is induced by continuation maps.
\end{thm}

It follows that $c^*(Q_{\Fi_{v'}}) \in SH^{2\mu_v}(Y,\Fi_v)$ is an invertible element for all $v,v'\in N_+.$ Indeed, that also holds for $v'\in N_0$ (we are able to construct $Q_{v'}\in QH^{2\mu_{v'}}(Y)$ for all $v'\in N_0$). In $QH^*(Y)$, the $Q_v$ are often not invertible; however: if $v,-v\in N_0$, then $Q_{v},Q_{-v}\in QH^*(Y)$ are mutual inverses (essentially because the rotations $\Fi_v,\Fi_{-v}$ are mutual inverses). More is true: 

\begin{thm}\label{Prop Rw rotation elements}
In the unital ring $SH^*(Y,\Fi_v)$, there is an invertible {\bf rotation class} 
$$R_{w}\in SH^{2\mu_w}(Y,\Fi_v) \; \textrm{ for all }w\in N,\textrm{ such that }R_w=c^*(Q_w)\textrm{ for all }w\in N_0.$$
Moreover, $R_w^{-1}=R_{-w}$, and the product satisfies
$$
R_{w_1}\star \cdots \star R_{w_k} = T^{a(w_1,\cdots,w_k)} R_{w_1+\cdots + w_k}
$$
for an explicit function $a(w_1,\cdots,w_k)\in \Z$:
$$
a(w_1,\cdots,w_k) :=\min H_{w_1+\cdots+w_k}-\min H_{w_1}-\cdots-\min H_{w_k},
$$
where $H_w(y):=(\mu(y),w):Y \to \R$ is the Hamiltonian for $\Fi_w$ determined by the moment map $\mu$ of the Hamiltonian action by the compact subtorus $\mathbb{T}_{\R}\cong (S^1)^d$ of $\mathbb{T}\cong (\C^*)^d$. Similarly, also 
$$Q_{v_1}\star \cdots \star Q_{v_k} = T^{a(v_1,\cdots,v_k)} Q_{v_1+\cdots + v_k} \;\textrm{ for all }\; v_1,\ldots,v_k\in N_0.$$

Thus, compatibly on $N_0$ via $c^*:QH^*(Y)\to SH^*(Y,\Fi_v)$, we obtain group homomorphisms
$$
N_0 \to QH^*(Y), \; w \mapsto T^{\min H_w} Q_w \qquad \textrm{ and } \qquad
N \to SH^*(Y,\Fi_v), \; w \mapsto T^{\min H_w} R_w,
$$
for all $v\in N_+,$
which are generalisations of the Seidel representation to our setting.
\end{thm}
\begin{rmk}
That rotation class $R_{w}$ can be directly constructed at the Floer chain level, without first building a $Q$-element ($R_w:=\mathcal{R}_{\Fi_w}(1)$ in the notation from \cite[Thereom 7.23]{RZ1}). This idea goes back to the discussion in the original paper by the first author \cite[Sec.4.1 and 4.2]{R14} where rotation elements were first constructed, and where it was observed that ``Seidel $Q$-elements'' $Q_{w}\in QH^*(Y)$ are usually not invertible in the non-compact setting. Above, this occurs because $-w$ may not be $N_+$. The interested reader can also find a brief discussion of this phenomenon within the proof of \cref{Lemma borderline Cstar actions}. 
\end{rmk}

\begin{rmk}
In the {\bf Fano} setting (i.e.\;the monotone setting, meaning $c_1(Y)\in \R_{>0}[\omega]$ in $H^2(Y,\R)$), the function $a$ above is also determined by the $\Z$-grading, placing $T$ in grading $|T|=2$, namely 
$$
a(w_1,\cdots,w_k) :=\mu_{w_1+\cdots+w_k}- \mu_{w_1}-\cdots- \mu_{w_k},
$$
where $\mu_w$ is the Maslov index of the $S^1$-action $\Fi_w$ (which is easy to compute, see \cite[Sec.5.1]{RZ1}).
\end{rmk}

The next natural question is whether the filtrations \eqref{Equation introduction filtration} for $(Y,\omega,\Fi_v)$ are related, as we vary $v\in N_+$. In general, it is very difficult to relate filtrations associated to two different $\C^*$-actions, because the Hamiltonian Floer cohomologies are constructed using two completely different classes of Hamiltonian functions, with different growth conditions at infinity. In the non-compact setting, it is usually not possible to build Floer continuation maps that relate the two resulting classes of Floer cohomologies.   

Using the notation from \cref{Lemma description of Nplus and Nzero}, we introduce a {\bf period-rescaling constant}: for $v,v'\in N_+$,%
\footnote{We will show in the main text that for any $v,v'\in N_+$, $k(v,v')$ is a finite positive rational number.
Although $k(v,v')$ depends on the weights $w^j$ of the choice of generators $f_{w^j}$ in \cref{Lemma description of Nplus and Nzero},
\cref{Theorem intro commuting actions result} does not depend on the choice of generators (the generators are only used as a technical tool to define a certain map $\Upsilon$ that controls Floer theory at infinity).
So, in practice, the strength of that Proposition may improve for a judicious choice of generators.}
\begin{equation}\label{Equation constant kvvprime}
k(v,v'):=\min_{j=1,\ldots,{\gens}} \frac{\langle  v,w^j  \rangle}{\langle  v',w^j  \rangle} \in \Q_{>0}.
\end{equation}

\begin{thm}%
\label{Theorem intro commuting actions result}
Each $v\in N_+$ defines a filtration $\Fil^{p}_{v}$ of $QH^*(Y)$ by ideals, labelled by $p\in \R \cup \{\infty\}$, by applying \eqref{Equation introduction filtration} to $(Y,\omega,\Fi_v)$.
For all periods $p\in \R \cup \{\infty\},$ and all $v\in N_+$, these filtrations satisfy
\begin{equation}\label{Equation filtrations are included}
\sum_{v'\in N_+} \Fil^{k(v,v')\cdot p}_{v'} = \Fil^{p}_{v}. 
\end{equation}
\end{thm}
\begin{ex}[$A_2$-singularity]\label{Example running example of intro 4}
In \cref{Example running example of intro} call $A,B,C$ the three actions in $(a),(b),(c)$. Let $\mathbb{T}:=\C^*\times \C^*$, with $(v_1,v_2)\in N=\Z^2$ acting by $B^{v_1}\circ C^{v_2}$. 
So $B=\Fi_{(1,0)}$, $C=\Fi_{(0,1)}$,
$A=B\circ C=\Fi_{v}$ where $v:=(1,1)$.
We choose the generators\footnote{for example $w^1=(1,2)\in M=\mathrm{Hom}(\Z^2,\Z)$ means $B_t^*X=t^1X$, $C_t^*X=t^2X$.} $f_{(1,2)}=X$, $f_{(2,1)}=Y$, $f_{(1,1)}=Z$. 
Then $\langle v, w^j  \rangle$ yields the weights $3,3,2$ of the $A$-action, and $\langle (1,0), w^j  \rangle$ are the weights $1,2,1$ for $B$. So $k(v,(1,0))=\min \{\tfrac{3}{1},\tfrac{3}{2},\tfrac{2}{1}\}=\tfrac{3}{2}$. Similarly $k(v,(0,1))=\min \{\tfrac{3}{2},\tfrac{3}{1},\tfrac{2}{1}\}=\tfrac{3}{2}$.
The non-trivial claim in \eqref{Equation filtrations are included} is the $\subset $ inclusion, since the summand for $v'=v$ gives the right hand side as $k(v',v')=1$; the summands for $v'\neq v$ typically do not equal the right hand side. For example, abbreviating 
$$\FF_A^p:=\FF_{(1,1)}^p, \quad \FF_B^p:=\FF_{(1,0)}^p, \quad \FF_C^p:=\FF_{(0,1)}^p,\quad
\textrm{ thus }\;
\FF_B^{3p/2}+\FF_C^{3p/2}\subset \FF_A^p
\;\textrm{ by \eqref{Equation filtrations are included},}
$$
then by \cite[Examples 1.44 and 1.51]{RZ1},
we have: 
$$
\FF_A^{1/3}=\k p_1 \oplus \k p_2 = H^2(M),
\qquad
\FF_B^{1/2}=\k p_2,
\qquad
\FF_C^{1/2}=\k p_1.
$$ 
In that case, the latter two are enough to generate $\FF_A^{1/3}$. Similarly 
$\Fil_{B}^{1} + \Fil_{C}^{1} = H^{\geq 2}(M) = \Fil_{A}^{2/3}$. They are not enough to generate for $p=1:$
$$\Fil_{B}^{3/2} + \Fil_{C}^{3/2} = H^{\geq 2}(M)\subsetneq H^*(M) = \Fil_{A}^{1}.$$
\end{ex}

Although the filtration is quite computable in practice, particularly by the {\MBF} methods we developed in \cite{RZ2}, general statements are harder to come by. An exception holds for integer periods $p\in \N$ in view of \eqref{Equation cNplus maps intro} and the fact that $SH^*(Y,\Fi)\cong QH^*(Y)_{Q_{\Fi}}$ in \cref{Theorem Intro SH as loc of QH}:

\begin{cor}
Recall from \cref{Theorem Intro SH as loc of QH} the generalised $0$-eigenspace of $Q_{\llambda_v}$:
$$E_v:=E_0(Q_{\llambda_v})= \ker Q_{\Fi_v}^{n_v} = \ker (c^*: QH^*(Y) \to SH^*(Y,\Fi_v)),$$
where $n_v=\min\{n\in \N: Q_{\llambda_v}^{n}\star E_v=0\}=\min\{n\in \N: Q_{\llambda_v}^{n}\star QH^*(Y)=E_v\}$ 
is the order of nilpotency of the $Q_{\llambda_v}$-action on $E_v$, for $v\in N_+$. 
For any $v\in N_+$, the filtration $\FF_{v}:=\FF_{\Fi_v}^p$ of $QH^*(Y)$ from \eqref{Equation introduction filtration} determined by the symplectic $\C^*$-structure $(Y,\omega,\Fi_v)$ at integer periods is
$$
\Fil_{v}^p=Q_{v}^{p}\star E_v \qquad \textrm{ for }p\in \N.
$$
Thus: $
0=\Fil_{v}^0\subsetneq 
\
\Fil_{v}^1
\
\subsetneq 
\
\Fil_{v}^2
\
\subsetneq \cdots 
\subsetneq 
\
 \Fil_{v}^{k_v-1}
 \
 \subsetneq
 \Fil_{v}^{k_v}\subseteq \Fil^{\infty}_{v}=QH^*(Y)
$
is the sequence of ideals:
$$
0\subsetneq 
\
Q_{v}^{k_v-1}E_v
\
\subsetneq 
\
Q_{v}^{k_v-2}E_v
\
\subsetneq \cdots 
\subsetneq 
\
Q_{v}E_v
 \
 \subsetneq
 \Fil_{v}^{k_v}=E_v \subseteq \Fil^{\infty}_{v}=QH^*(Y).
$$
\end{cor}

\subsection{Semiprojective toric manifolds}\label{Subsection intro Semiprojective toric varieties}
\strut\\ \indent
We already encountered defined semiprojective toric manifolds in \cref{Example intro GIT} and \cref{Example def semiprojective toric mfds intro}. 

Another equivalent definition of those is: a non-compact\footnote{The compact case corresponds precisely to projective toric manifolds $Y$, equivalently $X=(\mathrm{point})$. We exclude these.
} toric manifold $Y$ for which the affinisation map $$\pi:Y\to \mathrm{Spec}(H^0(Y,\mathcal{O}_Y))$$ is projective, and $Y$ has at least one torus-fixed point. 

These non-singular toric varieties are quasiprojective, and admit a K\"{a}hler form $\omega$ which is invariant under the compact subtorus $\mathbb{T}_{\R}\cong (S^1)^n$ of the algebriac torus $\mathbb{T}\cong (\C^*)^n$ of the toric variety $Y$ (where $n=\dim_{\C}Y$, since $Y$ is the closure of an embedded dense open $\mathbb{T}$-orbit).
Such $Y$ have a well-defined non-compact convex moment polytope 
\begin{equation}
\label{Equation intro delta equation}
\Delta=\mathrm{image}(\mu)=\{x\in M_{\R}:\langle x,e_i\rangle \geq \lambda_i \}\subset \R^n,
\end{equation}
where $e_i$ are the minimal ray generators for the fan $\Sigma$, and $\Delta$ is the image of the moment map $\mu$ for $\mathbb{T}_{\mathbb{R}}$. The vertices of $\Delta$ correspond precisely to the torus fixed points via $\mu$. One can also define semiprojective toric manifolds starting from a precise class of such polyhedra $\Delta$.

\begin{ex}\label{Example toric intro}
In \cref{Subsection An explicit example}, we describe an explicit (non-compact) Fano semiprojective toric surface, $$\pi: \mathrm{Bl}_p(\C \times \P^1) \to \C,$$ namely: $Y$ is the blow up at $p=\{0\}\times \{[1:0]\}\in \C \times \P^1$, and $\pi$ is the projection. 
Holomorphic curves appear at infinity: $\pi^{-1}(c)\cong \P^1$ for each $c\neq 0\in \C$.
We illustrate the moment polytope $\Delta$ and fan $\Sigma$:
\begin{center}
\begin{picture}(0,0)%
\includegraphics{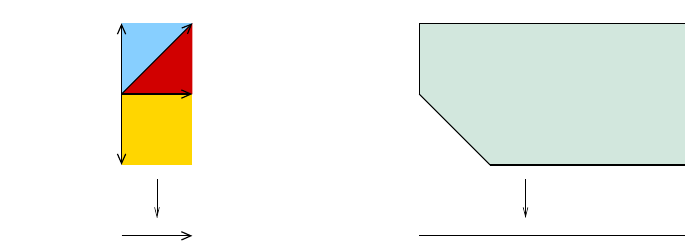}%
\end{picture}%
\setlength{\unitlength}{2486sp}%
\begingroup\makeatletter\ifx\SetFigFont\undefined%
\gdef\SetFigFont#1#2#3#4#5{%
  \reset@font\fontsize{#1}{#2pt}%
  \fontfamily{#3}\fontseries{#4}\fontshape{#5}%
  \selectfont}%
\fi\endgroup%
\begin{picture}(8722,3060)(-4424,-3286)
\put(2341,-2761){\makebox(0,0)[lb]{\smash{{\SetFigFont{10}{12.0}{\familydefault}{\mddefault}{\updefault}$\pi$}}}}
\put(2296,-1411){\makebox(0,0)[lb]{\smash{{\SetFigFont{10}{12.0}{\familydefault}{\mddefault}{\updefault}$\Delta$}}}}
\put(-2339,-2761){\makebox(0,0)[lb]{\smash{{\SetFigFont{10}{12.0}{\familydefault}{\mddefault}{\updefault}$\pi$}}}}
\put(-1889,-1456){\makebox(0,0)[lb]{\smash{{\SetFigFont{10}{12.0}{\familydefault}{\mddefault}{\updefault}$e_1=(1,0)$}}}}
\put(-3554,-1411){\makebox(0,0)[lb]{\smash{{\SetFigFont{10}{12.0}{\familydefault}{\mddefault}{\updefault}$\Sigma$}}}}
\put(-4409,-2266){\makebox(0,0)[lb]{\smash{{\SetFigFont{10}{12.0}{\familydefault}{\mddefault}{\updefault}$e_4=(0,-1)$}}}}
\put(-4229,-691){\makebox(0,0)[lb]{\smash{{\SetFigFont{10}{12.0}{\familydefault}{\mddefault}{\updefault}$e_2=(0,1)$}}}}
\put(-1889,-511){\makebox(0,0)[lb]{\smash{{\SetFigFont{10}{12.0}{\familydefault}{\mddefault}{\updefault}$e_3=(1,1)$}}}}
\put(3601,-421){\makebox(0,0)[lb]{\smash{{\SetFigFont{10}{12.0}{\familydefault}{\mddefault}{\updefault}$D_4$}}}}
\put(3601,-2581){\makebox(0,0)[lb]{\smash{{\SetFigFont{10}{12.0}{\familydefault}{\mddefault}{\updefault}$D_2$}}}}
\put( 46,-511){\makebox(0,0)[lb]{\smash{{\SetFigFont{10}{12.0}{\familydefault}{\mddefault}{\updefault}$(-1,1)$}}}}
\put( 46,-1636){\makebox(0,0)[lb]{\smash{{\SetFigFont{10}{12.0}{\familydefault}{\mddefault}{\updefault}$(-1,0)$}}}}
\put(946,-2536){\makebox(0,0)[lb]{\smash{{\SetFigFont{10}{12.0}{\familydefault}{\mddefault}{\updefault}$(0,-1)$}}}}
\put(991,-2041){\makebox(0,0)[lb]{\smash{{\SetFigFont{10}{12.0}{\familydefault}{\mddefault}{\updefault}$D_3$}}}}
\put(496,-1006){\makebox(0,0)[lb]{\smash{{\SetFigFont{10}{12.0}{\familydefault}{\mddefault}{\updefault}$D_1$}}}}
\end{picture}%

\end{center}
\end{ex}

For a toric variety with torus $\mathbb{T}\cong (\C^*)^n$, $1$-parameter subgroups of $\mathbb{T}$ are labelled by lattice points $v\in N\cong \mathrm{Hom}(\C^*,\mathbb{T})\cong \Z^n$.
We denote the union of the cones of the fan $\Sigma$ of $Y$ by
$$
|\Sigma|\subset N\otimes_{\Z}\R\cong \R^n.
$$

\begin{prop}
For semiprojective toric manifolds, the complete and the contracting actions are
$$
N_0:=N \cap |\Sigma|
\qquad
\textrm{ and }
\qquad
N_+:=N \cap \mathrm{Int}|\Sigma|.
$$
In particular, $N_+\neq \emptyset$, since $\dim |\Sigma| = n$. Therefore they are 
also symplectic $\C^*$-manifolds globally defined over $\C^m$, for any contracting $\C^*$-action $\Fi_v$, i.e.\;all $v\in N_+.$ 
\end{prop}
    
From now on, $(Y,\Fi_v)$ is any semiprojective toric manifold with a choice of $v\in N_+$.

Recall from \cite{RZ1} that the {\bf core} of a $\C^*$-action consists of all the points of $Y$ which \emph{converge} under the action of $\R_+\subset \C^*$ (if the $S^1$-part of the action is Hamiltonian with moment map $f$, then the core equals the union of the unstable manifolds of the $-\nabla f$ flow). The $\C^*$-action of a symplectic $\C^*$-manifold is always a path-connected compact subset, but it can be quite singular (unlike the smooth but often disconnected $\C^*$-fixed locus $\F=\mathrm{Crit}(f)\subset \mathrm{Core}(Y)$). In our current setup, and generally for projective equivariant morphisms, the core is the preimage of $0$ under the map $\Psi: Y \to \C^m$ that we construct, in particular the core is cut out by analytic equations. It follows that $Y$ deformation retracts to the core, in particular $H^*(Y)\cong H^*(\mathrm{Core}(Y))$.

\begin{ex}\label{Example non contracting action}
    In \cref{Example toric intro}, $N_0=\N\times \Z$ and $N_+=\N_{>0}\times \Z$, and $\mathrm{Core}(Y,\Fi_v)=D_1\cup D_3$ for $v\in N_+$ (which is compact and independent of $v$), whereas $\mathrm{Core}(Y,\Fi_{v})=D_1\cup D_2\cup D_3 \cup D_4$ is non-compact for $v=(0,\pm 1)\in N_0\setminus N_+$.
\end{ex}
\cref{Theorem intro commuting actions result} applies for all $v\in N_+=N \cap \mathrm{Int}|\Sigma|,$ so we deduce: 

\begin{cor}
    $SH^*(Y,\Fi_v)$ is independent of $v\in N_+$ up to continuation isomorphism, and it arises from $QH^*(Y)$ by localising at $Q_{\Fi_v}$, or equivalently by localising at all $Q_{\Fi_{v'}}$ for $v'\in N_+$ ($c^*Q_{\Fi_{v'}}=R_{v'}$ is automatically invertible in $SH^*(Y,\Fi_v)$).
\end{cor}

Recall the notation arising in the definition of $\Delta$ in \eqref{Equation intro delta equation}: the ray generators
$$e_1,\ldots,e_r\in N_0=N\cap |\Sigma| \subset N \cong \Z^n$$ of the fan $\Sigma$, and the parameters $\lambda_i\in \Z$ (which determine a toric K\"{a}hler form). 
Denote the toric divisors by $D_i = \mu^{-1}(\Delta_i)$ (the preimage of the $i$-th facet $\Delta_i :=  \{x\in \R^n:\langle x,e_i\rangle = \lambda_i \}$ of $\Delta$ via the moment map).
From now on, also abbreviate
$$x^v:=x_{j_1}^{v_1}x_{j_2}^{v_2}\cdots x_{j_n}^{v_n},$$ whenever $v=v_1e_{j_1}+\cdots + v_n e_{j_n} \in N\cap |\Sigma|$ in a basis $e_{j_1},\ldots,e_{j_n}$ of a maximal cone of $\Sigma$.

\begin{thm}\label{Prop intro semiproj toric}
If $Y$ is Fano or CY, we have the following explicit presentation, such that $Q_{\Fi_v}=x^v$:
$$
\k[x_1,\ldots,x_r]/\mathcal{J} \cong QH^*(Y),\; x_i \mapsto Q_{\Fi_{e_i}}.
$$
Given any $v\in N_+$, localising that isomorphism  at $Q_v=x^v$ yields
$$
\begin{array}{cccccccc}
QH^*(Y)[x^{\pm v}] & \;  \; 
\cong \;  \; & \k[x_1^{\pm 1},\ldots,x_r^{\pm 1}]/\mathcal{J} & 
\;  \; 
\cong \; \;   & SH^*(Y,\llambda_v) & 
\;  \; 
\cong \;  \;  & \mathrm{Jac}(W),\\
 Q_{\Fi_{e_i}} & 
\mapsto & x_i
 & 
\mapsto & R_{\Fi_{e_i}}= c^*Q_{\Fi_{e_i}} &  
\mapsto  & T^{-\lambda_i}z^{e_i},
\end{array} 
$$
where 
$\mathcal{J}$ is the ideal generated by the linear relations and quantum Stanley--Reisner relations (\cref{Subsection Presentation of quantum and symplectic cohomology of Fano semiprojective toric manifolds}).

When $Y$ is Fano, we have
$$
Q_{\Fi_{e_i}}=\mathrm{PD}[D_i]\in QH^2(Y),
$$
whereas if $Y$ is CY, there may be higher order $T$ correction terms:
$$
Q_{\Fi_{e_i}}=\mathrm{PD}[D_i]+T^{>0}\textrm{-terms} \in QH^2(Y).
$$
In particular, $\mathcal{O}(-1)\oplus \mathcal{O}(-1) \to \C P^1$ (a crepant resolution of the conifold $\{z_1 z_2 -z_3z_4=0\}\subset \C^4)$ is a CY semiprojective toric manifold for which such higher order corrections must arise.
\end{thm}

\begin{ex}\label{Example cohomology computation introduction for tori example}
    In \cref{Example toric intro}, we find that
\begin{align*}
H^*(Y;\Z) &\cong
\Z[x_1,x_2]/(x_1x_2,\ x_1^2,\ x_2^2),\\ 
QH^*(Y) &\cong 
\k[x_1,x_2]/(x_1x_2+Tx_1,\ x_1^2,\ x_2^2+Tx_1-T^2),\\ SH^*(Y,\Fi_v) &\cong QH^*(Y)[x^{\pm v}] 
\cong 
 QH^*(Y)[x_1^{\pm 1},x_2^{\pm 1}] = 0,
\end{align*}
for all $v\in N_+=\{(a,b)\in \Z^2:a>0\}\subset N=\Z^2$. 
An interesting feature of this example, is that the vanishing of symplectic cohomology occurs even though $Y$ is Fano (the vanishing occurs in general for CY symplectic $\C^*$-manifolds by \cref{Theorem 1 filtration}, but it was not previously observed in Fano examples).

The exterior edges $v_0=(0,\pm 1)\in N_0\setminus N_+$ give invertible classes $R_{v_0}=x^{v_0}=x_2^{\pm 1}$ in $SH^*(Y,\Fi_v)$. 

We do not yet know if there is a reasonable notion of symplectic cohomology $SH^*(Y,\Fi_{v_0})$ when $v_0\in N\cap \partial |\Sigma|=N_0\setminus N_+$, i.e.\;for complete but non-contracting actions.
We saw in general that $SH^*(Y,\Fi_v)\cong QH^*(Y)[Q_v^{-1}]$ is $v$-invariant for all $v\in N_+$. In the above example, at least one of those properties must fail because we show that the localisation of quantum cohomology no longer vanishes:
$$
QH^*(Y)[x^{\pm v_0}]\neq 0.
$$
\end{ex}

\begin{rmk} The interesting feature that localising at $x^v$ inverts all $x_i$ is due to \cref{Prop Rw rotation elements}. Outside of the Fano setup, Smith
\cite{smith2020quantum} gave a description of quantum cohomology for semiprojective toric manifolds, and it is in general a difficult question to determine which monomials of $QH^*(Y)$ are being inverted when passing to symplectic cohomology \cite[Rmk.1.8]{smith2020quantum}.
\end{rmk}

We can also compute the $S^1$-equivariant quantum cohomology $E^-QH^*(Y,\Fi_v)$ for $v\in N_+$. We recall our convention that we $u$-complete the usual $S^1$-equivariant quantum cohomology for the $S^1$-action $\Fi_v$, so $$E^-QH^*(Y,\Fi_v)=QH^*_{S^1}(Y,\Fi_v)\otimes_{\ku}\kuu$$ where $u$ is the formal equivariant parameter in degree $2$ generating $H^*(BS^1)=H^*(\C P^{\infty}) =\k[u]$. Thus $E^-QH^*(Y,\Fi_v)$ is a $\ku$-module.

Instead of the ideal $\mathcal{J}$, we now use the ideal 
$$
\mathcal{S} \subset \k[x_1,\ldots,x_r] 
$$
generated by the quantum Stanley-Reisner relations, but not all the linear relations: we only include the following amongst the generators of $\mathcal{S}$:
$$
\{ \sum \langle e_i,\xi\rangle x_i: \langle \xi,v\rangle =0\} \subset \Z[x_1,\ldots,x_r].
$$
We briefly motivate this: the linear relations would let $\xi$ run over a basis for $M=\mathrm{Hom}(\mathbb{T},\C^*)\cong \Z^n$; in the presentation of classical equivariant cohomology for $Y$ in \cite[Theorem 12.4.14 p.601]{cox2011toric}, $M\cong H^2(B\mathbb{T})\cong \Z^n$ acts on $H^*_{\mathbb{T}}(Y)$ by sending $y\in M$ to $-\sum \langle e_i,y\rangle \, x_i$ (and $x_i\mapsto \mathrm{PD}[D_i]$). We are not using the full torus $\mathbb{T}$, but just the $\C^*$-subgroup generated by $\Fi_v$, corresponding to $\Z v \subset N$, so we ``set to zero'' all the equivariant parameters normal to $\Z v$.

Our convention below is that $E^-SH^*(Y,\Fi_v)$, for $v\in N_+$, refers to the action using the free weight $(1,0)$ (see \cref{Definition free weight}), so in particular $E^-SH^*(Y,\Fi_v)$ is a unital $\k$-algebra. We constructed in \cite{RZ3} the equivariant rotation classes $EQ_{\Fi}$ analogous to the rotation classes $Q_{\Fi}$.

We remark that just like $QH^*(Y)=H^*(Y;\k)$ as a $\k$-vector space, $E^-QH^*(Y,\Fi_v)\cong H^*_{S^1}(Y,\Fi_v)[\![u]\!]$ as a $\ku$-module (cf.\;also \cref{Lemma equivariant formality}). Let $\k_{\geq 0},\k_{>0}$ inside $\k$ denote the $T^{\geq 0}$-terms and $T^{>0}$-terms in $\k$, respectively.

\begin{thm}\label{Prop intro semiproj toric}
If $Y$ is Fano or CY, we have the following explicit presentation for the equivariant quantum cohomology $E^-QH^*(Y,\Fi_v)$, such that $Q_{\Fi_v}=x^v$.
$$
\ku [x_1,\ldots,x_r]/\mathcal{S} \cong QH^*(Y),\; x_i \mapsto EQ_{\Fi_{e_i}}.
$$
Given any $v\in N_+$, localising that isomorphism at $EQ_v=x^v$ yields
$$
\begin{array}{cccccccc}
E^-QH^*(Y)[x^{\pm v}] & \;  \; 
\cong \;  \; & \k[x_1^{\pm 1},\ldots,x_r^{\pm 1}]/\mathcal{S} & 
\;  \; 
\cong \; \;   & E^-SH^*(Y,\llambda_v) & 
\;  \; 
\cong \;  \;  & \mathrm{Jac}(W),\\
 EQ_{\Fi_{e_i}} & 
\mapsto & x_i
 & 
\mapsto & ER_{\Fi_{e_i}}= c^*(EQ_{\Fi_{e_i}}) &  
\mapsto  & T^{-\lambda_i}z^{e_i}.
\end{array} 
$$
When $Y$ is Fano, we have
$$
EQ_{\Fi_{e_i}}\in \mathrm{PD}[D_i]+u\k_{\geq 0} \subset QH^2(Y),
$$
whereas if $Y$ is CY, there may be higher order $T$ correction terms:
$$
Q_{\Fi_{e_i}}\in \mathrm{PD}[D_i]+u \k_{\geq 0}+u^0\k_{>0}H^2(Y) \in  QH^2(Y),
$$
recalling that $H^2(Y)$ is spanned by $\mathrm{PD}[D_1],\ldots,\mathrm{PD}[D_r]$.

By \cref{Thm intro injectivity thm}, $u\k$-correction terms are to be expected in both cases, since the injective continuation maps for integer slopes correspond to iteratively multiplying by $EQ_{\Fi_v}$ like in \eqref{Equation cNplus maps intro}, cf.\;\cite{RZ3}.
\end{thm}

\begin{rmk}\label{Remark equiv classical presentation}
By replacing the quantum Stanley-Reisner relations by the classical Stanley-Reisner relations (by setting $T=0$), and replacing the coefficients $\k,\ku$ above by $\Q$ coefficients, those presentations of $QH^*(Y)$ and $E^-QH^*(Y,\Fi_v)$ are presentations of the ordinary cohomologies $H^*(Y)$ and $H^*_{S^1}(Y)$. If we do not put any of the linear relations in $\mathcal{S}$ above, then this is the presentation of the ordinary $\mathbb{T}$-equivariant cohomology. The presentation for $H^*(Y)$ is due to Hausel-Sturmfels \cite{HS02}, whereas the equivariant versions are proved in \cite[Theorem 12.4.14 p.601]{cox2011toric}. These do not use the Fano/CY assumption.
\end{rmk}

\noindent \textbf{Acknowledgements.} 
We thank Paul Seidel and Jack Smith 
for helpful conversations.

\section{Commuting, complete, and contracting $\C^*$-actions on affine varieties}\label{Section commuting complete and contracing actions}

\subsection{Complete and contracting $\C^*$-actions on affine varieties}

\begin{de}\label{Definition complete and contracting actions}
For a $\C^*$-action $\psi$ on any space $Y$, $\psi$ is a {\bf complete action} if convergence points $\lim_{\C^*\ni t\to 0}\psi_v(y)$ exist for all $y\in Y,$ and a {\bf contracting action} if $\psi$ is complete and the fixed locus $\mathrm{Fix}(\psi)$ is compact. For affine varieties $Y\subset \C^m$, this is meant in the Euclidean subspace topology.
\end{de}

Let $\psi$ be a $\C^*$-action on an affine variety $X$. Then the coordinate ring $\C[X]$ is\footnote{The action $\C^*\times X \to X$ defines a $\C$-algebra homomorphism $\C[X]\to \C[X\times \C^*]=\C[X][t,t^{-1}]$, $f\mapsto \sum_{n\in \Z} f_n t^n$ (finite sum), then $f=\oplus f_n\in \oplus_{n\in \Z}\C[X]_n=\C[X].$} $\Z$-graded, where $\C[X]_n$ are the functions $f$ of weight $n$: $\psi_t^*f = t^n f$ for all $t\in \C^*$. So the homogenous elements for this grading are ``$\C^*$-eigenvectors'', and $\C[X]_0$ contains the constant functions $\C$. Choose\footnote{Pick finitely many generators for the unital $\C$-algebra $\C[X]$, and then use their homogeneous parts.} homogeneous non-constant generators $f_{w^1},\ldots,f_{w^{\gens}}$ for the unital $\C$-algebra $\C[X]$, with weights $w^j \in \Z$ (there is a slight abuse of notation: the $w^j$ may not be distinct, but we consider the corresponding $f_{w^j}$ as distinct).

The generators $f_{w^j}$ determine an embedding 
$$X\hookrightarrow \C^{\gens}, \; x\mapsto (x_1,\ldots,x_m), \textrm{ where }x_j=f_{w^j}(x)\in \C,$$ 
and we call $x_j$ the {\bf coordinates} of $x\in X$, so $x$ corresponds to the maximal ideal $\mathfrak{m}_x = (f_{w^j}-x_j)$ generated by all $f_{w^j}-x_j \in \C[X]$. 

\begin{lm}\label{Lemma how you act on the coordinate ring}
A convergence point $\lim_{\C^*\ni t \to 0} \psi_t(x)$ exists precisely if $w^j \geq 0$ for all $j$ with $x_j\neq 0$, in which case the limit point has coordinates $\delta_j x_j$ where $\delta_j = 0$ for $w^j >0$, and $\delta_j=1$ for $w^j=0$.
In particular, $\lim_{\C^*\ni t \to 0} \psi_t(x)$ is the point $0\in \C^m$ precisely if $w^j > 0$ for all $j$ with $x_j\neq 0$. 
\end{lm}
\begin{proof}
$\mathfrak{m}_{\psi_t(x)}=\psi_t^*\mathfrak{m}_x = (t^{\langle v,w^j \rangle}f_{w^j}-x_j) = (f_{w^j}-t^{-\langle v,w^j \rangle}x_j)$, so $\psi_t(x)$ has coordinates $t^{-\langle v,w^j \rangle}x_j$.
\end{proof}

\begin{cor}\label{Cor description of complete and contracting actions}
$(\psi$ is complete$)\Leftrightarrow (w^j\geq 0$ for all $j)\Leftrightarrow (\C[X]_n=0$ for $n<0)$.
\\
For $\psi$ complete, $\mathrm{Fix}(\psi)\subset X$ is the affine subvariety for $\mathcal{I}(\mathrm{Fix}(\psi))=\oplus_{n>0}\C[X]_n$ so $\C[\mathrm{Fix}(\psi)]\cong \C[X]_0.$
\\
Assuming $X$ is connected (e.g.\;irreducible), all contracting actions contract $X$ to $0\in X\subset \C^m$, and
$$(\psi\textrm{ is contracting})\Leftrightarrow (\textrm{all }w^j> 0) \Leftrightarrow 
(\C[X]_n=0\textrm{ for }n<0,\textrm{ and }\C[X]_0=\C).$$
\end{cor}
\begin{proof}
As $f_{w^j}\neq 0$, some $x\in X$ has $x_j=f_{w^j}(x)\neq 0$. By \cref{Lemma how you act on the coordinate ring}, $\psi$ is complete if and only if all $w^j \geq 0$. Note that all monomials (except $1$) generated by the $f_{w^j}$ have weights which are $\N_{>0}$-linear combinations of the $w^j$, which justifies the second claim about complete actions.

By the last claim in \cref{Lemma how you act on the coordinate ring}, $\psi$ contracts $X$ to $0$ precisely if $w^j>0$ for all $j$.

Note that 
$0\in \mathrm{Fix}(\psi)$, and that points in $\mathrm{Fix}(\psi)$ can have non-zero coordinates $x_j$ only for $w^j=0$, by \cref{Lemma how you act on the coordinate ring}. For $\psi$ complete, it follows that $\mathrm{Fix}(\psi)\subset X$ is an affine subvariety determined by the ideal $D:=\oplus_{n> 0}\C[X]_n\subset \C[X]=\oplus_{n\geq 0}\C[X]_n$.
Thus, $\mathrm{Fix}(\psi)\subset X\subset \C^m$ is either a finite set of points ($\Leftrightarrow \C[X]_0$ is a finite dimensional complex vector space) or it is non-compact in the Euclidean topology (e.g.\;by Noether's normalisation theorem).

If $X$ is irreducible, then 
$\C[X]$ is an integral domain, so $\C[\mathrm{Fix}(\psi)]=\C[X]/D\cong \C[X]_0$ is an integral domain, so $\mathrm{Fix}(\psi)$ is irreducible and hence connected.
More generally ``$X$ is connected'' implies ``$\mathrm{Fix}(X)$ is connected'', as follows. Suppose $\mathrm{Fix}(\psi)$ is disconnected. Then $1=e_1+e_2\in \C[\mathrm{Fix}(\psi)]$ would be a sum of two orthogonal non-zero idempotents:\footnote{geometrically, the vanishing sets of $e_1$ and $e_2$ give a disjoint cover of $\mathrm{Fix}(\psi)$ by closed (and open) sets.} $e_i^2=e_i$, $e_1e_2=0$.  
Then two such idempotents also exist in $\C[X]_0 \cong \C[\mathrm{Fix}(\psi)]$, so $X$ is also disconnected.

For $X$ connected, the fixed locus of any complete $\psi$ is therefore either a point (and thus the point $0$) or it is non-compact; thus contracting $\psi$ have $\mathrm{Fix}(\psi)=\{0\}$, so $\C=\C[\mathrm{Fix}(\psi)]\cong \C[X]_0$. As $f_{w^j}$ is non-constant, the condition $\C[X]_0=\C$ implies that $w^j\neq 0$. The rest follows.
\end{proof}

\subsection{Commuting $\C^*$-actions}

Let $\Fi,\psi$ be $\C^*$-actions on any space.
They {\bf commute} if $\Fi_s\circ \psi_t = \psi_t \circ \Fi_s$ for all $(s,t)\in \C^*\times \C^*$. This ensures that $\Fi_s\circ \psi_t$ defines a $(\C^*)^2$-action, so we get a lattice of $\C^*$-actions: $(a,b)\in \Z^2$ corresponds to the action $\Fi^a \psi^b$, acting by $(\Fi^a\psi^b)_t=\Fi_t^a\circ \psi_t^b$.

\begin{lm}\label{Lemma commuting actions preserve fixed loci}
For $\psi,\Fi$ commuting $\C^*$-actions on any space $X$, $\psi_t$ preserves each connected component 
$\F_{\a}$ of $\F:=\mathrm{Fix}(\Fi)$. 
If $\Fi,\psi$ both contract $X$ to a point, then they contract $X$ to the same point.
\end{lm}
\begin{proof}
    For $y\in \F_\a$, $\psi_t(y)=\psi_t\circ \Fi_t(y)=\Fi_t(\psi_t(y)).$ So $\psi_t(y)\in \mathrm{Fix}(\Fi)$ for all $t\in \C^*$. Indeed $\psi_t(y)\in \F_\a$ since $\psi_1(y)=y\in \F_\a$ and $\C^*$ is connected. Thus $\psi_t(\F_\a)\subset \F_\a$ for all $t\in \C^*$. The last claim follows because $\psi$ preserves the point $\mathrm{Fix}(\Fi)$ but also contracts it within $X$ to the point $\mathrm{Fix}(\psi)$.
\end{proof}

More generally, $d$ commuting $\C^*$-actions on $X$ define an action of a torus $(\C^*)^d$. This gives rise to a lattice $N=\Z^n$ of $\C^*$-actions, and a grading $\C[X]=\oplus_{w\in \Z^d} \C[X]_w$ by weights $w\in \Z^d$. Explicitly, letting $\langle v,w\rangle := \sum v_i w_i$, the $\C^*$-action $\Fi_v$ for $v\in V$ acts on $f\in \C[X]_w$ by 
\begin{equation}\label{Equation action of Cstar action on weight w function}
    (\Fi_v(t))^*f=t^{\langle v,w\rangle}f.
\end{equation}
\subsection{Torus actions}\label{Subsection Torus actions on affine variety}
Now let $\mathbb{T}\cong (\C^*)^d$ be any $d$-dimensional algebraic torus acting on $X$. One-parameter subgroups are labelled by the lattice of cocharacters, $N$, and the coordinate ring is graded by characters $w\in M:=N^{\vee}$:
$$N=\mathrm{Hom}(\C^*,\mathbb{T}), \;\;\textrm{ and }\;\;\C[X]=\bigoplus_{w\in M} \C[X]_{w} \textrm{ where }M:=\mathrm{Hom}(\mathbb{T},\C^*).$$ 
Define $\langle \cdot,\cdot \rangle$ to be the natural (non-degenerate) pairing induced by composition of cocharacters and characters,
\begin{equation}\label{Equation pairing characters}
\langle \cdot,\cdot \rangle: N\times M=\mathrm{Hom}(\C^*,\mathbb{T}) \times \mathrm{Hom}(\mathbb{T},\C^*) \longrightarrow \mathrm{Hom}(\C^*,\C^*)\cong \Z.
\end{equation}
For $v\in \Z^d$ the corresponding $\C^*$-action $\psi$ acts by the same formula \eqref{Equation action of Cstar action on weight w function}. Therefore the $\Z$-grading that $\Fi_v$ induces on $\C[X]$ has weight $n\in \Z$ part given by
\begin{equation}\label{Equation decomposition for Fiv}
\C[X]_n = \bigoplus_{w\in M: \langle v,w\rangle = n} \C[X]_w.
\end{equation}
\begin{cor}\label{Corollary descr of complete contr in aff case for fiv}
$(\Fi_v$ is complete$)\Leftrightarrow (\C[X]_w=0$ for all $w$ with $\langle v,w \rangle<0).$\\
Assuming $X$ is connected, 
$(\Fi_v$ is contracting$)\Leftrightarrow  $ in addition to the above: $\C[X]_w=\C$ if $\langle v,w \rangle=0$ (more precisely: $\C[X]_0=\C$, and $\C[X]_w=0$ for $w\neq 0$ with $\langle v,w \rangle =0$).
\end{cor}
\begin{proof}
This is immediate from \cref{Cor description of complete and contracting actions} and \eqref{Equation decomposition for Fiv}.
\end{proof}
Choose homogeneous non-constant generators $f_{w^1},\ldots,f_{w^{\gens}}$ for the unital $\C$-algebra $\C[X]$, with weights $w^j \in M$.
The $f_{w^j}$ define an embeding $X \hookrightarrow \C^m$, and the actions satisfy:
\begin{equation}\label{Equation weight of the action of fwj}
(\psi)^*(f_{w^j})=t^{\langle v,w^j \rangle} f_{w^j} \quad\;\textrm{ for all }\;t\in \C^*.
\end{equation}
\begin{cor}\label{Cor convergence points in affine case}
    A convergence point $\lim_{\C^*\ni t \to 0} (\Fi_v)_t(x)$ exists precisely if $\langle v,w^j \rangle \geq 0$ for all $j$ with $x_j\neq 0$, in which case the limit point has coordinates $\delta_j x_j$ where $\delta_j = 0$ for $\langle v,w^j \rangle >0$, and $\delta_j=1$ for $\langle v,w^j \rangle =0$.
    In particular, $\lim_{\C^*\ni t \to 0} (\Fi_v)_t(x)=0$ precisely if $\langle v,w^j \rangle > 0$ for all $j$ with $x_j\neq 0$.
\end{cor}
\begin{proof}
This is immediate by
\cref{Lemma how you act on the coordinate ring}
and \eqref{Equation weight of the action of fwj}.
\end{proof}

\begin{cor}\label{Cor description of fixed locus}
    The $\mathbb{T}$-fixed locus $\{0\}\subset \mathrm{Fix}(\mathbb{T})\subset X$ is the affine subvariety defined by the ideal in $\C[X]$ generated by $\{f_{w^j}: w^j\neq 0\}$, meaning $\mathrm{Fix}(\mathbb{T}) = \{ x\in X: x_j:=f_{w^j}(x)=0 \textrm{ for all }w^j\neq 0\}$.
    
    In particular, $\C[X]_0=\C \Longleftrightarrow$ there is a unique $\mathbb{T}$-fixed point (namely $0\in X\subset \C^m$).
\end{cor}
\begin{proof}
If $f_{w^j}\in \C[X]_w$ were non-zero at $x\in \mathrm{Fix}(\mathbb{T})$, then 
\eqref{Equation weight of the action of fwj} would force $\langle v,w^j\rangle=0$ for all $v\in N$ (otherwise the $j$-th coordinate $x_j$ is not $\mathbb{T}$-invariant).
As \eqref{Equation pairing characters} is non-degenerate, $w^j=0$.
\end{proof}

\begin{de}
    Let $N_0(X)=\{v\in N: \Fi_v\textrm{ is complete}\}$ and $N_+(X)=\{v\in N: \Fi_v \textrm{ is contracting}\}$, so $N_+(X)\subset N_0(X)\subset N$ (also, $0\in N_0(X)\setminus N_+(X)$ unless $X\subset \C^m$ is compact: a finite union of points).
\end{de}

We define the ``{\bf support}'' of the $M$-grading on $\C[X]$ by
\begin{equation}\label{Equation definition of support P}
P:= \{w\in M: \C[X]_w\neq 0\}.
\end{equation}
Observe that $P\subset \{0\}\cup  \mathrm{span}_{\N}\{w^1,\ldots,w^N\}$, and this is an equality when $X$ is irreducible (i.e.\;$\C[X]$ is an integral domain).
\begin{ex}
Suppose $\mathbb{T}=(\C^*)^d$ is generated by $d$ contracting $\C^*$-actions. Then: $\mathbb{T}$ contracts $X$ to $0$ (assuming $X$ is connected); $P \subset \{0\}\cup (\N\setminus \{0\})^d$; and $\N^d\setminus\{0\}\subset N_+(X)$ (which could be strict).
\end{ex}

\begin{cor}\label{Subsection Observations about commuting Cstar actions on affine varieties}
For any affine toric variety $X$ with a $\mathbb{T}$-action,
\begin{equation}\label{Equation positivity of weights torus not strict}
\begin{split}
N_0(X)  &= 
\{v\in N: \langle v,w^j \rangle \geq 0 \textrm{ for }j=1,\ldots,{\gens} \}
\\
&
= 
\{v\in N: \langle v,w \rangle \geq 0 \textrm{ for all } w\neq 0 \in P \}.
\end{split}
\end{equation}
For $v\in N_0(X)$, $\mathrm{Fix}(\Fi_v)\subset X$ is the affine subvariety for the ideal $\mathcal{I}(\mathrm{Fix}(\psi))=\oplus \{\C[X]_w: \langle v,w \rangle >0\}$, with coordinate ring $\C[\mathrm{Fix}(\psi)]\cong \oplus \{\C[X]_w: \langle v,w \rangle =0\}.$

Assuming $X$ is connected, 
\begin{equation}\label{Equation NplusX}
N_+(X)  =
\{v\in N: \langle v,w^j \rangle > 0 \textrm{ for }j=1,\ldots,{\gens}\}.
\end{equation}
If $N_+(X)\neq \emptyset$, then $\C[X]_0=\C$ and the actions for $N_+(X)$ will contract $X$ to $0\in X\subset \C^m$.

Assuming $X$ is connected, and $\C[X]_0=\C$ $(\Leftrightarrow \mathrm{Fix}(\T)=\{0\}$ by \cref{Cor description of fixed locus}$)$, 
\begin{equation}\label{Equation NplusX 2}
N_+(X)  =
\{v\in N: \langle v,w \rangle > 0 \textrm{ for }w\neq 0\in P 
\}.
\end{equation}
\end{cor}
\begin{proof}
The claim mostly follows from \cref{Corollary descr of complete contr in aff case for fiv}.
Observe that the monomial $\prod f_{w^j}^{n_j}$ for $n_j\in \N$ lies in $\C[X]_{w}$ for $w=\sum n_j w^j$, and these monomials generate $\C[X]$. 
If any of the inequalities $\langle v,w^j \rangle \geq 0$ were strict for some $n_j\neq 0$, then automatically $\langle v,w \rangle>0$. Thus \eqref{Equation NplusX} implies \eqref{Equation NplusX 2}.
Conversely, the condition in \eqref{Equation NplusX 2} implies the condition \eqref{Equation NplusX} for all non-zero $w^j$; and the assumption $\C[X]_0=\C$ implies that all $w^j\neq 0$ since the $f_{w^j}$ are non-constant.
\end{proof}

\begin{rmk}
If $N_+(X)\neq \emptyset$, there is a maximal $1\leq k\leq d$ such that, after an automorphism of the torus $\mathbb{T}\cong (\C^*)^d$, the first $k$ factors of $(\C^*)^d$ are contracting $\C^*$-actions and this $k$-dimensional subtorus $\mathbb{T}'\subset \mathbb{T}$ is maximal with respect to inclusion.
However, such a subtorus is typically not unique.
\end{rmk}

\subsection{Further technical properties of $N_0(X)$ and $N_+(X)$}
\label{Rmk further technical properties of Nplus X}
Simplify notation to $\mathbb{T}=(\C^*)^d$, $N=\Z^d$, $M=\Z^d$, by choosing an isomorphism $\mathbb{T}\cong (\C^*)^d$. The cone generated by $P$ is
$$\sigma:=\mathrm{span}_{\R_{\geq 0}}\{w\in \Z^d: \C[X]_w\neq 0\} =
\mathrm{span}_{\R_{\geq 0}}\{w^1,\ldots,w^{\gens}\}\subset \R^d.$$ 
In particular, $\sigma\cap \Z^d$ contains $P$ but may be larger.
Using the terminology from \cite[p.4]{fulton1993introduction}, 
$\sigma$ is a convex rational polyhedral cone. If $N_+(X)=\emptyset$ then\footnote{If both $w,-w$ lie in $P$ then $N_+(X)=\emptyset$ by \eqref{Equation NplusX}.} it is strongly convex:
``strong'' refers to $\sigma$ not containing a line through the origin, equivalently $\{0\}$ is a face of $\sigma$.
Its dual cone is
$$N_0(X)_{\R}:=\sigma^{\vee}=\{a\in \R^d: \langle a,w \rangle \geq 0 \textrm{ for }w\in \sigma \}
= \{a\in \R^d: \langle a,w^j \rangle \geq 0 \textrm{ for }j=1,\ldots,{\gens} \}.
$$
This is a convex rational polyhedral cone, so its integral points,
\begin{equation}\label{Definition N0 in coordinate case}
N_0(X) =N_0(X)_{\R} \cap \Z^d =\{ v\in \Z^d: \langle v,w^j\rangle \geq 0 \textrm{ for }j=1,\ldots,{\gens}\},
\end{equation}
define a finitely generated semigroup \cite[p.12] {fulton1993introduction}, such that $N_+(X)\subset N_0(X)$. 

The above ``strong'' condition for $\sigma$ is equivalent to the cone $\sigma^{\vee}\subset \R^d$ being (top) $d$-dimensional, in particular $\sigma^{\vee}$ then has non-empty interior. More generally, there is an order-reversing correspondence between faces of $\sigma$ and faces of $\sigma^{\vee}$ of complementary dimension \cite[p.12]{fulton1993introduction}, in particular the cone $\sigma^{\vee}$ has dimension $d-k$ where $k$ is the dimension of the vector space $\sigma\cap (-\sigma)$ (which is the smallest face of $\sigma$, and it is the largest vector subspace contained in $\sigma$).

The cone $N_0(X)_{\R}$ may also not have the ``strong'' condition, because the vector subspace 
$$W=N_0(X)_{\R}\cap (-N_0(X)_{\R})=\{a\in \R^d:\langle a,w \rangle=0 \textrm{ for all }w\in \sigma \}$$
may be non-trivial. Note: $W\cap \Z^d$ determines the (unique) largest algebraic subtorus of $(\C^*)^d$ which acts trivially on $X$. Via the vector space quotient $p:\R^d\to \R^d/W$, the image $N_0(X)_{\R}/W$ of $N_0(X)_{\R}$ becomes a strongly convex rational polyhedral cone whose preimage via $p$ is $N_0(X)_{\R}$. For the quotient, we use the lattice $\Z^d/(\Z^d\cap W).$
For any strongly convex rational polyhedral cone, there is a canonical choice of finitely many ``minimal generators'' $e_i$: the first non-zero integral points along the (ray) edges of the cone (e.g.\;see \cite[p.29, Lemma 1.2.15]{cox2011toric}). Even if $N_0(X)_{\R}$ is not ``strong'', any lifts $\widetilde{e}_i\in \Z^d$ of $e_i$ via the quotient $\Z^d \to \Z^d/(\Z^d \cap W)$ will ensure $N_0(X)_{\R} = W+\mathrm{span}_{\R_{\geq 0}}(\widetilde{e}_i)$.

Since $\langle\cdot,\cdot \rangle$ is integer-valued on $\Z^d$, and assuming $X$ is connected, \eqref{Equation NplusX} gives:
\begin{equation}\label{Equation Nplus as a polyhedron}
N_+(X) = \{ v\in \Z^d: \langle v,w^j\rangle \geq 1 \textrm{ for }j=1,\ldots,{\gens}\},
\end{equation}
i.e.\;they are the integral points of the polyhedron
$$
\{ a\in \R^d: \langle a,w^j\rangle \geq 1 \textrm{ for }j=1,\ldots,N\}=Q+N_0(X)_{\R},
$$
where $Q$ is a (compact) polytope \cite[p.318]{cox2011toric}, indeed $N_0(X)_{\R}$ is called the ``recession cone'' of the polyhedron. The polyhedron has finitely many vertices, and there are no vertices precisely when $W\neq \{0\}$ \cite[Lemma 7.1.1]{cox2011toric}. When there are vertices, one can choose $Q=(\textrm{convex hull of vertices})$.

\section{Commuting $\C^*$-actions on symplectic $\C^*$-manifolds}\label{Section commuting complete and contracing actions}
\subsection{Symplectic $\C^*$-actions and the Core}

\begin{de}
Let $Y$ be a symplectic $\C^*$-manifold with a choice of compatible almost complex structure $I$.
 A {\bf symplectic $\C^*$-action} $\psi$ is a {\ph} $\C^*$-action on $Y$ whose $S^1$-part is Hamiltonian. We will denote the moment map of the $S^1$-action by $H_{\psi}.$
\end{de}

We briefly recall some properties from \cite[Sec.3.1]{RZ1}, about such actions $\psi$.

For $s\in \R$, $\psi_{e^{2\pi s}}$ and $\psi_{e^{2\pi i s}}$ equal the time $s$ flow of the vector fields $\nabla H_{\psi}$ and $X_{H_{\psi}}=I\nabla H_{\psi}$.

Now suppose $\psi$ is also contracting (cf.\;\cref{Remark contracting actions}).
The $\C^*$-fixed locus satisfies 
$$\mathrm{Fix}(\psi)=\mathrm{Crit}(H_{\psi})=\left\{\lim_{\C^*\ni t\to 0} \psi_t(y): y\in Y \right\},$$
and it is compact. Also, $\mathrm{Fix}(\psi)$ is a smooth submanifold, and a subset of the (typically singular) {\bf core}:
$$
\mathrm{Core}(Y,\psi)=\left\{y\in Y: \lim_{\C^*\ni t \to \infty} \psi_t(y) \textrm{ exists in }Y\right\}.
$$
The core is compact and path-connected.

\begin{lm}\label{Lemma commuting actions preserve core}
Let $\Fi,\psi$ be commuting symplectic $\C^*$-actions. 
Suppose $\Fi$ is contracting. Then $\psi_t$ preserves $\Core(Y,\Fi)$ as well as each connected component of $\mathrm{Fix}(\Fi)$, and 
$$\Core(Y,\Fi)\subset \Core(Y,\psi).$$
\end{lm}
\begin{proof}
To show $\psi_t(\mathrm{Core}(Y,\Fi))\subset \mathrm{Core}(Y,\Fi)$, we use a similar proof as the one for the fixed locus in \cref{Lemma commuting actions preserve fixed loci}.
 Let $y\in \mathrm{Core}(Y,\Fi)$, so $\lim_{s\to \infty} \Fi_s (y)=y_{\infty}$ exists. Then $\Fi_s(\psi_t(y))=\psi_t(\Fi_s(y))\to \psi_t(y_{\infty})$ as $s\to \infty$, so $\psi_t(y)\in \mathrm{Core}(Y,\Fi)$.

Thus $\psi_t: \mathrm{Core}(Y,\Fi)\to \mathrm{Core}(Y,\Fi)$. As $\Fi$ is contracting, its core $\mathrm{Core}(Y,\Fi)$ is compact. Recall $\psi_{e^{2\pi s}}$ is the time $s\in \R$ flow of $\nabla H_{\psi}$. The forward flowline of a gradient flow in a compact set will converge to a critical point (in $\infty$ time). So for $y\in \mathrm{Core}(Y,\Fi),$
$\psi_{e^{2\pi s}}(y)\to y_{\infty}\in \mathrm{Fix}(\psi)\cap \mathrm{Core}(Y,\Fi)$ as $\R\ni s\to \infty$, since this is the forward flow for $\nabla H_{\psi}$. 
By applying the $S^1$-action, $\psi_{e^{2\pi s}}(\psi_{e^{i\theta}}(y))=\psi_{e^{i\theta}}(\psi_{e^{2\pi s}}(y))\to \psi_{e^{i\theta}}(y_{\infty})=y_{\infty}$. So $y\in \mathrm{Core}(Y,\psi)$. Thus $\mathrm{Core}(Y,\Fi)\subset \mathrm{Core}(Y,\psi)$.
\end{proof}

\begin{cor}\label{Lemma commuting contracting actions description of fix and core}
Let $\Fi,\psi$ be commuting contracting symplectic $\C^*$-actions. Then 
$$\Core(Y,\Fi)=\Core(Y,\psi).$$
\end{cor}
\begin{proof}
This follows by symmetry from \cref{Lemma commuting actions preserve core}.
\end{proof}

\begin{cor}\label{Cor contracting actions compose well}
Let $\Fi,\psi$ be commuting contracting symplectic $\C^*$-actions. Then $\Fi\psi$ is also a contracting symplectic $\C^*$-action (with $\mathrm{Core}(Y,\Fi \psi)= \mathrm{Core}(Y,\Fi)$).
\end{cor}
\begin{proof}
The only non-trivial claim, is that $\Fi\psi$ is contracting (the statement about cores then follows by \cref{Lemma commuting contracting actions description of fix and core}, since $\Fi\psi$ and $\Fi$ commute).
By \cref{Remark contracting actions}, it is enough to find a compact set $C\subset Y$ such that for any $y\in Y$, $\Fi_t(\psi_t(y))$ lies in $C$ for small enough $t\in \C^*$. Let $C$ be any compact neighbourhood of $\mathrm{Core}(Y,\Fi)=\mathrm{Core}(Y,\psi)$ (using \cref{Lemma commuting contracting actions description of fix and core}). Since 
$$\lim_{t\to 0}\psi_t(y)\in \mathrm{Fix}(\psi)\subset \mathrm{Core}(Y,\psi) \subset \mathrm{Int}(C),$$ 
$\psi_t(y)\in C$ for small $t$. Similarly, any $z\in C$ has $\lim_{t\to 0}\Fi_t(z) \in \mathrm{Fix}(\Fi)\subset \mathrm{Int}(C)$, so $\Fi_t(z)\in C$ for small $t$ and, since $C$ is compact, we can minimise over $z\in C$ so that $\Fi_t(C)\subset C$ for sufficiently small $t\in \C^*$. Thus $\Fi_t(\psi_t(y))\in \Fi_t(C)\subset C$ for sufficiently small $t\in \C^*$, as required.
\end{proof}

\subsection{The semigroup of contracting actions $N_+(Y)$ of an algebraic torus action}
\label{Subsection semigp contracting actions for sympl Cstar mfd}

\begin{de}\label{Definition symplectic T action}
Let $\mathbb{T}\cong (\C^*)^d$ be an algebraic torus.
 A {\bf symplectic $\mathbb{T}$-action} $\Fi$ is a {\ph} action of $\mathbb{T}$ on $Y$, such that its maximal compact subtorus $\mathbb{T}_{\R}\cong (S^1)^d$ acts in a Hamiltonian way. We will denote the moment map by $\mu: Y \to \mathrm{Lie}(\mathbb{T}_{\R})^*\cong \R^d$.
\end{de}

\begin{rmk}
We do not necessarily need to choose an isomorphism $\mathbb{T}\cong (\C^*)^d$, but we indicate this and other induced isomorphisms (such as, on the associated lattice $N\cong \Z^d$ of 1-parameter subgroups) to help navigate the notation.
Note that giving $d$ commuting symplectic $\C^*$-actions on $Y$ is the same as giving a $\T$-action for $\mathbb{T}=(\C^*)^d$, by considering each $\C^*$-factor.
\end{rmk}

Recall the $1$-parameter subgroups are parametrised by the lattice
$$
N:= \mathrm{Hom}(\C^*,\mathbb{T})\cong \Z^d.
$$
We write $\Fi_v$ for the $\C^*$-action corresponding to $v\in N,$ and using \cref{Definition complete and contracting actions} we define: 
\begin{equation}\label{Equation contracting and complete for sympl Cstar mfd}
N_0(Y):=\{v: \Fi_v
\textrm{ is a complete action}\}
\subset
N_+(Y)=\{v: \Fi_v
\textrm{ is a contracting action}\}
\subset N.
\end{equation}
\begin{cor}\label{Corollary core is same for all v in Nplus}
$N_+(Y)\subset N$ is a semigroup under addition, and the core is independent of $v\in N_+,$
$$\mathrm{Core}(Y):=\mathrm{Core}(Y,\Fi_v)\subset Y.$$ 
\end{cor}
\begin{proof}
This follows by \cref{Cor contracting actions compose well}.
\end{proof}

\section{Equivariant projective morphisms}\label{Section Equivariant projective morphisms}
\subsection{Setting up notation and assumptions}\label{Subsection Setting up notation and assumptions}
Let $Y$ be a non-singular variety with at $\mathbb{T}$-action, admitting a $\mathbb{T}$-equivariant projective morphism to an affine variety $X$,
$$
\pi:Y\to X.
$$
We later explain that $Y$ is quasi-projective via a $\mathbb{T}$-equivariant embedding $Y \hookrightarrow \C^m \times \C P^n$ compatibly via $\pi$ with an embedding $X\hookrightarrow \C^m$ as in \cref{Subsection Torus actions on affine variety}. Thus one can view $Y$ as a complex manifold, and $\pi$ as an analytic map ($X$ is typically singular), and we succinctly refer to this point of view by saying ``Euclidean topology'', as opposed to the Zariski topology. We sometimes write ``$X\subset \C^m$'' when talking about topological properties of $X$: this is to indicate that we mean the Euclidean/analytic topology. 

\begin{lm}
$\pi$ is a surjective $\mathbb{T}$-equivariant projective morphism onto an affine subvariety of $X$.
\end{lm}
\begin{proof}
Since $\pi$ is projective morphism, it is also a proper morphism, and therefore it is a closed morphism (indeed it is universally closed). The image of $\pi$ in $X$ is therefore closed (being the image of the closed set $Y$), and thus it is an affine subvariety.
\end{proof}

It follows that after replacing $X$ by an affine subvariety, we may assume that $\pi$ is surjective.
By studying one connected component at a time, we may assume that $Y$ is connected (equivalently, $Y$ is irreducible, since $Y$ is non-singular).
Therefore, $X=\pi(Y)$ is also connected.

We are only interested in the case when $Y$ is non-compact (in the Euclidean topology), which is equivalent to the condition that $X=\pi(Y)$ is not a point (recall by Noether normalisation that the only compact subsets of affine varieties $X\subset \C^m$ are finite unions of points).

Since $\pi$ is projective, the fibres are projective (but possibly singular) varieties. 

\begin{lm}
For any subtorus $\mathbb{T'}\subset \mathbb{T}$ (e.g.\;a one-parameter subgroup), its fixed loci are related by
\begin{equation}\label{Equation fixed loci relation}
\pi(\mathrm{Fix}_{Y}(\mathbb{T'}))= \mathrm{Fix}_{X}(\mathbb{T'}).
\end{equation}
\end{lm}
\begin{proof}
The inclusion $\subset$ is obvious. For the other direction: for any $\mathbb{T'}$-fixed point $x\in X$, the $\mathbb{T}'$-action restricts to a $\mathbb{T}'$-action on $\pi^{-1}(x)$ and thus it has a fixed point since $\pi^{-1}(x)$ is a projective variety.
\end{proof}

As $\pi$ is projective, it is also proper, in both the Zariski and the Euclidean topologies (in the latter, proper means preimages of compact sets are compact).

Recall the lattice $N\cong \Z^d$ of 1-parameter subgroups of $\mathbb{T}\cong (\C^*)^d$,  $N_0(X)\subset N_+(X)\subset N$, and $N_0(Y)\subset N_+(Y)\subset N$, from \cref{Subsection Torus actions on affine variety} and \cref{Subsection semigp contracting actions for sympl Cstar mfd}. We abusively write $\Fi_v$ for the $\C^*$-action corresponding to $v\in N$, both in $Y$ and in $X$, e.g.\;equivariance means $\pi \circ \Fi_v = \Fi_v \circ \pi$ for all $v\in N$.

We will prove later that those semigroups inside $N$ agree in $X$ and $Y$:
$$
N_+: = N_+(Y) = N_+(X) \;\;\textrm{ and }\;\; N_0:=N_0(Y)=N_0(X)
$$
To give $Y$ the structure of a symplectic $\C^*$-manifold we need to have at least one contracting $\C^*$-action (\cref{Remark contracting actions}). So the examples relevant to us will involve assuming that $N_+\neq \emptyset$, although we will not need this assumption until \cref{Subsection equiv proj res is symp cstar mfd}.

\subsection{The $\Upsilon$ map}\label{Subsection Upsilon map}

Pick non-constant homogeneous generators $f_{w^1},\ldots,f_{w^{\gens}}$ for the unital $\C$-algebra $\C[X]$ (so $\mathbb{T}$-eigenvectors), where $f_{w^j}\in \C[X]_{w^j}$. The weights $w^j\in M:=N^{\vee}$ are possibly not distinct. (Later, when we assume $N_+\neq \emptyset$, we will have $w^j\neq 0$ by \cref{Cor CX0 is C in equiv proj section}). 

\begin{lm}
This choice of homogeneous generators determines a {\ph} proper map
\begin{equation}\label{Definition of Upsilon}
\Upsilon: Y \to \C^{\gens}, \;\;\; \Upsilon=(\Upsilon_1,\ldots,\Upsilon_{\gens}):=(f_{w^1},\ldots,f_{w^{\gens}}) \circ \pi.
\end{equation}
It is $\mathbb{T}$-equivariant if we declare that the $\mathbb{T}$-action on the $j$-th factor of $\C^{\gens}$ is prescribed by $w^j\in N^{\vee}$.
\end{lm}
\begin{proof}
By construction, $(f_{w^1},\ldots,f_{w^{\gens}}): X \to \C^{\gens}$ is a proper analytic embedding.
As $\pi$ is a projective morphism, it is proper.
The claim follows, using that $\pi$ is a $\mathbb{T}$-equivariant analytic map.
\end{proof}

\subsection{The symplectic form}

\begin{lm}\label{Lemma equiv proj morph sympl form}
    $Y$ is a quasi-projective variety admitting a $\mathbb{T}$-equivariant K\"{a}hler form $\omega$, such that the maximal compact real torus $\mathbb{T}_{\R}\subset \mathbb{T}$ acts in a Hamiltonian way.
\end{lm}
\begin{proof}
The argument is similar to \cite[Sec.8.3]{RZ1}.
As $\pi$ is a $\mathbb{T}$-equivariant projective morphism over an affine variety $X$, one can build a $\mathbb{T}$-equivariant embedding $i: Y \hookrightarrow X \times \C P^n \hookrightarrow \C^{\gens} \times \C P^n$, for some $n$, where the $\C^{\gens}$-entry of $i$ is precisely the map $\Upsilon$. 
In particular $Y$ is quasi-projective, and we can pull back the standard K\"{a}hler form from $\C^{\gens} \times \C P^n$, so that the maximal compact subgroup $\mathbb{T}_{\R}\cong (S^1)^d$ acts in a Hamiltonian way on $Y$ since it does so on $\C^{\gens} \times \C P^n$.
\end{proof}

The $S^1$-action by $\Fi_v$ is therefore Hamiltonian for each $v\in N$, with\footnote{the natural pairing $(\cdot,\cdot)$ between $\mathrm{Lie}(\mathbb{T})^*$ and $v\in \mathrm{Lie}(\mathbb{T})$ becomes the usual dot product on $\R^n$ in explicit coordinates.} Hamiltonian $H_v:Y \to \R$: 
\begin{equation}\label{Equation Ham from moment map}
H_v:Y \to \R, \; H_v(y)=(\mu(y),v),
\end{equation}
using $\mu$ from \cref{Definition symplectic T action}. The $S^1$-action by $\Fi_v$ is the flow of the Hamiltonian vector field $X_{H_v}$.

\subsection{Properties of the one-parameter subgroups}

\begin{prop}\label{Lemma Nplus for equiv proj morph}
$N_+(Y)=N_+(X)$ and $N_0(Y)=N_0(X)$, so we write $N_+$ and $N_0$ from now on.
\end{prop}
\begin{proof}
As $\pi$ is continuous, if $\mathrm{Fix}_Y(\Fi_v)$ is compact, then $\mathrm{Fix}_X(\Fi_v)=\pi(\mathrm{Fix}_Y(\Fi_v))$ is compact (using \eqref{Equation fixed loci relation}).
Conversely, as $\pi$ is proper, if $\mathrm{Fix}_X(\Fi_v)$ is compact, then $\mathrm{Fix}_Y(\Fi_v))$ is a closed subset in the compact set $\pi^{-1}(\mathrm{Fix}_Y(\Fi_v))$, thus compact.

Thus, it remains to show $N_0(Y)=N_0(X)$. For $N_0(Y)\subset N_0(X)$ we only use surjectivity and continuity of $\pi$: lift a point of $X$ to $Y$, consider its convergence point in $Y$, project back down via $\pi$ to get the convergence point in $X$.

To show $N_0(X)\subset N_0(Y)$: let $v\in N_0(X)$ and $y\in Y$, abbreviate $\psi:=\Fi_v$, then we need to show that $\lim \psi_t(y)$ exists as $t \to 0$ in $\C^*$. By assumption, $\pi(\psi_t(y))=\psi_t(\pi(y)) \to x_0\in X$ exists, as $t \to 0$.
Pick a compact neighbourhood $C$ of $x_0$ in $X$. Then, for small $t\in \C^*$, $\psi_t(y)$ is trapped in $\pi^{-1}(C)$, which is compact as $\pi$ is proper.
One could now use properties of properness in analytic geometry; we instead use a dynamical argument as in the proof of \cref{Lemma commuting actions preserve core}: $\psi_{e^{-2\pi s}}(y)$ for $s>0\in \R$ is the flowline of $y$ for $-\nabla H_{\psi}$; it lies in the compact set $\pi^{-1}(C)$; so it converges to a critical point $y_0\in Y$. Since $\mathrm{Crit}(H_{\psi})=\mathrm{Fix}(\psi)$, the point $y_0$ is in particular fixed by the $S^1$-action. Thus, $\psi_{e^{-2\pi s}}(\psi_{e^{i\theta}}(y))=\psi_{e^{i\theta}}(\psi_{e^{-2\pi s}}(y))\to \psi_{e^{i\theta}}(y_0)=y_0$ as $s\to \infty$. We deduce that $\psi_t(y)\to y_0$ for $t\in \C^*$.
\end{proof}

\begin{lm}\label{Lemma trivial subtorus trick} 
There is an algebraic subtorus $S\cong (\C^*)^k$ in $\mathbb{T}$ acting trivially on $X$, so $T:=\mathbb{T}/S\cong (\C^*)^{d-k}$ acts faithfully on $X$.
Write $N_S\cong \Z^k$ and $N_T:=N/(S\cap N)\cong \Z^{d-k}$ for the lattices of cocharacters for $S$ and $T$. Let $N_{T,+}\subset N_{T,0} \subset N_T$ be the relevant semigroups for the $T$-action on $X$. 
$$N_0 \cong N_{T,0}\times N_S \;\;\textrm{ and }\;\; N_+\cong  N_+(T)\times N_S.$$
\end{lm}
\begin{proof}
The kernel of the $\mathbb{T}$-action on $X$ is an algebraic subtorus $S$ (using that the action is analytic). The short exact sequence of algebraic tori $0\to S \to \mathbb{T} \to T \to 0$ can be split, e.g.\;by first choosing a splitting $N\cong N(T) \times N_S$ of the SES on lattices. So $\mathbb{T}\cong \mathbb{T}'\times S$ for a (non-unique choice of) subtorus $\mathbb{T}'\subset \mathbb{T}$ with $\mathbb{T}'\cong T$ via the quotient $\mathbb{T}\to \mathbb{T}/S$. 
The action of $\Fi_v$ on $X$ is unaffected if we multiply by $\Fi_s$ for any $s\in S$, since $\Fi_s$ acts by the identity on $X$. Thus $N_0(X)\cong N_0(T)\times N_S$ and $N_+(X)\subset N_{T,+}\times N_S$ are immediate.
\end{proof}

\begin{ex}\label{Example how torus acts on X via equiv proj}
In \cref{Example toric intro}, before blowing up there is a natural weight one action of $T:=\C^*$ on $\C$, which is contracting, and a natural toric $\C^*$-action on $\P^1$. This yields an action of $\mathbb{T}:=(\C^*)^2$ which lifts to the blow-up, $Y$.
The map $\pi$ is the projection to $X=\C$, the subtorus $S=\{1\}\times \C^*\subset \mathbb{T}$ acts trivially on $X$, and $\mathbb{T}/S \cong T$ acts on $X=\C$ by the contracting action.
Thus $N_+(X)=\N_{>0}\times \Z \subset N=\Z^2$, and these indeed agree with the contracting actions $N_+(Y)$ on $Y$.
\end{ex}

\subsection{The symplectic $\C^*$-manifold structure}\label{Subsection equiv proj res is symp cstar mfd}

We now make the additional assumption that at least one of the $1$-parameter subgroups of $\mathbb{T}$ is a contracting action on $X$ or on $Y$ (these are equivalent conditions by \cref{Lemma Nplus for equiv proj morph}). Equivalently, we assume:
$$N_+\neq \emptyset.$$
This condition implies, but is not implied by:

\begin{cor}\label{Cor CX0 is C in equiv proj section}
$\mathrm{Fix}_X(\mathbb{T})=\{0\}$, 
equivalently $\C[X]_0=\C$ (cf.\;\cref{Subsection Torus actions on affine variety}),
in particular all $w^j\neq 0$ in \cref{Subsection Upsilon map}.
Also, all $v\in N_+$ contract $X$ to $0\in X\subset \C^m$.
\end{cor}
\begin{proof}
For $v\neq 0 \in N_+(X)$, $\mathrm{Fix}_X(\Fi_v)$ is a point, namely $0\in X\subset \C^m$. By \cref{Lemma commuting actions preserve fixed loci}, $\mathrm{Fix}_X(\mathbb{T})\subset \mathrm{Fix}_X(\Fi_v)=\{0\}$, so $\mathrm{Fix}_X(\mathbb{T})=\{0\}$. The rest follows from \cref{Cor description of complete and contracting actions} and 
\cref{Cor description of fixed locus}.
\end{proof}

\begin{thm}\label{Theorem equiv proj morph gives sympl Cstar mfd}
Any $v\in N_+$ induces on $(Y,\omega,\Fi_v)$ the structure of a symplectic $\C^*$-manifold, globally defined over $\C^{\gens}$. Indeed the $\Psi$-maps \eqref{Equation intro Psi} can be built explicitly as proper holomorphic maps using $\Upsilon$,
\begin{equation}\label{Equation Psiv map from upsilon}
\Psi_v : = z_v \circ \Upsilon: Y \to \C^{\gens},
\end{equation}
where $z_v:\C^{\gens} \to \C^{\gens}$, $z_v(z) = (z_1^{p_1},\ldots,z_{\gens}^{p_{\gens}})$ for $p_j:= \mathrm{lcm}_v/\langle v,w^j \rangle$,
where  $\mathrm{lcm}_v:=\mathrm{lcm}\{\langle v,w^j \rangle\}$.
The map $\Psi_v$ is $\Fi_v$-equivariant, using the diagonal $\C^*$-action on $\C^{\gens}$ of weight $\mathrm{lcm}_v$.
\end{thm}
\begin{proof}
Let $\Fi_v$ be the $\C^*$-action for $v\in N_+$. By \eqref{Equation NplusX} and \eqref{Equation NplusX 2} (see also \eqref{Equation definition of support P}),
\begin{equation}\label{Equation positivity of weights torus}
\langle v,w \rangle > 0 \textrm{ for all }w\neq 0 \textrm{ with }\C[X]_w \neq \{0\}, \textrm{ in particular } \langle v,w^j \rangle > 0 \textrm{ for }j=1,\ldots,{\gens}. 
\end{equation}

Let $\mathcal{R}$ denote the Reeb field for $\C^{\gens}$ whose flow is the diagonal $S^1$-action of weight one on $\C^{\gens}$, and let $\mathcal{R}_j$ be the Reeb field for the $j$-th copy of $\C$ in $\C^{\gens}$ whose flow is the natural $S^1$-action of weight $1$ on that $\C$ factor. 
Comparing $S^1$-actions, using $\mathbb{T}$-equivariance of $\Upsilon$, \eqref{Equation weight of the action of fwj} and \eqref{Equation Psiv map from upsilon},
\begin{equation}\label{Equation Upsilon XH in terms of Reebs}
\Upsilon_* X_{H_v} = \bigoplus_{j=1}^{\gens} \langle v , w^j \rangle \mathcal{R}_j, \quad \textrm{ and }\quad 
(\Psi_v)_*X_{H_v} = \bigoplus_{j=1}^{\gens} \langle v , w^j \rangle (z_v)_*(\mathcal{R}_j) =  \bigoplus_{j=1}^{\gens} \mathrm{lcm}_v \cdot \mathcal{R}_j = \mathrm{lcm}_v\cdot \mathcal{R}.
\end{equation}
Thus $\Psi_v$ satisfies all the requirements needed to yield the claim. 
\end{proof}

\begin{rmk}
By choosing generators, we can identify $\mathbb{T}= (\C^*)^d$, $N=\Z^d$, $M=\Z^d$ (see  \cref{Subsection Torus actions on affine variety}).
The standard basis vector $e_i\in N=\Z^d$ is the $\C^*$-action $\Fi_{e_i}$ on $X$ corresponding to the $i$-th $\C^*$-factor of $\mathbb{T}$.
Moreover $\langle v,w^j\rangle = \sum_i v_i w_i^j$ (the winding number of $w^j \circ v: \C^*\to \C^*$), $\Fi_v=\prod \Fi_{e_i}^{v_i}$, $H_v = \sum_i v_i H_i$, where $H_i:=H_{e_i}$ is the moment map for the $i$-th $S^1$-factor of $\mathbb{T}_{\R}\cong (S^1)^d$ (more abstractly see \eqref{Equation Ham from moment map}).
\end{rmk}

\subsection{The core}

\begin{prop}\label{Core is same for all v} 
$\mathrm{Core}(Y,\Fi_v)\!=\pi^{-1}(0)=\!\Upsilon^{-1}(0)$ 
 for all $v\in N_+$, 
and $Y$ deformation retracts onto the core.
\end{prop}
\begin{proof}
The first claim can be proven by algebraic arguments 
as in \cite[Lem.3.15]{vzivanovic2022exact}, or by dynamical arguments as follows.
As
$\pi^{-1}(0)$ is projective, it is compact. So the $+\nabla H_v$-flow on the $\Fi_v$-invariant compact set $\pi^{-1}(0)$ automatically admits limits which are  $\C^*$-limits (see the proof of \cref{Lemma commuting contracting actions description of fix and core}); so $\pi^{-1}(0)\subset \mathrm{Core}(Y,\Fi_v)$. Conversely, $\pi(\mathrm{Core}(Y,\Fi_v))\subset \mathrm{Core}(X,\Fi_v)=\{0\}$ by equivariance. Thus $\mathrm{Core}(Y,\Fi_v)=\pi^{-1}(0)=\Upsilon^{-1}(0)$.
The deformation retraction follows by \cite[Prop.3.14]{RZ1}, using that the core is an analytic variety: it is cut out by analytic equations via $\Upsilon=0$.
\end{proof}

\subsection{Invariance of symplectic cohomology and comparison of filtrations}

\begin{lm}
$k(v,v'):=\min_{j=1,\ldots,{\gens}} \frac{\langle  v,w^j  \rangle}{\langle  v',w^j  \rangle}$ from \eqref{Equation constant kvvprime} is finite, positive, and rational, for all $v,v'\in N_+$.
\end{lm}
\begin{proof}
This uses that $\langle  v,w^j  \rangle>0$ for all $j=1,\ldots,{\gens}$ and all $v\in N_+$ (see \eqref{Equation NplusX})
\end{proof}

By \cref{Theorem Intro SH as loc of QH} and \cref{Theorem equiv proj morph gives sympl Cstar mfd}, we know $Q_{\Fi_v}\in QH^*(Y)$ is well-defined for all $v\in N_+$, and $c^*:QH^*(Y) \to SH^*(Y,\Fi_v)$ is surjective and corresponds to localisation at $Q_{\Fi_v}$. We also obtain a filtration $\Fil^p_{v}$ of $QH^*(Y)$ by ideals, via \eqref{Equation introduction filtration}, induced by $\Fi_v$.

\begin{thm}
\label{Theorem commuting Cstar actions theorem from toric section}
For all $v,v',v''\in N_+$:
\strut\begin{enumerate}
    \item There is an isomorphism $SH^*(Y,\Fi_v) \cong SH^*(Y,\Fi_{v'})$ obtained as the direct limit of the Floer continuation maps
\begin{equation}\label{Equation monotone homotopy intro}
\qquad\qquad HF^*(k\lambda H_{v'}) \to HF^*(\lambda H_v)   
\;\textrm{ for any }
k\leq k(v,v')\textrm { and } \lambda>0,
\end{equation}
such that $k\lambda, \lambda$ are generic slopes for $H_{v'},H_v$ respectively.
\item  $
\Fil^{kp}_{v'} \subset \Fil^{p}_{v}.
$
 \item If $v=v'+v''$, so $\Fi_{v}=\Fi_{v'}\circ \Fi_{v''}$, then
 \begin{equation}\label{Equatoin Q products for equiv proje}
     Q_{\Fi_{v'}}\star Q_{\Fi_{v''}}=T^{a(v',v'')} Q_{\Fi_{v}} \quad \in QH^*(Y),
 \end{equation}
 where $T$ is the formal Novikov variable, and $a(v',v'')\in \R$ will be defined in \cref{Corollary a function computed}.
\end{enumerate}
\end{thm}
\begin{proof}
The key issue is whether the maximum principle for Floer continuation solutions applies -- we will assume the reader is familiar with the discussion of this principle in the context of symplectic $\C^*$-manifolds, which is extensively discussed in our two foundational papers \cite{RZ1,RZ2}. Here we vary the $\C^*$-action, so the new key idea is to use the new function $\Upsilon$ from \eqref{Definition of Upsilon} to control Floer theory, rather than the $\Psi$-function \eqref{Equation Psiv map from upsilon} that we normally used in \cite{RZ1,RZ2}.

We claim that the maximum principle for the $\C$-coordinate $\Upsilon_j$ ensures that \eqref{Equation monotone homotopy intro} is well-defined.
Let $H_s$ be a homotopy arising from a convex combination of $\lambda H_v$ and $k\lambda H_{v'}$. Let $u$ be a Floer continuation solution in $Y$ for $H_s$.
Its $\Upsilon$-projection to $\C^{\gens}$ has $j$-th coordinate 
$$v_j:=\Upsilon_j\circ u\in \C.$$ 
We claim that $v_j$ satisfies the Floer continuation equation for a \emph{monotone} homotopy $h_s$ of radial Hamiltonians in $\C$ (so the maximum principle applies to $\Upsilon_j$, cf.\;\cite[Appendix D]{R13}). Using the notation from the proof of \cref{Theorem equiv proj morph gives sympl Cstar mfd},  the Hamiltonian vector field is 
\begin{equation}\label{Equation continuation positivity proof}
X_{h_s}=[s k\lambda \langle v',w^j \rangle + (1-s) \lambda \langle v,w^j \rangle]\,\mathcal{R}_j
\end{equation}
for $s\in [0,1]$. It remains to show that $h_s$ is \emph{monotone}, meaning that the slopes (the term in squared brackets above) has $\partial_s$-derivative $\leq 0.$ This $\partial_s$-derivative is 
$$k\lambda \langle v',w^j \rangle - \lambda \langle v,w^j \rangle \leq 0,$$
using $k\leq k(v,v').$
We now apply a standard ``ladder argument'' (cf.\;\cite[Sec.C]{R13}). To define a map $$SH^*(Y,\Fi_v)=\varinjlim HF^*(\lambda H_v) \to \varinjlim HF^*(\lambda' H_{v'})=SH^*(Y,\Fi_{v'}),$$ it suffices to build $HF^*(\lambda H_v) \to HF^*(\lambda' H_{v'})$ for $\lambda'\gg\lambda$, compatibly with compositions (the compatibility on cohomology of continuation maps is a standard result, provided they satisfy the maximum principle). Composing with a reverse map for $\lambda''\gg \lambda'$ the composites are continuations $HF^*(\lambda H_v)\to HF^*(\lambda''H_v)$ and these maps induce the identity $\mathrm{id}_{SH^*(Y,\Fi_v)}$, because they form part of the directed system whose direct limit defines $SH^*(Y,\Fi_v)$. This concludes the proof of Claim 1.

To prove Claim 2, we now use the fact that continuation maps are compatible. Denote $c^*_{v,\lambda}$ the $c^*_{\lambda}$-map in \eqref{Equation introduction filtration} for $\Fi_v$. Compatibility of continuation maps implies that, on cohomology,
$$
c^*_{v,\lambda} = \psi_{\lambda}\circ c^*_{v',k\lambda}
$$
where $\psi_{\lambda}$ is the map \eqref{Equation monotone homotopy intro}. Therefore $\ker c^*_{v',k\lambda}\subset \ker c^*_{v,\lambda}$ for generic $\lambda$. Claim 2 now follows from \eqref{Equation introduction filtration}.

Claim 3 follows from the property that $\varphi\mapsto Q_{\varphi}$ respects group multiplication up to $T$-factor corrections (composition for $\varphi$-maps, and quantum product for $Q_{\varphi}$-elements). This is explained in \cite[Theorem 7.23]{RZ1} and \cite[Appendix]{RZ3}, leaning on the original constructions from Seidel \cite{Sei97} and Ritter \cite{R14}.
\end{proof}

\subsection{Rotation classes for complete actions}

For toric manifolds it will be important to have well-defined classes $Q_{\Fi_v}$ not just for $v\in N_+$ but also for $v\in N_0=\{v\in N: \langle v,w^j \rangle \geq 0 \textrm{ for }j=1,\ldots,{\gens} \}$.

\begin{lm}\label{Lemma borderline Cstar actions}
For $v_0\in N_0$ the (possibly non-contracting) $\C^*$-action $\Fi_{v_0}$ yields a well-defined class $Q_{\Fi_{v_0}}\in QH^*(Y)$, such that $c^*Q_{\Fi_{v_0}}\in SH^*(Y,\llambda_v)$ is invertible for all $v\in N_+$. 
\end{lm}
\begin{proof}
Abbreviate by $\psi$ the $S^1$-action by $\Fi_{v_0}$.
We use a mild generalisation of \cite[Theorem 7.23]{RZ1}, following \cite[Sec.2C]{R16}. In the current proof, we will assume the reader has some familiarity with that argument, and we also recommend the appendix from \cite{RZ3} where we discuss Seidel isomorphisms for commuting $\C^*$-actions.
Note we have two $\C^*$-actions in play: $\Fi_v$ and $\psi$, where $\Fi_v$ is the action that endows $Y$ with a symplectic $\C^*$-structure $(Y,\omega,\Fi_v)$, by \cref{Theorem equiv proj morph gives sympl Cstar mfd}.

Since $\Fi_v,\psi$ commute (as they arise from a torus action), the discussion from \cite[Sec.A.2]{RZ3} applies: the ``pull-back'' Hamiltonian used in the Seidel isomorphism induced by $\psi$ is
$$\psi^*H_v := H_v\circ \psi - H_{v_0}\circ \psi = H_v-H_{v_0},$$ 
and $\psi^*I:=d\psi^{-1}\circ I \circ d\psi=I$ since $\psi$ is a holomorphic action.
Seidel's chain isomorphism is
$$
\mathcal{S}_{\psi}: CF^*(H_v,I) \to CF^{*+2|\psi|}(\psi^*H_v,\psi^*I),
$$
where $|\psi|$ is the Maslov index of $\psi$. We then fit this into a commutative diagram analogous to \cite[Theorem 7.23]{RZ1},
$$
\xymatrix@C=55pt@R=15pt{ SH^*(Y,\Fi_v)  \ar@{<-}_{\varinjlim}[d]
\ar@{->}_-{\sim}^-{\mathcal{R}_{\psi}}[rr]  & & SH^{*+2|\psi|}(Y,\Fi_v)
 \ar@{<-}^{\varinjlim}[d] \\
HF^*(\lambda H_v)  \ar@{<-}_{c_{\lambda}^*}[d] 
\ar@{->}_-{\sim}^-{\mathcal{S}_{\psi}}[r]  & HF^{*+2|\psi|}(\lambda H_v-H_{v_0})
\ar@{->}[r]^-{\textrm{continuation}}
 &
 HF^{*+2|\psi|}(\lambda H_v)
 \ar@{<-}^{c_{\lambda}^*}[d]
 \\
 QH^*(Y)
\ar@{->}[rr]^-{r_{\psi}} & & QH^{*+2|\psi|}(Y)
}
$$
The justification of this diagram reduces to explaining why the indicated continuation map is well-defined, which we carry out below.
This was also the key moment in the paper where the first author \cite{R14} first constructed the rotation maps $\mathcal{R}_{\psi}$ and $r_{\psi}$, where one saw why rotation elements $R_{\psi}=\mathcal{R}_{\psi}(1)$ are always invertible whereas $Q$-elements $Q_{\psi}=r_{\psi}(1)$ need not be. The reason $R_{\psi}=c^*Q_{\psi}\in SH^*(Y,\Fi_v)$ is invertible is because the upper square of the diagram can also be constructed for $\psi^{-1}$: in the diagram we replace $\lambda H_v - H_{v_0}$ by $\lambda H_v + H_{v_0}$, and we instead use a continuation $HF^*(\lambda H_v + H_{v_0}) \to HF^*(\lambda' H_v)$ for $\lambda'\gg \lambda$ (the analogous argument as for \eqref{Equation monotone homotopy intro} again ensures the maximum principle for $\Upsilon_j$). This trick usually fails for $QH^*(Y)$, because $c^*_{\lambda'}$ is usually not an isomorphism when $\lambda'$ is too large, and continuation maps are usually only defined when they increase the slope, so we cannot ``go back'' to $QH^*(Y)$ from $HF^*(\lambda'H_v).$

The map $c^*_{\lambda}: QH^*(Y)\to HF^*(\lambda H_v)$ is an isomorphism for all small enough $v$ (see \cite{RZ1}), and then the bottom square of the diagram identifies the middle horizontal map with the lower horizontal map, and by definition $Q_{\psi}:=r_{\psi}(1)$.

It remains to show that there exists a monotone homotopy from $\lambda H_v$ to $\lambda H_v - H_{v_0}$.
Arguing as in \eqref{Equation continuation positivity proof}, by considering the $\Upsilon_j$ projection, we now have $s$-dependent slopes
$$
s (\lambda \langle v,w^j \rangle - \langle v_0,w^j \rangle) + (1-s) \lambda \langle v,w^j \rangle =
-s\langle v_0,w^j \rangle+
\lambda \langle v,w^j \rangle,
$$
and we need the $\partial_s$-derivative to be $\leq 0$. This derivative equals $-\langle v_0,w^j \rangle$ and it is $\leq 0$ by \cref{Equation positivity of weights torus not strict}.
It follows that the maximum principle applies to each $\Upsilon_j$ projection, just as in the proof of \cref{Theorem commuting Cstar actions theorem from toric section}(1), so the required continuation map is well-defined.
\end{proof}

\subsection{The function $a(v',v'')$.}\label{Remark about t powers when multiply Q classes}

The powers of $T$ that arise in \eqref{Equatoin Q products for equiv proje} have to do with an inevitable technical issue that already arose in the original Seidel representation \cite{Sei97}: namely, that the rotation classes $Q_{\psi}\in QH^*(Y)$ depend not just on the Hamiltonian $S^1$-action $\psi$, but also on a choice of {\bf lift} $\widetilde{\psi}$ of the action of $\psi$ to a certain cover of the space $\mathcal{L}_0Y$ of free contractible loops in $Y$, consisting of certain equivalence classes of filling discs for the contractible loops. This technical issue is discussed in detail in \cite[Appendix]{RZ3} and in \cite[Sec.7.5 and Theorem 7.23]{RZ1}.

In \cite[Appendix]{RZ3}, for general symplectic $\C^*$-manifolds, we described the general procedure for finding powers of $T$ that arise in relations such as \eqref{Equatoin Q products for equiv proje}. In the monotone (Fano) setting, one can deduce the precise powers of the formal Novikov variable $T$ a posteriori, by considering the $\Z$-grading (our grading convention is that we place $T$ in grading $|T|=2$), e.g.\;see \cite[Sec.4D]{R16}.

We now explain a new result that describes the precise powers of $T$ that arise in the setting of Hamiltonian torus actions. The assumption $N_+\neq \emptyset$ ensures the following:

\begin{lm}
    There exists a $\mathbb{T}$-fixed point $c$ in $Y$.
\end{lm}
\begin{proof}
    By \cref{Cor CX0 is C in equiv proj section}, $\mathrm{Fix}_X(\mathbb{T})=\{0\} \neq \emptyset$.
    By \eqref{Equation fixed loci relation}, $\mathrm{Fix}_Y(\mathbb{T})$ must also be non-empty.
\end{proof}

\begin{rmk}
The subsequent discussion does not heavily rely on the existence of a $\mathbb{T}$-equivariant projective morphism $\pi: Y \to X$. It would suffice that $Y$ is a connected symplectic $\C^*$-manifold $(Y,\omega,\Fi)$ where $\Fi$ commutes with a symplectic $\mathbb{T}$-action (see \cref{Definition symplectic T action}) with $\mathrm{Fix}_Y(\mathbb{T})\neq \emptyset$, and that we are given some semigroup $N_0\subset N\cong \Z^d$ for which the rotation classes $Q_{\Fi_v}\in QH^*(Y)$ can be defined. The fact that  
\begin{equation}\label{Equation product of Q classes correct T weight} 
Q_{\Fi_{v'}}\star Q_{\Fi_{v''}}=T^{a(v',v'')} Q_{\Fi_{v'+v''}}
\end{equation}
holds for some power of $T$ is then a general feature of the construction of the $Q$ classes, and the question of determining $a(v',v'')$ reduces to a question about how we lift the $S^1$-actions: a question that only depends on $Y$, not on the existence of $X$.
\end{rmk}

We briefly review some background. Let $\Fi,\psi$ be two Hamiltonian $S^1$-actions on $Y$. Suppose the additional data of ``lifts'' $\widetilde{\Fi},\widetilde{\psi}$ is made. Then the construction of the rotation classes ensures that
\begin{equation}\label{Equation Q classes with correct lifts multiply}
Q_{\widetilde{\Fi}}\star Q_{\widetilde{\psi}} = Q_{\widetilde{\Fi}\widetilde{\psi}}.
\end{equation}
We call $\widetilde{\Fi}\widetilde{\psi}$ the ``induced lift'' of $\Fi\psi$.
The issue that causes the $T$-correction factors in \eqref{Equation product of Q classes correct T weight} is that we may have already chosen a lift for $\Fi_{v'+v''}$ to define its rotation class, and this lift is not the same as the one induced by the lifts chosen for $\Fi_{v'},\Fi_{v''}$. We now explain this in greater detail.

Let $c$ be any fixed point of $\Fi$. As shown in \cite[Appendix]{RZ3}, the choice of a so-called ``lift'' $\widetilde{\Fi}$ is completely determined\footnote{This follows from a connectedness property of a certain cover of the space of contractible loops (whose deck group is $G$), which holds in our setting because we are considering $S^1$-actions that admit fixed points \cite[Appendix]{RZ3}.}
by how $\widetilde{\Fi}$ acts on $c$, where $c$ is viewed as a constant disc at $c$ that fills the constant $S^1$-orbit at $c$. Two choices of lifts will differ by the action of a group of spherical classes: 
$$A\in G:=\pi_2(Y)/(\ker [\omega] \cap \ker c_1),$$
where $c_1:=c_1(TY)$. More explicitly, suppose $\widetilde{\psi}$ and $\psi^{\wedge}$ are two choices of lifts, satisfying
$\widetilde{\psi}(c)=c$ and $\psi^{\wedge}(c)=A\cdot c$ for $A\in G$ (where $G$ is viewed as the deck group of the cover of $\mathcal{L}_0 Y$ mentioned above). Then in fact
$\psi^{\wedge}=A\cdot \widetilde{\psi}$, and therefore
the rotation classes satisfy
\begin{equation}\label{Equation correction to Q from lifts}
Q_{\psi^{\wedge}} = Q_{A\cdot \widetilde{\psi}}=  T^{-\omega[A]} Q_{\widetilde{\psi}},
\end{equation}
where $T$ is the Novikov formal variable, and $\omega[A]$ is the evaluation of $[\omega]\in H^2(Y,\R)$ on $A\in H_2(Y)$.
There is also a degree shift, that we omitted: the right-hand side is shifted down in grading by $2c_1(A)$ (see \cite[Appendix]{RZ3}). Our sign convention above is opposite to the original paper by Hofer-Salamon \cite{HS95} were the spherical classes were introduced, because we use cohomological conventions, explained in \cite[Sec.3.2]{liebenschutz2020intertwining}: the deck group action by $G$ on the cover of the loop space by ``gluing in sphere classes'' corresponds to rescaling by $T^{-\omega(A)}$ and adding $-2c_1(A)$ to the grading.

The exponent $a(v',v'')$ in \eqref{Equation product of Q classes correct T weight} is precisely $\omega[A]$ for the spherical class $A$ that encodes how the induced lift differs from the chosen lift of $\Fi_{v'+v''}.$

Strictly speaking, in \eqref{Equation correction to Q from lifts} one should also introduce a second formal variable to keep track of grading shifts by $2c_1(A)$, or one indicates the grading shifts. We refer to \cite[Appendix]{RZ3} for details. In the Fano case, those grading shifts are automatically encoded by placing the variable $T$ in suitable grading; in the CY case the matter does not arise as $c_1=0$; whereas outside of those two cases one can avoid this issue by making due with a $\Z/2$-grading in Floer theory, alternatively one can have a $\Z$-grading by essentially encoding $G$ into the Novikov ring as in \cite{Sei97}.

Let $H_{\psi}:Y \to \R$ be a moment map for $\psi$. Let $\mathrm{Min}\,  H_{\psi}\subset Y$ denote its minimal locus. By adding a constant, we also assume that $\min  H_{\psi}:=H_{\psi}(\mathrm{Min}\,  H_{\psi})=0.$ 
When we write $Q_{\psi}$, without indicating a lift of $\psi$, we tacitly assume that we use the {\bf canonical lift} $\psi^{\wedge}$ of $\psi$ which fixes the constant disc at a point $c_{\psi}\in \mathrm{Min}\,  H_{\psi}$.
With this choice, $Q_{\psi}\in QH^{2\mu_{\psi}}(Y)$ where $\mu_{\psi}$ is the Maslov index of $\psi$ described in \cite[Sec.5.1]{RZ1}. Moreover, in the K\"{a}hler Fano or K\"{a}hler CY case we obtain
\begin{equation}\label{Equation Q class is PD of min plus}
Q_{\psi} = \mathrm{PD}[\mathrm{Min}\,  H_{\psi}] + (T^{>0}\textrm{-terms}) \in QH^{2\mu_{\psi}}(Y),
\end{equation}
under the assumption that the linearised $\psi$ action on the tangent space $T\mathrm{Min}\,  H_{\psi}$ only has weights $0$ and $1$ \cite[Proposition 1.35]{RZ1}.

The reason the issue \eqref{Equation product of Q classes correct T weight} arises, is that there may not be a fixed point $c$ which is simultaneously a minimum for $H_{\Fi_{v'}}$, $H_{\Fi_{v''}}$ and $H_{\Fi_{v'+v''}}$. 

Even the mere omission of the assumption $\min H_{\psi}=H_{\psi}(\mathrm{Min}\,  H_{\psi})=0$, causes $T$-corrections:

\begin{lm}\label{Lemma correction caused by constant in moment map}
If $\Fi=\psi$ and we choose the lifts that fix a constant disc at $c_{\psi}\in \mathrm{Min}\,H_{\psi}$, but we choose the moment map $H_{\Fi}=\lambda + H_{\psi}$ for a constant $\lambda\in \R$, then $Q_{\Fi}=T^{\lambda}Q_{\psi}$.
\end{lm}
\begin{proof}
This is technical, so the proof will be surgical without explaining the context. In the proof of the first Lemma in Sec.7.8 of \cite{R14}, we have%
\footnote{Regarding the sign: by definition (cf.\;\cite[Thm.7.23]{RZ1}), and using cohomological conventions, $Q_{\widetilde{\Fi}}=\mathrm{PSS}^{-1}\circ \mathcal{S}_{\widetilde{\Fi}} \circ \mathrm{cont}\circ \mathrm{PSS}(1)$, where $\mathrm{PSS}:QH^*(Y)\stackrel{\cong}{\to} HF^*(H_{\textrm{small}})$ for a small slope Hamiltonian (the ones used to build the canonical maps $c^*$); $\mathrm{cont}$ is a ``monotone'' continuation map from $H_{\mathrm{small}}+H_{\psi}$ to $H_{\mathrm{small}}$; and the Seidel isomorphism $\mathcal{S}_{\widetilde{\Fi}}: HF^*(H_{\mathrm{small}}+H_{\psi})\to HF^{*+\textrm{shift}}(H_{\mathrm{small}})$ is given by multiplication by $\psi_{-t}$ on $1$-orbits $x(t)$ (cf.\;\cite[Appendix]{RZ3}). The solutions are not spheres in $Y$, but they are spheres in a certain clutching bundle $E_{\psi}\to \C P^1$ with symplectic form $\widetilde{\Omega}+\pi^*\sigma$ \cite[Sec.5.2 and 5.4]{R14}. So the power of $T$ with which solutions are counted gets an addition of $\widetilde{\Omega}(s_{\widetilde{\psi}})$.} 
$K=-\lambda$ rather than $K=0$, and the claim follows from $$\textstyle\int \widetilde{\Omega}(s_{\widetilde{g}})=\int_{D^+}\phi'(s) \cdot (-\lambda) ds\wedge dt = (\phi(0)-\phi(-\infty))\cdot (-\lambda) = (-1-0)\cdot (-\lambda) = \lambda. \qedhere$$
\end{proof}

\begin{ex}\label{Example CP2 example Q classes}
We explain the issue with a familiar compact example: the Fano toric manifold $\C P^2$. The moment polytope is a solid triangle, and its three edges correspond to the three minimal loci of the Hamiltonians for the actions corresponding to $e_1:=(1,0)$, $e_2:=(0,1)$, $e_3=(-1,-1)$ in $N=\Z^2$ (each is a hyperplane in $\C P^2$), which in \eqref{Equation Q class is PD of min plus} turn out to give $x:=Q_{\Fi_{e_i}}=\mathrm{PD}[\textrm{hyperplane class}]=3[\omega]\in H^2(\C P^2)$. Note: $e_1+e_2+e_3=(0,0)$, which corresponds to a relation for the $S^1$-actions: $\Fi_{e_1}\Fi_{e_2}\Fi_{e_3}=\mathrm{id}.$ If this same relation held for the lifts, then we would obtain $x^3 = 1$. Instead, in $QH^*(\C P^2)$ we should obtain $x^3 = T^3$, e.g.\;due to the $\Z$-grading (placing $T$ in grading $|T|=2$). The reason $T^3$ arises is that for at least one of the three actions we cannot choose the canonical lift.  
\end{ex}

We now introduce a new trick. Recall we assumed that there is at least one $\mathbb{T}$-fixed point, say 
$$c\in \mathrm{Fix}(\mathbb{T}).$$
\begin{de}
Write $\widetilde{\Fi_v}$ for the lift determined by the condition that it fixes the constant disc at $c$.
\end{de}

\begin{lm}\label{Cor group hom with can lifts}
The map $N_0 \to QH^*(Y)$, $v \mapsto Q_{\widetilde{\Fi}_v}$ is a group homomorphism (in particular, without $T$-factor corrections).
\end{lm}
\begin{proof}
For $v',v''\in N_0$, by definition $\widetilde{\Fi_{v'}}\widetilde{\Fi_{v''}}(c)=\widetilde{\Fi_{v'}}(c)=c$ and $\widetilde{\Fi_{v'+v''}}(c)=c$. Since $\Fi_{v'}\Fi_{v''}=\Fi_{v'+v''}$ are the same, and their lifts agree on $c$, it follows that $\widetilde{\Fi_{v'}}\widetilde{\Fi_{v''}}=\widetilde{\Fi_{v'}}$. Finally, use \eqref{Equation Q classes with correct lifts multiply}.
\end{proof}

It remains to compare $Q_{\widetilde{\Fi}_v}$ and $Q_{\Fi_v}$. 
Let $H_v:=(\mu,v)$ be the moment map from 
 \eqref{Equation Ham from moment map}.
Since $\min H_v$ may not be zero, we also define
$$H_v^{0}:=H_{v}-\min H_v,$$
which is zero along the minimum locus. 
Recall
$Q_{\Fi_v}:=Q_{\psi_v^{\wedge}}$ involves using the moment map $H_v^0$ and using a lift $\psi_v^{\wedge}$ of $\psi_v$ that fixes the constant disc at a point $c_v\in \mathrm{Min}\,  H_v=\mathrm{Min}\,  H_v^0$. Write $Q_v$ to mean that same lift but using the moment map $H_v=\min H_v + H_v^0$ instead of $H_v^0$, so by \cref{Lemma correction caused by constant in moment map}:
\begin{equation}\label{Equation canonical lifted rotation classes}
 Q_{\Fi_v}:=Q_{\Fi_v^{\wedge}}
 \qquad \textrm{ and } \qquad
Q_v:=T^{\min H_v} Q_{\Fi_v^{\wedge}} =T^{\min H_v}Q_{\Fi_v}.
\end{equation}

Pick a path $\gamma(s)$ from the minimum $c_v\in \mathrm{Min}\,  H_v$ to $c$.
We explained in \cite[Appendix, Lemma A.11]{RZ3}, that applying the $S^1$-action we get a sphere $u(s,t)=(\Fi_v)_{-t}\gamma(s)$ where $t\in S^1$ (at the two ends we get $c$ and $c_v$ since they are fixed points of $\Fi_v$), which defines a spherical class $A_v\in G$. Moreover,
\begin{equation}\label{Equation omega on Av class}
\omega[A_v] = H_v(c_v) - H_v(c) = \min H_v - H_v(c).
\end{equation}
By \cite[Lemma A.7]{RZ3}, the canonical lift $\Fi_v^{\wedge}$ (determined by the condition $\Fi_v^{\wedge}(c_v)=c_v$) acts by 
$$\Fi_v^{\wedge}(c)=A_v\cdot c.$$
By definition, $\widetilde{\Fi}_v(c)=c$, therefore, using \eqref{Equation correction to Q from lifts},
\begin{equation}\label{Equation canonical lift versus T fixed pt lift}
\Fi_v^{\wedge} = A_v\cdot \widetilde{\Fi}_v, \quad \textrm{ and thus }\quad 
Q_{\Fi_v^{\wedge}} = T^{-\omega[A_v]}Q_{\widetilde{\Fi}_v}.
\end{equation}
\begin{cor}\label{Corollary a function computed}
For $v',v''\in N_0$,
\begin{equation}\label{Equation afunction for Q classes multiplied}
\begin{split}
& Q_{v'}\star Q_{v''} =  Q_{v'+v''}
\qquad
\textrm{ and }
\qquad Q_{\Fi_{v'}}\star Q_{\Fi_{v''}}=T^{a(v',v'')} Q_{\Fi_{v'+v''}}
\\
& \textrm{where }\qquad 
a(v',v''):=\min H_{v'+v''}-\min H_v -\min H_{v''}.
\end{split}
\end{equation}
More generally, for any $v_1,\ldots,v_k \in N_0$,
$$
Q_{\Fi_{v_1}}\star Q_{\Fi_{v_2}} \star \cdots \star Q_{\Fi_{v_k}} =  
T^{a(v_1,\ldots,v_k)} Q_{\Fi_{v_1+v_2+\cdots+v_k}},
$$
where $a(v_1,\ldots,v_k):=\min H_{v_1+\cdots+v_k }-\min H_{v_1}-\min H_{v_2}-\cdots-\min H_{v_k}.$
\end{cor}
\begin{proof}
Using \eqref{Equation canonical lifted rotation classes}, \eqref{Equation canonical lift versus T fixed pt lift}
and \cref{Cor group hom with can lifts}, \eqref{Equation canonical lifted rotation classes},
 \begin{equation*}
     \begin{split}
 Q_{v'}\star Q_{v''}
 &=
 T^{\min H_{v'}+\min H_{v''}}
 Q_{\Fi_{v'}^{\wedge}}\star Q_{\Fi_{v''}^{\wedge}}
 \\
 &=
 T^{\min H_{v'}+\min H_{v''}-\omega[A_{v'}+A_{v''}]}
 Q_{\widetilde{\Fi}_{v'}}\star Q_{\widetilde{\Fi}_{v''}}
  \\
 &=
  T^{\min H_{v'}+\min H_{v''}-\omega[A_{v'}+A_{v''}]}
 Q_{\widetilde{\Fi}_{v'+v''}}
   \\
 &=
  T^{\min H_{v'}+\min H_{v''}-\omega[A_{v'}+A_{v''}-A_{v'+v''}]}
 Q_{\Fi_{v'+v''}^{\wedge}}
    \\
 &=
  T^{\min H_{v'}+\min H_{v''}-\omega[A_{v'}+A_{v''}-A_{v'+v''}]-\min H_{v'+v''}}
 Q_{v'+v''}.
 \end{split}
 \end{equation*}
The final $T$-exponent is zero by noting that $H_v(c)=(\mu(c),v)$ is linear in $v$, so \eqref{Equation omega on Av class} and \eqref{Equation canonical lifted rotation classes} implies 
$$a(v',v'')=-\omega[A_{v'}+A_{v''}-A_{v'+v''}]=\min H_{v'+v''}-\min H_{v'}-\min H_{v''}.$$
The final claim follows by an induction argument, i.e.\;by repeatedly applying \eqref{Equation afunction for Q classes multiplied}, and noticing that
$a(v_1,\ldots,v_k)+a(v_1+\ldots+v_k,v_{k+1}) = a(v_1,\ldots,v_{k+1})$.
\end{proof}

\begin{rmk}
By \cite[Appendix]{RZ3}, the grading shift is as follows: $Q_{\Fi_{v'+v''}}$ in \cref{Equation afunction for Q classes multiplied} is shifted down in grading by $2c_1[A_{v'}+A_{v''}-A_{v'+v''}]$.
\end{rmk}

\begin{cor}\label{Cor group hom with can lifts version 2}
The following map is a group homomorphism,
\begin{equation}\label{Equation Q represenation without correction factors}
N_0 \to QH^*(Y), \quad v \mapsto Q_v:=T^{\min H_v} Q_{\Fi_v}.
\end{equation}
\end{cor}
\begin{proof}
    This follows from \cref{Corollary a function computed}.
\end{proof}

\begin{ex}
As a sanity check, in \cref{Example CP2 example Q classes} for $\C P^2$, the moment polytope $\Delta\subset \R^2$ is the image of $\C P^2$ via $(H_{e_1},H_{e_2})$, and gives a solid triangle with vertices $(0,0),(3,0),(0,3)$ (the factor $3$ ensures that $[\omega_{\Delta}]=c_1=3[\omega_{FS}])$. Then $\min H_{e_1} = 0$ corresponds to the bottom facet of $\Delta$; $\min H_{e_2} = 0$ corresponds to the left facet; whereas $H_{e_3} = -H_{e_1}-H_{e_2}$ achieves $\min H_{e_3}=-3$ on the slanted facet. So $Q_{e_3}=T^{\min H_{e_3}} Q_{\Fi_{e_3}} = T^{-3}x$.
Via the analogoue of \eqref{Equation Q represenation without correction factors} for $\C P^2$, the relation $v_1+v_2+v_3 = 0$ maps to the relation $x\star x \star (T^{-3}x) = 1$. So indeed we get $x^3 = T^3$ in $QH^*(Y)$.
\end{ex}

\begin{prop}\label{Proposition Q class when vectors lie in a cone}
Suppose $v=m_1 v_{1}+\cdots + m_k v_{k}$, for $m_j\in \N$, $v_j\in N_0$, such that 
\begin{equation}\label{Equation common min}
\mathrm{Min}(H_{v_{1}})\cap \cdots \cap \mathrm{Min}(H_{v_{k}}) \neq \emptyset.
\end{equation}
Then $$Q_{\Fi_v} = Q_{\Fi_{v_{1}}}^{m_1}\star \cdots \star Q_{\Fi_{v_{k}}}^{m_k}.$$
\end{prop}
\begin{proof}
By the linearity of $H_v=(\mu,v)$ in $v$, we get 
$$
H_v = m_1 \, H_{v_{1}}+\cdots + m_k \, H_{v_{k}}.
$$
The $H_{v_{1}}, \ldots,H_{v_{k}}$ attain a common minimum at a point $p$ in \eqref{Equation common min}. Since $m_j\geq 0$, this is also a minimum of $H_v$. Thus $a(m_1 v_1,\ldots,m_k v_k)=H_v(p)-\sum m_j H_{v_j}(p) = (\mu(p),v-\sum m_j v_j) = (\mu(p),0)=0.$ 
\end{proof}
\section{Semiprojective toric manifolds}\label{Subsection semiprojective toric}
\subsection{Definition of semiprojective toric manifolds}

We follow \cite[Sec.2]{HS02} and \cite[p.332]{cox2011toric}.
A semiprojective toric manifold is a non-compact toric manifold $Y$ for which the affinisation map 
\begin{equation}\label{Equation affinisation map}
\pi:Y\to X:=\mathrm{Spec}(H^0(Y,\mathcal{O}_Y))
\end{equation}
is projective, and $Y$ has at least one torus-fixed point (i.e.\;the moment polytope $\Delta$, discussed later, has at least one vertex).
Note that $Y$ is non-compact precisely if $X\neq (\mathrm{point})$, which we will always assume. The case $X=(\mathrm{point})$ is just the study of (compact) projective toric manifolds $Y$ (e.g.\,$\P^n$).

Being a toric variety, $Y$ is the closure of an open dense torus $\mathbb{T}\cong (\C^*)^n$, where $n=\dim_{\C}Y$. Via $\pi$, we also obtain a $\mathbb{T}$-action on $X$. We will explain in \cref{Subsection Semiprojective toric manifolds as a subclass of equivariant projective morphisms} that the $\pi$-equivariant projective morphism \cref{Equation affinisation map} satisfies all the requirements from \cref{Subsection Setting up notation and assumptions}.

By \cite[Cor.2.7]{HS02}, these toric varieties can be described in terms of certain twisted\footnote{Called \textit{projective} GIT quotients in \cite{HS02}, but we emphasize that these are typically non-compact varieties.} GIT quotients of $\C^{n+d}$ by a complex torus $\mathbb{T}^d$ (in which case, $\mathbb{T}$ arises from $(\C^*)^{n+d}/\mathbb{T}^d$).

They can also be described in terms of a certain type of {\bf fan} $\Sigma$: any \emph{triangulation}\footnote{$\Sigma$ is a simplicial fan whose rays lie in $\mathcal{B}$ and whose support $|\Sigma|$ equals pos$(\mathcal{B})=$(convex $\R$-span of $\mathcal{B})\subset N_{\R}:=N\otimes_{\Z} \R$.} of a spanning set $\mathcal{B}$ of a lattice $N\cong \Z^n$, such that $\Sigma$ is \emph{unimodular}\footnote{each maximal cone is spanned by a basis of $N$. This ensures smoothness of $Y$.} and \emph{regular}.\footnote{it admits an ample $\mathbb{T}$-Cartier divisor.
Equivalently, there is a continuous convex function pos($\mathcal{B}$)$ \to \R$, integer-valued on $N\cap $pos($\mathcal{B}$), which restricts to a different linear function on each maximal cone of $\Sigma$ (see also \cite[Def.15.2.8.]{cox2011toric}).}
Unimodularity implies that $\Sigma$ has an $n$-dimensional cone, so $Y$ is simply connected \cite[p.56]{fulton1993introduction}.

There are other equivalent definitions of semiprojective toric manifolds  (cf.\,\cite[Thm.2.6]{HS02} and  \cite[Prop.7.2.9]{cox2011toric}), one of which is mentioned in the next Subsection, where we discuss more technical properties of $X,Y$.
\subsection{More technical properties of semiprojective toric manifolds}\label{Remark description of X}

An equivalent definition is that a semiprojective toric manifold $Y$ is a non-singular quasi-projective toric variety whose fan $\Sigma$ has full dimensional convex support, meaning: 
$$|\Sigma|\subset N_{\R} \textrm{ is convex, and }\dim |\Sigma|=\dim_{\R}N_{\R},$$
cf.\;\cite[p.265]{cox2011toric}. We abbreviate $n=\dim_{\C} \mathbb{T}=\dim_{\C} Y$, so $\mathbb{T}\cong (\C^*)^n$ and $N_{\R}\cong \R^n$. That condition on the support $|\Sigma|$ (defined as the union of the cones of $\Sigma$) implies that $|\Sigma|$ equals the union of the $n$-dimensional cones and arises as a cone:
\begin{equation}\label{Equation support is span of ray gens}
|\Sigma| = \mathrm{span}_{\R_{\geq 0}}(e_i: \textrm{ for all }i) \subset N_{\R},
\end{equation}
where $e_i$ are the (finitely many) ``minimal generators'' for $\Sigma$: the first non-zero integral points along the (ray) edges, i.e.\;the one-dimensional cones of $\Sigma$ (cf.\;\cref{Rmk further technical properties of Nplus X}). 

Recall\footnote{see respectively \cite[p.39]{fulton1993introduction},
\cite[p.113, Thm.3.1.19]{cox2011toric}, and \cite[p.29]{fulton1993introduction},
\cite[p.113, Thm.3.1.19]{cox2011toric}.} that $Y$ is projective precisely if $\Sigma$ is complete:\;$|\Sigma|=N_{\R}\cong \R^n$; and 
$Y$ is non-singular precisely if $\Sigma$ is regular: each cone of $\Sigma$ is generated by part of a $\Z$-basis of the lattice $N$.\
Our assumptions that $Y$ is non-singular and non-compact are therefore equivalent to: $\Sigma$ is regular and $|\Sigma|\neq \N_{\R}$.

Abbreviating $M:=N^{\vee}=\mathrm{Hom}(N,\Z)\cong \Z^n$ (the lattice that labels characters), the dual cone 
$$|\Sigma|^{\vee}=\{v\in M_{\R}: \langle v,e_i\rangle =0\textrm{ for all }i\}$$ 
therefore defines a finitely generated semigroup $R:=|\Sigma|^{\vee}\cap M$. The characters $M$ define rational functions on $Y$, and the semigroup $R$ describes  precisely the characters $\chi^s$ that are global functions: they define a $\C$-linear basis for $H^0(Y,\mathcal{O}_Y)=\C[X]$,
so semigroup generators for $R$ yield generators for the unital $\C$-algebra $\C[X]$ \cite[p.193, Ex.4.3.4]{cox2011toric}.

Following \cite{cox2011toric},\footnote{p.321, p.332 and Ex.7.2.3. for the description of $\pi$ (further properties of $\pi$ are described in p.321-322 and p.326), and p.318 Sec.7.1 for the recession cone. Compare also with \cref{Rmk further technical properties of Nplus X}.}
$X$ is a normal affine toric variety whose fan $\Sigma_X$ consists of only one cone\footnote{more precisely, a strongly convex rational polyhedral cone.} that can be explicitly constructed from the fan $\Sigma$ of $Y$ as follows. Let $W:=|\Sigma|\cap (-|\Sigma|)$ be the maximal $\R$-linear subspace inside $|\Sigma|$. Then $\Sigma_X$ is the strongly convex cone $$\Sigma_X:=|\Sigma|/W,$$ together with all of its faces (above, $|\Sigma|/W$ means the image of $|\Sigma|$ via the vector space quotient $N_{\R}\to N_{\R}/W$, cf.\;also the notation in  \eqref{Equation fan in NR}).
The lattice used for $X$ is 
$$N_X:=N/(N \cap W),$$
in particular this determines the torus for $X$ as a quotient of the torus for $Y$,
$$
T \cong \mathbb{T}/S.
$$
If one starts with \eqref{Equation affinisation map}, and defines the natural $\mathbb{T}$-action on $X$ defined by \eqref{Equation affinisation map}, then $S$  is the subtorus of $\mathbb{T}$ which acts trivially on $X$.
Recall that any toric variety contains an open dense torus, and the action of that embedded torus on itself extends to the torus action of the toric variety.
One can view $\mathbb{T}$ and $T$ as the embedded open dense tori for $Y,X$ respectively.
From this point of view, the map $\pi: Y \to X$ from \eqref{Equation affinisation map} restricted to $\mathbb{T}\subset Y$ is the quotient map, 
$\pi|_{\mathbb{T}}:\mathbb{T} \to T=\mathbb{T}/S$. This indirectly determines $\pi$ by continuity, if one considers limit points in $Y$. 
More explicitly, $\pi: Y \to X$ is the morphism of toric varities induced\footnote{in the sense of \cite[p.22-23]{fulton1993introduction}, see also \cite[Sec.3.3 p.125, Lemma 3.3.21 p.135, Prop.1.3.14 p.41]{cox2011toric}.}
by the morphism on cones $\Sigma \to \Sigma_X$ given by the obvious quotient map.

\begin{cor}\label{Corollary pi surjective for STM}
$\pi:Y \to X$ is surjective, and $X,Y$ are connected.
\end{cor}
\begin{proof}
The connectedness is immediate because $X,Y$ are closures of (connected) algebraic tori.
Since $\pi$ is proper, it is a closed map (indeed it is universally closed), both in the Zariski and the Euclidean/analytic topology (e.g.\;\cite[p.143 and Thm.3.4.11 p.144]{cox2011toric}). 
Since $\pi(\mathbb{T})=T$, the continuous closed map $\pi$ sends the closure $Y=\overline{\mathbb{T}}$ to the closure $\overline{T}=X$.
\end{proof}

By \cite[p.332-333 and p.329-330]{cox2011toric}, $Y$ admits a ``moment polytope'', more precisely this is a non-compact full dimensional lattice polyhedron $\Delta\subset M_{\R}\cong \R^n$, 
\begin{equation}\label{Equation moment polytope for SPM} 
\Delta=\{x\in M_{\R}:\langle x,e_i\rangle \geq \lambda_i \}\subset \R^n,
\end{equation}
where $e_i$ are the rays generators of $\Sigma$ from \eqref{Equation support is span of ray gens}, %
and they are also the primitive inward-pointing normals to the facets of $\Delta$ (the top-dimensional faces); the parameters $\lambda_i\in \Z$ are discussed below. To avoid confusion: $\Delta$ is not a (compact) polytope in the usual sense of ``convex hull of finitely many points in $\R^n$'', but it is a convex polytope in the sense of ``intersection of a collection of half-spaces'', in particular it is an unbounded rational polyhedron. It was described in \cite[p.498 and p.503]{HS02}, and we briefly mention some properties.

One could define semiprojective toric varieties starting from such polyhedrons \cite[p.332]{cox2011toric}, e.g.\;the normal vectors $e_i$ to facets meeting at a vertex of $\Delta$ will define the top-dimensional faces of $|\Sigma|$ etc.: more precisely, $\Sigma$ is the normal fan of $\Delta$  \cite[p.76-77]{cox2011toric}. Thus, our assumption that $Y$ is non-compact corresponds to the condition that $\Delta$ is non-compact, and that $Y$ is non-singular corresponds to the condition that at each vertex of $\Delta$, the normals to the facets of $\Delta$ extend to a $\Z$-basis of $N\cong \Z^n$.

The parameters $\lambda_i\in\R$ are more subtle.\footnote{\cite[p.66]{fulton1993introduction} and \cite[Thm.7.2.4, p.328, p.333-334]{cox2011toric}.} 
They arise from an additional choice: the existence of a $\mathbb{T}$-invariant Cartier divisor $D=-\sum \lambda_i D_i$ on $Y$ with strictly convex support function $\Fi_D$, and $\lambda_i:=\Fi_D(e_i)\in \Z$ (the $D_i$ are the toric divisors, corresponding to the facets of $\Delta$; we discuss these later).
This is additional information because it arises from presenting $Y$ as a quasi-projective variety by constructing an embedding as in the proof of \cref{Lemma equiv proj morph sympl form}, and pulling back an ``$\mathcal{O}(1)$'' bundle. This is related to the general feature (e.g.\,see \cite[Sec.A2]{R16}) that a moment polytope for a non-singular toric variety determines a fan and a symplectic form, but a fan does not encode the information needed to determine a moment polytope and a symplectic form. For semiprojective toric manifolds, we can apply \cref{Lemma equiv proj morph sympl form} to the $\mathbb{T}$-equivariant projective morphism $\pi$ above to obtain this additional data; one can also construct a closed $\mathbb{T}$-equivariant embedding $Y\hookrightarrow \C P^a\times \C^b$ in the GIT picture, after some choices \cite[p.503]{HS02}.  Either way, the K\"{a}hler form arises by pulling back a $\mathbb{T}_{\R}$-invariant K\"{a}hler form via such an embedding, but auxiliary choices were made in addition to having the fan $\Sigma$.

There is also a ``moment polytope'' for $X$, in the sense of \cite[Sec.4.2]{fulton1993introduction}, and it agrees with the so-called ``recession cone'' of the polyhedron $\Delta$:
$$\Delta_X:=\{x\in \R^n:\langle x,e_i\rangle \geq 0 \}.$$ 
The above moment polytopes are related by $$\Delta=\Delta_X+(\textrm{the convex hull of the vertices of }\Delta).$$

\begin{ex}
The lower part of the picture in \cref{Example toric intro} shows $\pi$ from \eqref{Equation affinisation map}, $\Sigma_X$ and $\Delta_X$. Here, $W$ is the $y$-axis. So $X\cong \C$ arises from the cone $\Sigma_X = \{(x,y)\in \R^2: x\geq 0 \}/\{(0,y): y\in\R \} \cong \{x\in \R:x\geq 0\}$ (so essentially $\mathrm{span}_{\R_{\geq 0}}(e_1)$, but more precisely it should be viewed in the quotient $\R^2/W$). 
\end{ex}

\subsection{Semiprojective toric manifolds as a subclass of equivariant projective morphisms}\label{Subsection Semiprojective toric manifolds as a subclass of equivariant projective morphisms}
The $\mathbb{T}$-equivariant projective morphism $\pi:Y \to X$ from \eqref{Equation affinisation map} satisfies all the requirements we had in \cref{Subsection Setting up notation and assumptions}; in particular the topological properties are shown in \cref{Corollary pi surjective for STM}. The $T$-action on $X$ has a unique fixed point for general reasons: for any toric variety, the fixed points are in one-to-one correspondence with the interiors of full-dimensional cones, and $X$ has only one full-dimensional cone. We will describe $N_+(X)$ explicitly in \cref{Cor complete actions semiproj toric case}, which implies that $N_+(X)\neq \emptyset,$ although the condition $N_+(X)\neq \emptyset$ could also be deduced by combining the discussion in \cref{Rmk further technical properties of Nplus X} and \cref{Lemma trivial subtorus trick}.  

Recall the notation $N=\mathrm{Hom}(\C^*,\mathbb{T})$, $\Fi_v$, $N_0$ and $N_+$ from \cref{Subsection Setting up notation and assumptions}
and \cref{Lemma Nplus for equiv proj morph}.
\begin{cor}
$(Y,\llambda_v)$ is a $\C^*$-symplectic manifold globally defined over $\C^m$ for all $v\in N_+$, where, following \cref{Subsection Upsilon map},
 $m$ is the number of (a choice of) non-constant $\mathbb{T}$-homogeneous generators 
 $$f_{w^1},\ldots,f_{w^m} \;\textrm{ for }\; \C[X]=H^0(Y,\mathcal{O}_Y).$$
\end{cor}
\begin{proof}
    This follows by combining \cref{Cor complete actions semiproj toric case} with \cref{Theorem equiv proj morph gives sympl Cstar mfd} (note that $X=\mathrm{Spec}(H^0(Y,\mathcal{O}_Y))$ implies $\C[X]=H^0(Y,\mathcal{O}_Y)$ by definition).
\end{proof}

\subsection{The moment map and moment polytope}\label{Subsection The moment map and moment polytope}
Let $\mathbb{T}_{\R}\cong (S^1)^n$ be the maximal compact subgroup of $\mathbb{T}$.
For the toric K\"{a}hler form $\omega$ from \cref{Lemma equiv proj morph sympl form},
the $\mathbb{T}_{\R}$-action is Hamiltonian with a moment map
$$
\mu:Y \to \mathrm{Lie}(\mathbb{T})^{*}\cong \R^{n},
$$
described in greater detail in \cite[Eq.(7)]{HS02}. The image of $\mu$ is the {\bf moment polytope} from \eqref{Equation moment polytope for SPM},
\begin{equation}\label{Equation moment polytope for SPM 2} 
\Delta=\mathrm{image}(\mu)=\{x\in M_{\R}:\langle x,e_i\rangle \geq \lambda_i \}\subset \R^n.
\end{equation}

Hausel--Sturmfels show in \cite[Prop.2.11]{HS02} that the presentation of the ordinary cohomology of $Y$ is just like in the case of projective toric varieties (e.g.\,compare \cite[Sec.3A]{R16}):  
\begin{equation}\label{Ordinary Cohomology of Semiprojective toric}
    \Z[x_1,\ldots,x_r]/(\textrm{Linear relations},\textrm{Stanley--Reisner relations})\cong H^*(Y;\Z),\; x_i \mapsto \mathrm{PD}[D_i],
\end{equation}
determined combinatorially from $\Delta$, where $D_i$ are the toric divisors (which we review next).
\subsection{Notation and recollections about toric varieties}\label{Subsection Notation and recollections about toric varieties}
    We follow \cite[Ch.3]{fulton1993introduction}, \cite[Sec.A1.4]{guillemin2012moment}. There is an inclusion-reversing correspondence\footnote{Indeed $\tau\leq \gamma \Leftrightarrow \overline{O_{\tau}}\supset O_{\gamma}.$ Moreover, $\dim_{\C} O_{\tau} =\mathrm{codim}_{\R}\, \tau$. Example: for $\tau=\{0\}\subset \R^n$, $O_{\tau}=\mathbb{T}\subset Y.$} between cones $\tau$ of the fan $\Sigma$ and the closures 
    \begin{equation}\label{Equation Vtau expression}
    V(\tau):=\overline{O_{\tau}}=\sqcup_{\gamma\geq \tau} O_{\gamma} \subset Y
    \end{equation}
    of the corresponding $\mathbb{T}$-orbits $O_{\tau}$. 
    There are {\bf distinguished points} $x_{\gamma}$ that uniquely determine orbits, 
    $$O_{\gamma}=\mathbb{T}\cdot x_{\gamma}.$$ 
    There is a cover of $Y$ by affine open sets $U_{\gamma}$, defined by
    $$\strut\hspace{42ex} U_{\gamma}:=\sqcup_{\tau\leq \gamma} O_{\tau} \qquad (\textrm{in particular, }\gamma\leq \gamma'\Rightarrow U_{\gamma}\subset U_{\gamma'}).$$ 
    From the point of view of $\Delta$, there is an inclusion-reversing correspondence between faces of $\Delta$ and cones of $\Sigma$. Namely, if a subset $S$ of indices determines a cone $\gamma_S$ generated by $(e_i: i\in S)$, then  
    $$\Delta_S:=\{x\in \R^n:\langle x,e_i\rangle = \lambda_i \textrm{ for } i\in S \}$$
    is a face of $\Delta$, moreover the closure of the orbit $O_{\gamma_S}$ satisfies
    $$V(\gamma_S)=\mu^{-1}(\Delta_S)\subset Y.$$
For example: the {\bf toric divisors} $D_i=V(e_i)=\mu^{-1}(\Delta_{\{i\}})\subset Y$ correspond to the rays $e_i$ of the fan; the $\mathbb{T}$-fixed points are $\mu^{-1}(\textrm{vertices of }\Delta)$ and correspond to the (top) $n$-dimensional cones of $\Sigma$.

The {\bf support} $|\Sigma|\subset N$ denotes the union of all cones. A cone is {\bf full-dimensional} if it has dimension $n$. 
The {\bf boundary} $\partial |\Sigma| $ is the topological boundary, and $\mathrm{Int}|\Sigma|=|\Sigma|\setminus \partial |\Sigma|$ is the {\bf interior}. So cones of $\Sigma$ are either  
{\bf exterior} or {\bf interior}, depending on whether they are a subset of $\partial |\Sigma| $  or $\mathrm{Int}|\Sigma|$.

The interior rays of $\Sigma$ correspond to the compact toric divisors $D_i$, which in turn correspond to the compact facets of $\Delta$ (the {\bf facets} are the codimension one faces). More generally, the interior cones $\gamma$ of $\Sigma$ correspond, in an inclusion-reversing way, to the compact faces of $\Delta$.  
\subsection{A description of the toric $\C^*$-actions for any toric variety}

By \cite[p.36]{fulton1993introduction}, the lattice traditionally denoted $N$, which is used in the construction of a toric variety, can be canonically identified with the $1$-parameter subgroups $\mathrm{Hom}(\C^*,\mathbb{T})$, which we denoted $N$ above. Note $N\cong \Z^n$. In particular, this lattice is used in the construction of the fan 
\begin{equation}\label{Equation fan in NR}
\Sigma\subset N_{\R}:=N\otimes_{\Z} \R.
\end{equation}
The $1$-parameter subgroup $\llambda_v$, for $v\in N$, yields a $\C^*$-action on $Y$, whose $S^1$-part has moment map \eqref{Equation Ham from moment map}.
We now discuss the properties of $\llambda_v$ for general toric varieties.
    We first give a local description, following \cite[p.17, p.29, p.37]{fulton1993introduction}. 
    Let $e_1,\ldots,e_n$ be a basis of $N$, so identify $N=\Z^n$.
    Suppose a cone $\gamma$ of $\Sigma$ is generated by $e_1,\ldots,e_k$. Then, 
    $$U_{\gamma}=\C^k\times (\C^*)^{n-k}\;\; \textrm{ and }\;\; \mathbb{T}=(\C^*)^n\subset U_{\gamma}.$$
    Any $S\subset \{1,\ldots,k\}$ defines a cone $\gamma_S\leq \gamma$, with $U_{\gamma_S}=\C^{\karo_1}\times \cdots \times \C^{\karo_k} \times (\C^*)^{n-k}\subset U_{\gamma}$, where $\C^{\karo_i} = \C$ if $i\in S$, and $\C^*$ else; 
    with distinguished point $x_{\gamma_S}=(\delta_1,\ldots,\delta_k,1,\ldots,1)\in U_{\gamma}$ where $\delta_i=0$ if $i\in S$, and $\delta_i=1$ else, with orbit $O_{\gamma_S}=\square_1\times \cdots \times \square_k \times (\C^*)^{n-k}$ where $\square_i=\{0\}$ if $i\in S$, and $\C^*$ else. Moreover the closure $V(\gamma_S)=\triangledown_1 \times \cdots \times \triangledown_k \times (\C^*)^{n-k}$ where $\triangledown_i=\{0\}$ if $i\in S$, and $\triangledown_i = \C$ else.

    Any $v=(v_1,\ldots,v_n)\in N=\Z^n$ yields the $\C^*$-subgroup $\llambda_v(t)=(t^{v_1},\ldots,t^{v_n})$ and acts on $y\in U_{\gamma}$ by 
    $$\llambda_v(t) y = (t^{v_1}y_1,\ldots,t^{v_n}y_n).$$
    Recall that the cone $\gamma$ is the $\R_{\geq 0}$-span of $e_1,\ldots,e_k$ rather than the $\R$-span. Let $ \R\gamma:=\mathrm{span}_{\R}(\gamma)\subset N_{\R}$. The sublattice generated by $N\cap \gamma$ is denoted
    $$N_{\gamma}:=N \cap \R \gamma.$$

    \begin{lm}\label{Lemma fixed point torus action}
    The fixed locus of $\llambda_v$ corresponds to the union of cones $\tau$ of $\Sigma$ with $v\in N_{\tau}$, thus:
    \begin{equation}\label{Equation fixedlocus Fiv}
    \mathrm{Fix}(\llambda_v) = \bigcup_{N_{\tau}\ni\,v} V(\tau) = \bigsqcup_{N_{\tau}\ni\,v}  O_{\tau}.
    \end{equation}
    Moreover, $v\in \gamma \Longleftrightarrow ($the limits \,$\lim_{t\to 0}\llambda_v(t)y$\, exist in $U_{\gamma}$ for all $y\in U_{\gamma})$.
    
    The action of $\llambda_v(t)$ on $O_{\tau}$ admits limits that lie in $O_{\gamma}$ as $t\to 0$ precisely if $$N_{\R} \to N_{\R}/\R\tau$$ 
    maps $v$ into the relative\footnote{``relative'' because it refers to using the subspace topology for the cone.} interior of $\overline{\gamma}:=(\gamma$ viewed in $N_{\R}/\R\tau)$. 
    
    The same condition ensures that the action of $\llambda_v(t)$ on $V(\tau)$ admits limits that lie in $O_{\gamma}$ as $t\to 0$. 
\end{lm}
\begin{proof}   
By the inclusion-reversing cone-orbit correspondence, to prove the first equality in \eqref{Equation fixedlocus Fiv}, it suffices to consider only the minimal cones $\tau$ with $N_{\tau}\ni v$. Since $U_{\gamma}$ is a cover of $Y$, we just need to show $\mathrm{Fix}(\llambda_v)\cap U_{\gamma}=V(\tau).$

    Continuing the notation preceding the Lemma, $y\in \mathrm{Fix}(\llambda_v)$ precisely if $v_i=0$ whenever $y_i\neq 0$, in particular $v_i=0$ for $i>k$.
        Reorder the indices $i$ so that $v_1,\ldots,v_m$ are non-zero (where $m\leq k$), and $v_{m+1},\ldots,v_n$ are zero.
        The minimal face $\tau\leq \gamma$ containing $v$ is spanned by $e_1,\ldots,e_m$ (so $\tau=\gamma_S$ for $S=\{1,\ldots,m\}$, above). Note that $N_{\tau}=\Z^m\times \{0\}\subset \Z^n=N$, and $\tau$ is the minimal cone such that $v=(v_1,\ldots,v_m,0,\ldots,0)\in N_{\tau}$.  By the results preceding the Lemma, 
$$
\mathrm{Fix}(\llambda_v)\cap U_{\gamma} = 0\times \C^{k-m}\times (\C^*)^{n-k} = V(\tau) \supset 0\times 0 \times (\C^*)^{n-k}=V(\gamma).
$$
This concludes the proof of the first equality in \eqref{Equation fixedlocus Fiv}.
The second equality follows from \eqref{Equation Vtau expression}, or by the explicit description of the orbits, preceding the Lemma. So far, this was \cite[Exercise p.56]{fulton1993introduction}.

    For the limit in the claim to exist, we need precisely $v_1,\ldots,v_m>0$ and $v_{m+1}=\cdots=v_n=0$ (thus $v\in \gamma$), and the limit is    
    $(0,\ldots,0,y_{m+1},\ldots,y_n)$. This is closely related to \cite[Claim 1, p.38]{fulton1993introduction} which states that $\lim_{t \to 0}\llambda_v(t)=x_{\tau}$, where $x_{\tau}\in U_{\gamma}$ is the distinguished point for the face $\tau\leq \gamma$ which contains $v$ in its relative interior.
    In our case, $x_{\tau}=(0,\ldots,0,1,\ldots,1)$ with the first $1$ in slot $m+1$.

    We now prove the last two claims. For $y\in O_{\tau}=(\C^*)^n\cdot x_{\tau}$, if the limit of $\llambda_v(t)(y)$ exists, it is in $V(\tau)=\overline{O_{\tau}} = \sqcup_{\tau \leq \gamma} O_{\gamma}$, say in $O_{\gamma}$ with $\tau\leq \gamma$.
    By reordering coordinates analogously to the above notation, we may assume 
    $$x_{\tau}=(0,\ldots,0,1,\ldots,1)\in O_{\tau}=(\C^*)^n\cdot x_{\tau}=0\times (\C^*)^{n-m}\supset 0\times (\C^*)^{n-k} = O_{\gamma}.$$
   The above limit exists in $V(\tau)$ precisely if $v_i>0$ for $i=m+1,\ldots,k$, and $v_i=0$ for $i=k+1,\ldots,n$ (no conditions on $v_1,\ldots,v_m$). Using overlines to denote the images in the quotient $N_{\R}/\R\tau$, note that $N_{\R}/\R\tau=\mathrm{span}_{\R}\{ \overline{e_{m+1}},\ldots,\overline{e_n} \}$, and $\overline{\gamma}$ is the $\R_{\geq 0}$-span of $\overline{e_{m+1}},\ldots,\overline{e_k}$. The relative interior of $\overline{\gamma}$ therefore consists of $\R_{>0}$-linear combinations involving all vectors $\overline{e_{m+1}},\ldots,\overline{e_k}$.
    So those conditions on $v_i$ are equivalent to $\overline{v}$ being in the relative interior of $\overline{\gamma}$. The final claim for $V(\tau)$ follows by the same argument (the only novelty is that $y$ is now allowed to have coordinate 
    $y_i=0$ for $m+1\leq i\leq k$).
\end{proof}

\begin{ex}\label{Example finding fix fi v}
    If $v=e_i$ is a ray of $\Sigma$, then $\mathrm{Fix}(\llambda_v)$ contains the toric divisor $D_i=\mu^{-1}(\Delta_{\{i\}})\subset Y$ associated to the facet $\Delta_{\{i\}}$ of $\Delta$ whose inward normal is $e_i$.
    Due to the $\R$-span $\R\gamma$ in the definition of $N_{\gamma}$, there may\footnote{Example:\;$\C P^2$. The fan has rays $(1,0)$, $(0,1)$, $(-1,-1)$. Then $[t^{-1}z_0:t^{-1}z_1:z_2]=[z_0:z_1:tz_2]$ is the
    action for $v=(-1,-1)$.
    Its fixed locus contains the toric divisor $\{[z_0:z_1:0]\}$ but also the fixed point $[0:0:1]$ ($\leftrightarrow$ a vertex of $\Delta$).} be more components in $\mathrm{Fix}(\llambda_v)$.
    In \cref{Example toric intro}, $v=(0,1)$ has $N_v=\{0\}\times \Z$ (not just $\{0\}\times \N$), so it contains $e_2,e_4$, and thus $\mathrm{Fix}(\Fi_v)=D_2\cup D_4.$

    If $v$ lies in the interior of a full-dimensional cone $\tau$ of $\Sigma$, 
    then $\mathrm{Fix}(\llambda_v)\supset V(\tau)=\{x_{\tau}\}$ contains the obvious $\mathbb{T}$-fixed point, but there may be other fixed components.\footnote{$\P^1\times \P^1$: the fan has rays $(\pm 1,0)$, $(0,\pm 1)$, and $v=(2,3)$ acts by $[t^2z_0:z_1]\times [t^3w_0:w_1]$ with four fixed points.}

    $\mathrm{Fix}(\llambda_v)$ consists of isolated fixed points if $v \in |\Sigma|$ is ``sufficiently generic'': $v\notin \cup \{N_{\tau}: \dim \tau <n\}$.
\end{ex}

\subsection{Complete and contracting toric $\C^*$-actions for semiprojective toric manifolds}\label{Subsection complete and contr actions for toric semiproje}

Recall the notation $N=\mathrm{Hom}(\C^*,\mathbb{T})\cong \Z^n$, $\Fi_v$, $N_0$ and $N_+$ from \cref{Subsection Setting up notation and assumptions}
and \cref{Lemma Nplus for equiv proj morph}.

\begin{prop}\label{Cor complete actions semiproj toric case}
Let $Y$ be any semiprojective toric manifold. Then
$$
N_0 = N\cap |\Sigma| \;\;\;\textrm{ and }\;\;\;
N_+ = N\cap \mathrm{Int}\,|\Sigma|,
$$
in particular $N_+\neq \emptyset$ since $\dim |\Sigma|=n$.
\end{prop}
\begin{proof}
By \cref{Lemma fixed point torus action}, $\lim_{t\to 0}\llambda_v(t)(y)$
exists in $Y=\cup_{\gamma} U_{\gamma}$ for all $y\in Y$ precisely if $v\in \cup \,\{\textrm{all cones }\gamma\} = |\Sigma|$. This implies $N_0=N\cap |\Sigma|.$

Recall $v\in N_+$ precisely if $v\in N_0$ and $\mathrm{Fix}(\Fi_v)$ is compact.  
Abbreviate $\psi=\Fi_v$ for $v\in N_0$. We follow \cite[p.52]{fulton1993introduction}. Let 
     $\mathrm{Star}(\gamma)$ be the collection of all cones of $\Sigma$ that contain $\gamma$ as a face, but viewed in the quotient lattice
     $N/N_{\gamma}$. Then
     \begin{equation}\label{Equation Vgamma toric}
     V(\gamma)=\overline{O_{\gamma}}=\sqcup \{O_{\tau}: \gamma \textrm{ is a face of }\tau\} = (\textrm{toric variety with fan }\mathrm{Star}(\gamma)).
     \end{equation}
    By \eqref{Equation fixedlocus Fiv}, we need to ensure that $V(\gamma)$ is compact for all $\gamma$ such that $v\in N_{\gamma}$.
    Recall that a toric variety with fan $\Sigma$ is compact precisely if $|\Sigma|=N_{\R}$ \cite[p.39]{fulton1993introduction}.
Apply this to $V(\gamma)$ whose fan is $\mathrm{Star}(\gamma)$: so we want $|\mathrm{Star}(\gamma)|=(N/N_{\gamma})\otimes \R$. This holds provided that $\gamma$ is an interior cone of $\Sigma$ (so its image $\overline{\gamma}=0\in (N/N_{\gamma})\otimes \R$ is surrounded by full-dimensional cones, which in the quotient will give a cover for $(N/N_{\gamma})\otimes \R$). Thus compactness of $\mathrm{Fix}(\llambda_v)$ is equivalent to requiring that $v$ is not spanned by an exterior cone, more precisely: $v\notin N\cap \partial |\Sigma|$. Since we already showed $v\in |\Sigma|$, the claim follows.
\end{proof}

\begin{rmk}
One can loosely think of $v\in N_0\setminus N_+ = N\cap \partial |\Sigma|$ as ``boundary'' vectors by \cref{Cor complete actions semiproj toric case}. However, this can cause some confusion as $\langle v,w^j\rangle \geq 1$ in \eqref{Equation Nplus as a polyhedron} suggests $N_+$ is bounded away from $N_0$:
$$
N_0\setminus N_+ = \{v\in N\cong \Z^d: \langle v,w^j\rangle \geq 0 \textrm{ for }j=1,\ldots,m, \textrm{ and } \langle v,w^j\rangle = 0 \textrm{ for at least one }j\}.
$$
We clarify this. As $\Sigma$ is unimodular and regular, $v\in N_0\setminus N_+$ lies in a full-dimensional cone $\gamma$ spanned by a $\Z$-basis of $N$. So we can pick a sequence $v_n\in \gamma \cap N \cap \mathrm{Int}|\Sigma|$ with $v = \lim  a_n$ as $n\to \infty$ where $a_n:=\tfrac{\|v\|}{\|v_n\|}v_n\in N_{\R}\cong \R^n$.
As $v_n\in N \cap \mathrm{Int}|\Sigma|$, $\Fi_{v_n}$ is contracting, so $v_n$ satisfies \eqref{Equation positivity of weights torus}. Therefore $a_n\in N_{\R}$ satisfies \eqref{Equation positivity of weights torus}, but not necessarily \eqref{Equation Nplus as a polyhedron}: indeed $\langle a_n,w^j\rangle >0\in \R$ may become small due to the rescaling factor $\tfrac{\|v\|}{\|v_n\|}\in \R$. For this reason, for the limit $v$, some $\langle v,w^j\rangle \in \Z$ can vanish.
\end{rmk}

\subsection{Compact and non-compact toric divisors, and the Core}

Recall from \cref{Subsection Notation and recollections about toric varieties} that the
toric divisors $D_i\subset Y$ correspond to the ray generators $e_i$ of $\Sigma$ (the $e_i$ are also discussed in \cref{Remark description of X}).

\begin{cor}\!\!\!\label{Cor core for torus action}\footnote{The description of the core is closely related to Hausel-Sturmfels \cite[Def.3.1, Thm.3.2, and Thm.3.5]{HS02}.}
The compact toric divisors $D_i$ correspond to the $e_i\in N\cap \mathrm{Int}|\Sigma|$, and the non-compact toric divisors $D_i$ correspond to the ray generators $e_i\in N\cap \partial |\Sigma|$.

$\mathrm{Core}(Y,\Fi_v)$ is the same for all $v\in N\cap \mathrm{Int}|\Sigma|$, it deformation retracts onto $Y$, and
$$\mathrm{Core}(Y):=\mathrm{Core}(Y,\Fi_v)=\cup\, \{\textrm{compact toric divisors }D_i\},$$
and via the correspondence with $\Sigma$, it corresponds to the union of the interior cones of $\Sigma$.
\end{cor}
\begin{proof} 
The claim about toric divisors follows from the proof of \cref{Cor complete actions semiproj toric case}. As toric divisors are preserved by $\mathbb{T}$, all compact toric divisors $D_i\subset \mathrm{Core}(Y,\Fi_v)$. That this is all of the core follows by \cref{Core is same for all v}, and the explicit description of $\pi$ in \cref{Remark description of X}, as follows. Recall $\Sigma_X$ is just one cone, 
and the interior of that cone corresponds to the unique $\mathbb{T}$-fixed point. The preimage via $\pi$ corresponds to the union of interior cones of $\Sigma$. Those in turn are in inclusion-reversing correspondence with the compact faces of $\Delta$ (which yield subsets of the compact toric divisors $D_i$).
\end{proof}

\begin{ex}
    Let $Y:=\mathrm{Tot}(\mathcal{O}(-1)\to \P^1)$ (the blow-up of $\C^2$ at the origin). It arises from the fan with rays $e_1=(1,0)$, $e_2=(0,1)$, $e_3=(1,1)$, and two full-dimensional cones $\gamma,\gamma'$ generated by $e_1,e_3$ and $e_2,e_3$. For $v=(1,1)$, $\llambda_v$ is the standard contracting action on fibres of $Y$. 
    As $e_3$ spans the only interior cone, we get $\mathrm{Core}(Y,\llambda_v)=D_3=V(e_3)=($the zero section $\P^1)$.
    Observe that this is consistent with the last claim in \cref{Lemma fixed point torus action} if we test whether $\lim_{t \to \infty}\Fi_v(t)y = \lim_{t\to 0}\Fi_{-v}(t)y$ exists for all $y\in V(e_3)$: $-v$ is indeed in the relative interior of (the point!) $\R_{\geq 0}\overline{e_3}$ in $\R^2/\R e_3$.
\end{ex}

\subsection{Quantum and symplectic cohomology of semiprojective toric manifolds}
\label{Subsection Quantum and symplectic cohomology of semiprojective toric manifolds}
\begin{thm}\label{Cor SH in toric case is QH localised}
Let $Y$ be a semiprojective toric manifold.
    \begin{enumerate}
    \item $Q_{\llambda_v}\in QH^*(Y)$ is defined for all $v\in N\cap |\Sigma|$.
        \item\label{Item the a powers in q product} If $v=v'+v''$ in $N\cap |\Sigma|$, then in $QH^*(Y)$, and using the notation from \cref{Corollary a function computed},
        $$Q_{\llambda_{v'}}\star Q_{\llambda_{v''}}=T^{a(v',v'')}Q_{\llambda_v} \qquad \textrm{ and } \qquad 
        Q_{v'}\star Q_{v''}=Q_{v}.
        $$ 
We have a unital $\k$-algebra homomorphism, 
\begin{equation}\label{Equation algebra hom for Q classes toric case}
N\cap |\Sigma| \to QH^*(Y), \qquad v\mapsto Q_v := T^{\min H_v} Q_{\Fi_v}.
\end{equation}
Moreover, if $v=v_1e_{j_1}+\cdots + v_n e_{j_n}$ lies in a maximal cone of $\Sigma$ generated by $e_{j_1},\ldots,e_{j_n}$, then
\begin{equation}\label{Equation the Qv class description}
Q_{\Fi_v} = 
Q_{\Fi_{e_{j_1}}}^{v_1} \star \cdots \star Q_{\Fi_{e_{j_n}}}^{v_n}.
\end{equation}
       \item For all $v\in N\cap \mathrm{Int}|\Sigma|$, the unital $\k$-algebra homomorphism
       $$c^*:QH^*(Y) \to SH^*(Y,\llambda_v)$$ 
       is defined, surjective, and corresponds to localisation at $Q_{\llambda_v}$, such that
$$c^*Q_{\llambda_{v'}}\in SH^*(Y,\llambda_v)$$ 
is invertible for all $v'\in N \cap |\Sigma|$.
        \item For all $v,v'\in N \cap \mathrm{Int}|\Sigma|$ there is a continuation isomorphism commuting with the $c^*$-maps,
$$SH^*(Y,\llambda_v)\cong SH^*(Y,\llambda_{v'}).$$
So $SH^*(Y,\llambda_v)$ is independent of $v\in N \cap \mathrm{Int}|\Sigma|$ up to canonical isomorphism.
        \end{enumerate}
\end{thm}
\begin{proof}
This follows from \cref{Cor core for torus action} and the general theory for $\mathbb{T}$-equivariant projective morphisms in 
\cref{Theorem commuting Cstar actions theorem from toric section} and \cref{Lemma borderline Cstar actions}. In particular, \eqref{Equation algebra hom for Q classes toric case} is \cref{Cor group hom with can lifts version 2}, and \eqref{Equation the Qv class description} follows by \cref{Proposition Q class when vectors lie in a cone} (note $v_j\in \N$: integral points of the cone are $\N$-linear combinations of its generators).
\end{proof}
\section{Presentation of quantum \& symplectic cohomology for semiprojective toric manifolds}
\label{Subsection Presentation of quantum and symplectic cohomology of Fano semiprojective toric manifolds}

We now prove the quantum/Floer version of \eqref{Ordinary Cohomology of Semiprojective toric} that generalises \cite[Thm.1.5]{R16}.
Let $r\geq n$ be the number of ray generators $e_1,\ldots,e_r$, equivalently the number of toric divisors $D_1,\ldots,D_r.$

The moment polytope $\Delta=\{x\in \R^n:\langle x, e_i\rangle \geq \lambda_i\}$ determines the {\bf superpotential} $$W=\sum T^{-\lambda_i}z^{e_i},$$
where $T$ is the formal Novikov variable. The {\bf Jacobian ring} is defined as
$$
\mathrm{Jac}(W):=\k[z_1^{\pm 1},\ldots,z_n^{\pm 1}]/(\partial_{z_1}W,\ldots,\partial_{z_n}W),
$$
where $\k$ is the Novikov field.

We will assume that the reader is familiar with the discussion and notation from \cite[Sec.3A]{R16}, which explains how from $\Sigma$ (or equivalently, from $\Delta$) one obtains combinatorially the so-called {\bf linear relations} and the {\bf quantum Stanley-Reisner relations}. 
These are polynomials in variables $x_1,\ldots,x_r$, and together we denote the ideal that they generate by
$$
\mathcal{J} \subset \k[x_1,\ldots,x_r].
$$
The {\bf linear relations}, just like in the classical cohomology case \eqref{Ordinary Cohomology of Semiprojective toric}, correspond to a generating set of $\Z$-linear relations between the classes $\mathrm{PD}[D_i]\in H^2(Y,\Z)$, or equivalently a generating set of $\Z$-linear relations between the $e_i$ in $N\cong \Z^n$. They are given by
$$
\sum_{i=1}^r \langle \xi, e_i\rangle \, x_i = 0 \textrm{ where }\xi\textrm{ runs through the standard basis for }\R^n.
$$
The {\bf quantum Stanley-Reisner relations} 
$$
x_{i_1}\cdots x_{i_a} = T^{\omega(\beta_{\wp})} x_{j_1}^{c_1}\cdots x_{j_b}^{c_b},
$$
where $\omega(\beta_{\wp})$ is discussed later, arise for each $\Z$-linear dependency relation in $N\cong \Z^n$ of the form $$e_{i_1}+\cdots+e_{i_a} = c_{j_1}e_{j_1}+\cdots +c_{j_b}e_{j_b},$$ for the so-called {\bf primitive subsets} $\wp:=\{i_1,\ldots,i_a\}$, in particular the $e_{i_1},\ldots,e_{i_a},e_{j_1},\ldots,e_{j_a}$ are distinct and $\omega(\beta_{\wp})>0$ \cite[Sec.3A]{R16}.
The {\bf primitive} condition is that $\{e_{i_1},\ldots,e_{i_a}\}$ does not define a cone for $\Sigma$, but any proper subset thereof does define a cone of $\Sigma$. Equivalently: the facets of $\Delta$ corresponding to $D_{i_1},\ldots,D_{i_a}$ have empty intersection, but any proper subset thereof has non-empty intersection. The existence of $e_{j_1},\ldots,e_{j_b}$ distinct from $e_{i_1},\ldots,e_{i_a}$, and of $c_1,\ldots,c_b\in \Z$, that ensure the above relation, is a classical combinatorial result due to Batyrev \cite{batyrev1993quantum}.

Abbreviate the coordinates of $e_i$ by $e_i=(e_{i,1},\ldots,e_{i,n})\in \Z^n$. Then there is a homomorphism
\begin{equation}\label{Equation jac ring batyrev}
\k[x_1,\ldots,x_r] \to \mathrm{Jac}(W),\qquad x_i \mapsto T^{-\lambda_i}z^{e_i} := T^{-\lambda_i}z_1^{e_{i,1}}\cdots z_n^{e_{i,n}}.
\end{equation}
Via this map, the linear relations generate precisely the ideal $(\partial_{z_1}W,\ldots,\partial_{z_n}W)$ in the definition of $\mathrm{Jac}(W).$ 
The quantum Stanley-Reisner relations are tautologically built into $\k[z_1^{\pm 1},\ldots,z_n^{\pm 1}]$ by the definition of $z^{e_{i}}$, because $e_{i_1}+\cdots+e_{i_a} = c_{j_1}e_{j_1}+\cdots +c_{j_b}e_{j_b}$ guarantees $z^{e_{i_1}}\cdots z^{e_{i_a}} = z^{c_{j_1} e_{j_1}}\cdots z^{c_{j_b}e_{j_b}}$.

The $S^1$-rotation by $\Fi_{e_i}$ on $Y$ is a natural rotation action around the toric divisor $D_i$, and $D_i$ is the minimum of the moment map $H_{e_i}(y):=(\mu(y),e_i)$ for $\Fi_{e_i}$ from \eqref{Equation Ham from moment map}. Indeed, by \eqref{Equation moment polytope for SPM 2} we have $H_{e_i}(y)\geq \lambda_i$, and the value $\lambda_i$ arises precisely along the facet of $\Delta$ corresponding to $D_i$. So the minimum of $H_{e_i}$ is attained precisely on $D_i$:
\begin{equation}\label{Equation value of Hei on Di}
H_{e_i}(D_i)=\lambda_i.
\end{equation}
The above dependency relation $e_{i_1}+\cdots+e_{i_a} = c_{j_1}e_{j_1}+\cdots +c_{j_b}e_{j_b}$ in $N\cap |\Sigma|$ implies the relation
$$
\Fi_{e_{i_1}}\cdots \Fi_{e_{i_a}} = 
\Fi_{e_{j_1}}^{c_1}\cdots \Fi_{e_{j_b}}^{c_b}, 
$$
between the 1-parameter subgroups. 
Recall the notation from \eqref{Equation algebra hom for Q classes toric case} (cf.\;\cref{Remark about t powers when multiply Q classes}):
$$
Q_v:= T^{\min H_v} Q_{\Fi_v} \quad \textrm{ for } v\in N\cap |\Sigma|.
$$
\begin{thm}\label{Theorem Quantum Stanley Reisner}
The rotation classes $Q_v$ satisfy
$$
Q_{e_{i_1}}\cdots Q_{e_{i_a}}
= Q_{e_{j_1}}^{c_1}\cdots Q_{e_{j_b}}^{c_b},
\qquad
\textrm{ and }
\qquad
Q_{\Fi_{e_{i_1}}}\cdots Q_{\Fi_{e_{i_a}}}
= T^{\omega[\beta_{\wp}]} Q_{\Fi_{e_{j_1}}}^{c_1}\cdots Q_{\Fi_{e_{j_b}}}^{c_b},
$$
in $QH^*(Y)$ (so using quantum product and quantum powers), where
$$
\omega[\beta_{\wp}]:=-\lambda_{i_1}-\cdots - \lambda_{i_a}+c_1 \lambda_{j_1}+\cdots + c_b \lambda_{j_b}.
$$
In particular, $\omega[\beta_{\wp}]>0$ is the evaluation of $[\omega]\in H^2(Y,\R)$ on a sphere class $\beta_{\wp}\in H_2(Y,\Z)$, determined by the intersection product conditions $\beta_{\wp}\cdot D_{i_k}=1$, $\beta_{\wp}\cdot D_{j_k}=-c_{j_k}$, and $\beta_{\wp}\cdot D_i=0$ for all other $i$.
\end{thm}
\begin{proof}
The above relation between the rotations implies the claimed equations by \cref{Cor SH in toric case is QH localised}; the only missing ingredient is explaining the exponent $\omega[\beta_{\wp}]$.
By \eqref{Equation algebra hom for Q classes toric case}, the exponent is the dot product between $(-1,\ldots,-1,c_1,\ldots,c_b)$ and $(\min H_{e_{i_1}},\ldots,\min H_{e_{i_a}},\min H_{e_{j_1}},\ldots,\min H_{e_{j_b}})$. By \eqref{Equation value of Hei on Di}, the latter vector is $(\lambda_{i_1},\ldots,\lambda_{i_a},\lambda_{j_1},\ldots,\lambda_{j_b}).$ So that dot product is $-\lambda_{i_1}-\cdots - \lambda_{i_a}+c_1 \lambda_{j_1}+\cdots + c_b \lambda_{j_b}$.
The rest follows from the explanation in \cite[Sec.3A]{R16}, since $[\omega]=-\sum \lambda_i \mathrm{PD}[D_i] \in H^2(Y)$ for general reasons. The final claims about $\beta_{\wp}$ are classical combinatorial results due to Batyrev \cite{batyrev1993quantum}.
\end{proof}

\begin{prop}\label{Prop Fano case get divisor}
If $Y$ is Fano, then the rotation classes for the ray generators $e_i$ of $\Sigma$ are:
$$
Q_{\Fi_{e_i}} = \mathrm{PD}[D_i]  \in QH^{2}(Y), \qquad \textrm{ and }\quad Q_{e_i}=T^{\lambda_i}Q_{\Fi_{e_i}}=T^{\lambda_i}\mathrm{PD}[D_i].
$$
If $Y$ is CY (meaninig $c_1(Y)=0$), then
$$
Q_{\Fi_{e_i}} = \mathrm{PD}[D_i] + T^{>0}\textrm{-terms} \in QH^{2}(Y), \qquad \textrm{ and }\quad Q_{e_i}=T^{\lambda_i}Q_{\Fi_{e_i}}.
$$
\end{prop}
\begin{proof}
In the Fano or CY setting, \eqref{Equation Q class is PD of min plus} from \cite[Prop.1.35]{RZ1} yields
$$
Q_{\Fi_{e_i}} = \mathrm{PD}[\min H_{e_i}] + T^{>0}\textrm{-terms} \in QH^{2\mu(\Fi_{e_i})}(Y).
$$
In our setting, $\mu(\Fi_{e_i})=1$, $D_i=\mathrm{Min}\, H_{e_i}$, $\min H_{e_i}=H_{e_i}(D_i)=\lambda_i$ (cf.\;\cite[Sec.3B]{R16}).
 The fact that no higher $T^{>0}$-corrections arise in the Fano setting, is the consequence of a dimension argument, using the monotonicity assumption and the fact that $\mu(\Fi_{e_i})=1$. This observation is explained in detail in \cite[Sec.7.5]{RZ1}, and originally goes back to McDuff-Tolman's proof of Batyrev's presentation of quantum cohomology for compact Fano toric manifolds \cite{mcduff2006topological}.
\end{proof}

\begin{rmk}\label{Remark meaning of monotone}
We call $(Y,\omega)$ {\bf monotone} if $[\omega]\in H^2(Y;\R)$ is a strictly positive multiple of $c_1(Y)$. Any non-compact Fano toric manifold $Y$ admits a moment polytope \eqref{Equation moment polytope for SPM}  with $\lambda_i=-1$, which ensures that the natural K\"{a}hler form $\omega_{\Delta}$ associated to $\Delta$ satisfies $[\omega_{\Delta}]=-\sum \lambda_i \mathrm{PD}[D_i]=\sum \mathrm{PD}[D_i]= c_1(Y)$, so $\omega_{\Delta}$ is a toric monotone K\"{a}hler form (see the discussion in \cite[Sec.A4,A5,A9]{R16}).
\end{rmk}

\begin{prop}\label{Prop QH and SH of semiproj var}
Let $Y$ be a monotone
semiprojective toric manifold $Y$, and let $v\in N_+$. Then 
\begin{align*}
\k[x_1,\ldots,x_r]/\mathcal{J} &\cong QH^*(Y),\; x_i \mapsto Q_{\Fi_{e_i}}=\mathrm{PD}[D_i],
\\
\k[x_1,\ldots,x_r,x^{\pm v}]/\mathcal{J} &\cong SH^*(Y,\llambda_v) \cong \mathrm{Jac}(W),\; x_i \mapsto c^*Q_{\Fi_{e_i}}=c^*(\mathrm{PD}[D_i]) \mapsto T^{-\lambda_i}z^{e_i},
\end{align*} 
where the second isomorphism is the localisation of the first at $Q_{\Fi_v}=x^v$ (cf.\;\cref{Theorem Intro SH as loc of QH}).

Rewriting the two isomorphisms in terms of $Q_{e_i}=T^{\lambda_i}Q_{\Fi_{e_i}}$ instead of $Q_{\Fi_{e_i}}$, we have:
$$T^{\lambda_i}x_i \mapsto Q_{e_i} \qquad \textrm{ and } \qquad T^{\lambda_i}x_i \mapsto c^*Q_{e_i} \mapsto z^{e_i}.$$
Via the first isomorphism, $Q_{\Fi_v}$ corresponds to
\begin{equation}\label{Equation xv the Qv class description}
x^{v}:=x_{j_1}^{v_1}x_{j_2}^{v_2}\cdots x_{j_n}^{v_n},
\end{equation}
if $v=v_1e_{j_1}+\cdots + v_n e_{j_n}$ lies in a maximal cone of $\Sigma$ generated by $e_{j_1},\ldots,e_{j_n}$. 

All $x^{v'}\in SH^*(Y,\llambda_v)$ are invertible, for all $v'\in N\cong \Z^n$, in particular all $x_i=x^{e_i}$ are invertible, so
\begin{equation}\label{Equation localising everything}
\k[x_1,\ldots,x_r,x^{\pm v}]/\mathcal{J} = \k[x_1^{\pm 1},\ldots,x_r^{\pm 1}]/\mathcal{J} \cong SH^*(Y,\Fi_v).
\end{equation}
\end{prop}
\begin{proof}
By \cref{Prop Fano case get divisor} (Fano setting), the equations in \cref{Theorem Quantum Stanley Reisner}
coincide with the quantum Stanley-Reisner relations.
It is a classical combinatorial result of Batyrev \cite{batyrev1993quantum} that \eqref{Equation jac ring batyrev}, after quotienting by $\mathcal{J}$, is an isomorphism onto $\mathrm{Jac}(W)$.
\cref{Equation xv the Qv class description} follows from \eqref{Equation the Qv class description}.
In \eqref{Equation localising everything}, $x^{v'}$ viewed inside $SH^*(Y,\Fi_v)$ corresponds to the rotation class $R_{\Fi_{v'}}\in SH^*(Y,\Fi_v)$ from \cref{Prop Rw rotation elements}, and they are invertible by construction, cf.\;\cite{R14,R16} or \cite[Appendix]{RZ3}. So in $SH^*(Y,\Fi_v)$ the inverse $R_{\Fi_{-v'}}$ of $R_{\Fi_{v'}}$ will represent the formal symbol $x^{-v'}$ (without needing to localise at $x^{v'}$).

The rest of the argument is now purely algebraic as follows, just like in \cite[Sec.3D]{R16} (which leaned on the original proof for compact Fano toric manifolds by McDuff-Tolman \cite{mcduff2006topological}).

Note that for classical reasons, namely \eqref{Ordinary Cohomology of Semiprojective toric}, one knows that $\mathrm{PD}[D_i]$ generate $H^*(Y,\Z)$. Moreover, by definition, $QH^*(Y)$ as a $\k$-vector space is $H^*(Y,\k)$, and quantum product is ordinary cup product together possibly with $T^{>0}$-corrections. As the maps in the claim produce the classes $\mathrm{PD}[D_i]$, it follows that the image is certainly the correct $\k$-vector space. So the only remaining question is whether the relations $\mathcal{J}$ exhaust all the necessary relations required to determine the quantum product. This reduces the question to an algebraic lemma \cite[Lemma 3.4]{R16} due originally to McDuff-Tolman \cite[Lemma 5.1]{mcduff2006topological}, which does not involve any geometry and in particular does not rely on whether we are in the compact or non-compact setting.
\end{proof}

\begin{rmk}\label{remark conifold}
One can generalise to the {\bf NEF} case (which includes the CY case $c_1(Y)=0$). NEF means that $c_1(Y)[A]\geq 0$ on nontrivial spheres $A\in \pi_2(Y)$ which have an $I$-holomorphic representative. The quantum Stanley-Reisner relations
hold provided one replaces the variables $x_i:=\mathrm{PD}[D_i]$ by $x_i:=Q_{\Fi_{e_i}}=\mathrm{PD}[D_i]+(\textrm{higher order }T)$ \cite[Sec.4K]{R16}, like in \cref{Prop Fano case get divisor}.
The fact that the relations in \cref{Theorem Quantum Stanley Reisner} are the Stanley-Reisner relations up to higher order $T$-corrections ensures that those relations (together with the linear relations) generate the required ideal for the presentation, by the algebraic trick from McDuff-Tolman \cite[Lemma 5.1]{mcduff2006topological}.

The equivariant versions of the above results, see \cref{Prop intro semiproj toric}, are proved in the same way, using the equivariant classes $EQ_{\Fi_v}$ constructed in \cite{RZ1}, and using the presentation of the ordinary equivariant cohomology mentioned in \cref{Remark equiv classical presentation}. Again, the fact that the relations in \cref{Theorem Quantum Stanley Reisner} hold for the $EQ_{\Fi_v}$ rotation classes, and agree with the classical Stanley-Reisner relations up to higher order $T$-terms, ensures that we get the correct kernel in the presentation, by the trick \cite[Lemma 5.1]{mcduff2006topological}.
\end{rmk}

\begin{ex}[Creapant resolution of the conifold]
An example with $c_1(Y)=0$ where one sees that this correction is necessary, is the resolved conifold. The conifold $\{z_1 z_2 -z_3z_4=0\}\subset \C^4$ has a fan generated by a $3$-dimensional cone generated by $e_1=(0,0,1)$, $e_2=(1,0,1)$, $e_3=(1,1,1)$, $e_4=(0,1,1)$ in $\Z^3$ and there are three obvious ways to resolve it, by subdividing the fan \cite[p.49]{fulton1993introduction}. Subdivide the cone by the plane through $e_1,e_3$: we get a semiprojective toric manifold with $c_1(Y)=0$, namely $\mathcal{O}(-1)\oplus \mathcal{O}(-1) \to \C P^1.$
Then \eqref{Ordinary Cohomology of Semiprojective toric} is
$$
H^*(Y;\Z)\cong \Z[x_1,\ldots,x_4]/(
x_1+x_3, x_3+x_4, x_1+x_2+x_3+x_4,
x_1x_3,x_2x_4)
\cong
\Z[x_1]/(x_1^2).
$$
As $e_1+e_3=e_2+e_4$, the quantum Stanley-Reisner relation(s) is $x_1x_3=x_2x_4$ replacing the ordinary SR-relations $x_1x_3=0=x_2x_4$ appearing above.
So using the quantum relation would contradict the equality $H^*(Y;\k)=QH^*(Y)$ of $\k$-vector spaces.
As mentioned previously, letting $Q_i:=Q_{\Fi_{e_i}}$, the correct relation is $Q_1\star Q_3=Q_2\star Q_4$,
which in order $T^0$ is $x_1x_3=x_2x_4$ but must, a posteriori, have a $T^{>0}$-correction term. As $x_1$ generates $H^2(Y;\Z)$, that correct 
relation is cohomologous to $x_1^2 = x_1^2 \pm T^{a}x_1^2$, some $a>0$, as $|T|=0$. So $x_1^2=0$, $QH^*(Y)=H^*(Y;\k)$ as rings, and $SH^*(Y,\llambda_v)=0$.
\end{ex}

\begin{rmk}
The total spaces of toric negative line bundles, which are Fano or CY, were investigated in \cite{R16}. For example 
$\mathcal{O}(-k)\to \C P^m$ for $1\leq k\leq m$ is Fano, and it is CY for $k=1+m$.
These are convex at infinity (cf.\;\cref{Introduction Motivation}).
We illustrate our results in a non-convex Fano example, next.
\end{rmk}

\section{Example: the non-compact Fano semiprojective toric surface $\mathrm{Bl}_{(0,\mathrm{pt})}(\C \times \P^1)$}
\label{Subsection An explicit example}

We adapted this example from \cite[Ex.7.1.12.]{cox2011toric}.
It is a non-compact Fano semiprojective toric surface, $\pi: \mathrm{Bl}_p(\C \times \P^1) \to \C$, namely $Y$ is the blow up at the point $p=\{0\}\times \{[1:0]\}\in \C \times \P^1$. 

The fan $\Sigma$ and the moment polytope $\Delta$ are illustrated below. In the lower part of the picture: the fan $\Sigma_X$ and the moment polytope $\Delta_X$ for $X=\C$, whose general properties were discussed in \cref{Remark description of X}.
\begin{center}

\end{center}

The affinisation 
$$\pi: Y \to X:=H^0(Y,\mathcal{O}_Y)\cong \C$$ 
over the torus fixed point of $X$ has fibre consisting of the two compact toric divisors $D_1,D_3$:
$$\pi^{-1}(0)=\mathrm{Core}(Y) = D_1 \cup D_3,$$ 
which are two transversely intersecting copies of $\P^1$. The other fibres are all copies of $\P^1$:
$$\pi^{-1}(x)\cong \P^1 \textrm{ for all }x\neq 0 \in \C.$$
In particular, those fibres are non-constant holomorphic curves appearing arbitrarily far out at infinity, so $Y$ cannot be convex at infinity in the sense described in  \cref{Introduction Motivation} (the symplectic form cannot be exact at infinity due to those curves).
 
The fan $\Sigma$ and its edges $e_i$ are indicated in \cref{Example toric intro}. The fan without $e_3$ gives $\C\times \P^1$; the subdivision by $e_3=e_1+e_2$ corresponds to blowing up the fixed point $p:=\{0\}\times \{[1:0]\}$. The affinisation $\pi: Y \to X$ in this case is the projection $\pi:Y=\mathrm{Bl}_p(\C\times \P^1)\to \C$.

The moment polytope is 
$$\Delta = \{y\in \R^2: \langle y,e_i \rangle \geq \lambda_i=-1 \}=\{(x,y)\in \R^2: x\geq -1, y\geq -1, x+y\geq -1, y\leq 1\}.
$$
From this, we see that $Y$ is non-compact Fano (see \cref{Remark meaning of monotone}).
One way to see that $Y$ is a semiprojective toric manifold, is to note that $\Sigma$ is a triangulation of the spanning set $\mathcal{B}=\{e_1,\ldots,e_4\}$ of the lattice $\Z^2$, and that $\Sigma$ is unimodular and regular.

The classical cohomology of $Y$ via \eqref{Ordinary Cohomology of Semiprojective toric} is presented as follows:\footnote{$x_2,x_4$ represent the locally finite cycles $[D_2],[D_4]\in H_2^{lf}(Y;\Z)\stackrel{\mathrm{PD}}{\cong} H^2(Y;\Z)$ which are intersection dual to the cycles $[D_3],[D_1]\in H_2(Y;\Z)\cong H^2(\mathrm{Core}(Y,\Fi_v);\Z)$, but $x_1=\mathrm{PD}[D_1]$, $x_3=\mathrm{PD}[D_3]$ refers to $[D_1],[D_3]$ viewed as lf-cycles.}
\begin{align*}
H^*(Y;\Z) & \cong \Z[x_1,\ldots,x_4]/(
x_1+x_3,  x_2 + x_3-x_4, x_1x_2, x_2x_4, x_3x_4)
\\
&\cong
\Z[x_1,x_2]/(x_1x_2,\ x_1^2,\ x_2^2).
\end{align*}
In the top line: the first two are the linear relations, and the last three are the classical Stanley-Reisner relations for the three primitive subsets $\wp$: $\{e_1,e_2\},$ $\{e_2,e_4\},$ and $\{e_3,e_4\}$.

By \cref{Cor complete actions semiproj toric case}, and noting that 
$$|\Sigma|=\cup (\textrm{cones}) = \R_{\geq 0} \times \R, \qquad \mathrm{Int}|\Sigma| = \R_{> 0} \times \R, \qquad \partial |\Sigma| = \{0\}\times \R,$$ 
the contracting and complete $\C^*$-actions $\Fi_v$ are described inside $N=\Z^2$ by the semigroups
$$
N_+ = \{ (a,b)\in \Z^2: a>0 \} = \N_{>0} \times \Z \qquad  \textrm{ and } \qquad
N_0 =  \{ (a,b)\in \Z^2: a\geq 0 \} = \N \times \Z.
$$
As expected by \cref{Cor core for torus action},  all $v\in N_+$ have $\mathrm{Core}(Y,\Fi_v)=D_1 \cup D_3$: the only two compact toric divisors. The fixed locus $\mathrm{Fix}(\Fi_v)$ always contains the three $\mathbb{T}$-fixed points of $Y$:
$$
F:=\mathrm{Fix}_Y(\mathbb{T}) = \{ D_1\cap D_4, D_1\cap D_3, D_3\cap D_2 \},
$$
By \cref{Lemma fixed point torus action}, we obtain:
$\mathrm{Fix}(\Fi_v)=F$ for $v\in N_+$ not a multiple of $e_1$ or of $e_3$; whereas 
$$\mathrm{Fix}(\Fi_{e_1})=F \cup D_1, \qquad \mathrm{Fix}(\Fi_{e_3})=F\cup D_3, 
\qquad \mathrm{Fix}(\Fi_{e_2})=\mathrm{Fix}(\Fi_{e_4})=D_2\cup D_4\cup (D_1\cap D_3).$$
As expected, the contracting $e_1,e_3\in N_+$ have compact fixed loci, whereas the complete non-contracting $e_2,e_4\in N_0\setminus N_+ = N\cap \partial |\Sigma|$ have a non-compact fixed locus.

Placing $T$ in grading $|T|=2$, and recalling that in the Fano case the powers of $T$ in the quantum Stanley--Reisner relations are determined by the $\Z$-grading \cite[Sec.4D]{R16},
\cref{Prop QH and SH of semiproj var} implies that for all $v\in N_+$ we obtain presentations:
\begin{align*}
QH^*(Y) & \cong \k[x_1,\ldots,x_4]/(
x_1+x_3,  x_2 + x_3-x_4, x_1x_2-Tx_3,x_2x_4-T^2,x_3x_4-Tx_1)
\\
& \cong
\k[x_1,x_2]/(x_1x_2+Tx_1,\ x_1^2,\ x_2^2+Tx_1-T^2),
\\
SH^*(Y,\Fi_v) & \cong QH^*(Y)[x^{\pm v}] 
\\
& \cong \k[x_1^{\pm 1},x_2^{\pm 1}]/(x_1x_2+Tx_1,\ x_1^2,\ x_2^2+Tx_1-T^2)
\\
& =
0,
\end{align*} 
where we simplified the presentations by solving for $x_3,x_4$ in terms of $x_1,x_2,$ and we exploited \eqref{Equation localising everything}: in $SH^*(Y,\Fi_v)$ the $x_1,x_2$ have to be invertible, so $x_1^2=0\in QH^*(Y)$ forces $c^*(x_1^2)$ to be both zero and a unit in $SH^*(Y,\Fi_v)$, so $SH^*(Y,\Fi_v)=0.$

Interestingly, quantum cohomology vanishes after localising at $x^v$ for $v\in N_+$, so $SH^*(Y,\Fi_v)=0$, despite $Y$ being Fano. As expected, symplectic cohomology is independent of $v\in N_+$, by \cref{Cor SH in toric case is QH localised}.

Recall that we needed $v\in N_+$ to ensure that $(Y,\omega,\Fi_v)$ is a symplectic $\C^*$-manifold, and thus symplectic cohomology is defined. 
We do not know if it is possible to make sense of symplectic cohomology for $v_0\in N\cap \partial |\Sigma|=N_0\setminus N_+$, and if it makes sense it is not clear whether it should still be the localisation of $QH^*(Y)$ at $x^v$, and whether it yields the same vanishing symplectic cohomology as for $v\in N_+$. We do know however that one of those statements has to break down for $v_0=(0,\pm 1)$:
$$
QH^*(Y)[x^{\pm v_0}]\neq 0,
$$
as there is a non-zero $\k$-homomorphism $QH^*(Y)[x^{\pm v_0}]\to \k$, $x_1\mapsto 0$, $x_2\mapsto T$, $x_3\mapsto 0$, $x_4\mapsto T$.

Finally, we illustrate how \eqref{Equation xv the Qv class description} arises in practice.
Consider $$v=(1,2)=e_2+e_3.$$ It lies in the maximal cone spanned by $e_2,e_3$, so  \eqref{Equation xv the Qv class description} implies that $$x^v:=x_2x_3$$
represents the rotation class $Q_{\Fi_v}$. We can also write $v=e_1+2e_2$, but in this case the $e_1,e_2$ do not arise from a maximal cone as would be required by \eqref{Equation xv the Qv class description}. If we abusively wrote $x^v:=x_1x_2^2$, then we would only be off by a power of $T$. Indeed $x_2x_3=T^{-1}x_1x_2^2$ must be generated by the quantum Stanley-Reisner relations, because relations such as $$e_2+e_3=e_1+2e_2$$ express relations for the rotations in the proof of \cref{Prop QH and SH of semiproj var}, and therefore 
$$Q_{\Fi_{e_2}}\star Q_{\Fi_{e_3}}=T^a Q_{\Fi_{e_1}} \star Q_{\Fi_{e_2}}^{2},$$
for some $a\in \R$ by \cref{Cor SH in toric case is QH localised}. In the monotone case the $T^a$ ambiguity is determined by considering the gradings: $|x_i|=2$, $|T|=2$, so $x_2x_3=T^{a}x_1x_2^2$ had to have $a=-1$.

\bibliography{FZ}

\providecommand{\bysame}{\leavevmode\hbox to3em{\hrulefill}\thinspace}
\providecommand{\MR}{\relax\ifhmode\unskip\space\fi MR }
\providecommand{\MRhref}[2]{%
  \href{http://www.ams.org/mathscinet-getitem?mr=#1}{#2}
}
\providecommand{\href}[2]{#2}
\begin{thebibliography}{\v{Z}22}

\bibitem[Bat93]{batyrev1993quantum}
Victor~V. Batyrev, \emph{Quantum cohomology rings of toric manifolds},
  {Ast\'{e}risque} \textbf{218} (1993), 9--34.

\bibitem[CLS11]{cox2011toric}
D.~A. Cox, J.~B. Little, and H.~K. Schenck, \emph{Toric varieties}, vol. 124,
  American Mathematical Soc., 2011.

\bibitem[Ful93]{fulton1993introduction}
W.~Fulton, \emph{Introduction to toric varieties}, no. 131, Princeton
  University Press, 1993.

\bibitem[Gui94]{guillemin2012moment}
V.~Guillemin, \emph{{Moment maps and combinatorial invariants of Hamiltonian
  $T^n$-spaces}}, vol. 122, Springer Science \& Business Media, 1994.

\bibitem[HS95]{HS95}
H.~Hofer and D.~Salamon, \emph{{Floer homology and Novikov rings}}, The {F}loer
  memorial volume, Progr. Math., vol. 133, Birkh\"{a}user, Basel, 1995,
  pp.~483--524.

\bibitem[HS02]{HS02}
T.~Hausel and B.~Sturmfels, \emph{Toric hyperk{\"a}hler varieties}, Doc. Math
  \textbf{7} (2002), 495--534.

\bibitem[LJ20]{liebenschutz2020intertwining}
T.~Liebenschutz-Jones, \emph{{An intertwining relation for equivariant Seidel
  maps}}, arXiv:2010.03342 (2020), 1--60.

\bibitem[MT06]{mcduff2006topological}
D.~McDuff and S.~Tolman, \emph{{Topological properties of Hamiltonian circle
  actions}}, Int. Math. Res. Not. \textbf{2006} (2006), 72826.

\bibitem[Rit13]{R13}
A.~F. Ritter, \emph{Topological quantum field theory structure on symplectic
  cohomology}, J. Topol. \textbf{6} (2013), no.~2, 391--489.

\bibitem[Rit14]{R14}
\bysame, \emph{{Floer theory for negative line bundles via Gromov-Witten
  invariants}}, Adv. Math. \textbf{262} (2014), 1035--1106.

\bibitem[Rit16]{R16}
\bysame, \emph{{Circle actions, quantum cohomology, and the {F}ukaya category
  of {F}ano toric varieties}}, Geom. Topol. \textbf{20} (2016), no.~4,
  1941--2052.

\bibitem[R{\v{Z}}23a]{RZ1}
A.~F. Ritter and F.~{\v{Z}}ivanović, \emph{{Filtrations on quantum cohomology
  from the Floer theory of $\C^*$-actions}}, arXiv:2304.13026 (2023), 1--70.

\bibitem[R{\v{Z}}23b]{RZ2}
\bysame, \emph{{Filtrations on quantum cohomology via Morse--Bott--Floer
  Spectral Sequences}}, arXiv:2304.14384 (2023), 1--90.

\bibitem[R{\v{Z}}24]{RZ3}
\bysame, \emph{{Filtrations on equivariant quantum cohomology and
  Hilbert-Poincar\'{e} series}}, arXiv (2024), 1--71.

\bibitem[Sei97]{Sei97}
P.~Seidel, \emph{{$\pi_1$} of symplectic automorphism groups and invertibles in
  quantum homology rings}, Geom. Funct. Anal. \textbf{7} (1997), no.~6,
  1046--1095.

\bibitem[Smi20]{smith2020quantum}
J.~Smith, \emph{Quantum cohomology and closed-string mirror symmetry for toric
  varieties}, Q. J. Math. \textbf{71} (2020), no.~2, 395--438.

\bibitem[\v{Z}22]{vzivanovic2022exact}
F.~\v{Z}ivanović, \emph{{Exact Lagrangians from contracting
  $\mathbb{C}^*$-actions}}, arXiv:2206.06361 (2022), 1--46.

\end{thebibliography}
\bibliographystyle{amsalpha}
\end{document}